\documentclass[final,leqno,onefignum,onetabnum]{siamltexmm2}

\usepackage{amsmath,amssymb,latexsym}
\usepackage{hyperref}
\usepackage{url}
\usepackage{graphicx}
\usepackage{subfigure}
\usepackage{array}
\usepackage{color}
\usepackage[thinlines]{easytable}
\usepackage{relsize}
\usepackage{wrapfig}
\usepackage[noend]{algpseudocode}
\usepackage{algorithm}
\usepackage{subfigure}
\usepackage{multicol}
\usepackage{float}

\def\Reals{\mathbb{R}} 
\def\E{\mathbb{E}} 
\def\cN{{\mathcal N}} 

\def\bd#1{\boldsymbol{#1}}
\def\argmin{\operatornamewithlimits{arg\,min}}

\graphicspath{{Figures/}}

\title{Alternating Minimization Algorithm with Automatic Relevance Determination for  Transmission Tomography under Poisson Noise}

\author{Yan~Kaganovsky{\thanks{Department of Electrical and Computer Engineering, Duke University, Durham, NC 27708, USA.}},
        Shaobo~Han{\footnotemark[1]},
        Soysal~Degirmenci{\thanks{Department of Electrical and Systems Engineering, Washington University in St. Louis, MO 63130, USA.}},
        David~G.~Politte{\thanks{Mallinckrodt Institute of Radiology, Washington University in St. Louis, MO 63110, USA.}}, David J. Brady\footnotemark[1],
        Joseph~A.~O'Sullivan\footnotemark[2],
        and~Lawrence~Carin\footnotemark[1]. 
\footnote{Emails: yankagan@gmail.com, shaobohan@gmail.com, s.degirmenci@wustl.edu, dgpolitte@gmail.com, dbrady@duke.edu, jodyosullivan@gmail.com,  lcarin@duke.edu.}}

\begin{document}

\maketitle
\newcommand{\slugmaster}{%
\slugger{}{}{}{}{}}

\begin{abstract}
We propose a globally convergent alternating minimization (AM) algorithm for image reconstruction in transmission tomography, which extends automatic relevance determination (ARD) to Poisson noise models with Beer's law.  The algorithm promotes solutions that are sparse in the pixel/voxel-differences domain by introducing additional latent variables, one for each pixel/voxel, and then learning these variables from the data using a hierarchical Bayesian model.
Importantly, the proposed AM algorithm is \emph{free of any tuning parameters} with image quality comparable to standard penalized likelihood methods.
Our algorithm exploits optimization transfer principles which reduce the problem into parallel 1D optimization tasks (one for each pixel/voxel), making the algorithm feasible for large-scale problems;   This approach considerably reduces the computational bottleneck of ARD associated with the posterior variances.  Positivity constraints inherent in  transmission tomography problems are also enforced. We demonstrate the performance of the proposed algorithm for x-ray computed tomography using synthetic and real-world datasets. The algorithm is shown to have much better performance than prior ARD algorithms based on approximate Gaussian noise models, even for high photon flux.
Sample code is available at \url{http://www.yan-kaganovsky.com/#!code/c24bp}.
\end{abstract}  



\pagestyle{myheadings}
\thispagestyle{plain}
\markboth{AM Algorithm with ARD for Transmission Tomography (Preprint Version)}{Y. Kaganovsky et al. (Preprint Version)}

\section{Introduction} \label{sec_intro}

Tomographic image reconstruction is the process of estimating an object from measurements of its line integrals along different angles \cite{slaney1988principles}. 
Tomographic reconstruction is an ill-posed problem \cite{natterer1986computerized} in the sense that there are multiple solutions that are consistent with the data. 
It is therefore desirable to make use of prior knowledge, which includes the statistical characteristics of the measured data, and properties that are expected of the image.   

The images in transmission tomography represent attenuation per unit length of the impinging beam due to energy absorption inside the imaged object. 
Well known examples are x-ray computed tomography (CT) \cite{kalender2006}, and electron tomography \cite{frank1992electron}. The attenuation must be nonnegative, since a negative value corresponds to an increase in the intensity of the exiting beam, which is non-physical.  These images can often be well approximated by a sparse representation in some transform domain. In some applications, sparsity is present directly in the native image domain \cite{li2002}, but more commonly it is present in the pixel-difference domain \cite{sonka2000handbook} or in some transform domain, such as the wavelet, curvelet or shearlet transforms \cite{simoncelli1999modeling,boubchir2005}.  


There are generally two types of approaches for tomographic image reconstruction. The first type consists of one-shot algorithms, such as filtered back-projection \cite{slaney1988principles} and its extensions, which rely on analytic formulas. These algorithms do not incorporate the type of prior knowledge mentioned above and also produce prominent artifacts when some of the data are missing. The second type consists of iterative algorithms which are based on minimizing some objective function. The latter enable the incorporation of prior knowledge about the data and the image of interest, e.g., the objective function can put less emphasis on measurements that are expected to be more noisy, include penalties that promote sparsity of the image in some representation, and positivity constraints for the image can also be included. In addition, iterative algorithms are far more robust to missing data.

Many transmission tomography problems involve physical processes consisting of quanta, e.g., the number of transmitted photons in x-ray CT, or the number of transmitted electrons in electron tomography. Therefore, there is a strong physical motivation for modeling these processes using independent Poisson random variables \cite{frank1992electron,sonka2000handbook}.  In the medical CT community, the Poisson noise model is commonly used \cite{sonka2000handbook} and provides the basis for a well known class of iterative algorithms  \cite{sonka2000handbook,lange,lange1995globally,Erdogan,OSullivan}. In addition, the mean of each measurement is modeled according to Beer's law \cite{sonka2000handbook}, which is an inherent non-linearity in transmission data.

In this paper, we develop a new class of statistical iterative algorithms for image reconstruction in transmission tomography that is inspired by automatic relevance determination (ARD) \cite{neal1995bayesian, tipping2001sparse, wipf2011latent,Danniel}. It incorporates prior knowledge that includes Poisson statistics of the underlying signals, Beer's law, positivity of the image, and also sparsity of the image in some underlying representation. What sets this new class apart from prior art is the fact that it \emph{automatically learns the balance between data-fidelity and prior knowledge}, thus avoiding the use of any tuning parameters and the difficulty in determining them. In addition, it also computes posterior variances, which can be used for adaptive sensing/experimental design, or for determining the system's sensitivity to noise.  To the best of our knowledge, the proposed algorithm is the first ARD-based algorithm for Poisson noise that is both scalable to very large problems and that is also guaranteed to converge. 

Our work is motivated by x-ray CT which is used in medical and security applications \cite{slaney1988principles}. Nevertheless, the methods presented here can be applied more generally to any type of transmission
tomography that involves count data. In formulating the problem presented in Sec.~\ref{model} we have tried to retain some generality in the hope that other imaging modalities could potentially benefit from the proposed approach. Naturally, this entails some simplifications of the forward model which are discussed later in Sec.~\ref{limits}.

It should be noted that most modern commercial CT scanners do not employ photon counting detectors. However, there are already photon counting detectors used in clinical studies \cite{Photon_Count}, and it is certainty possible that they will be integrated into commercial CT scanners in the near future.
Modern CT scanners also utilize a high photon flux, which led various researchers to propose algorithms that are based on an approximate post-log Gaussian noise model \cite{bouman1996,ramani2012splitting,sukovic2000penalized} rather than a Poisson noise model. 
The Gaussian model can be derived by making a 2nd order Taylor series approximation to the log-likelihood corresponding to the Poisson noise model with Beer's law, under the assumption of moderate to high photon flux. This results in a data-weighted least-squares datafit term \cite{bouman1996}, which also corresponds to a post-log Gaussian noise model. 
However, in low-count scenarios, which are critical to maintain low radiation dose in pediatric and small animal CT, there could be a substantial performance benefit in using the Poisson model once photon-counting detectors become widely available. For example, it has been shown empirically that the approximate Gaussian model leads to a systematic negative bias that increases as photon flux decreases \cite{fessler1995hybrid}. 

\subsection{Model} \label{model}
We consider a Poisson noise model for transmission tomography \cite{sonka2000handbook} given by 
\begin{align}
p(\bd{y}|\bd{x})=\prod_{i=1}^{n}p(y_{i}|\bd{x})=\prod_{i=1}^{n}\text{Poisson}[\eta_i\exp(-\bd{\phi}^T_{i}\bd{x})], \label{poisson}
\end{align}
where here and henceforth $p(\bd{y}|\bd{x})$ denotes the likelihood function, $\bd{y}\in \mathbb{Z}^n$ contains the measurements, $\text{Poisson}[\lambda]$ denotes the Poisson univariate probability distribution with rate $\lambda$, and the exponential function encompasses Beer's law (a physical property of the signals \cite{sonka2000handbook}). The column vector $\bd{x} \in \Reals_+^{p\times 1}$ denotes the lexicographical ordering of the 2D/3D image to be reconstructed, where the entry $x_j$ is the linear attenuation coefficient at the $j$th pixel/voxel in units of inverse length. 
The operation $\bd{\phi}_i^T\bd{x}$ represents a line integral over the image corresponding to the $i$th source-detector pair (see Fig.~\ref{fig_geo}) and to the measurement $y_i$.  
We denote by $\bd{\Phi} \in \Reals_+^{n\times p}$ the system matrix that is formed by concatenating the row vectors $\bd{\phi}^T_i\in\Reals_+^{1\times p}$ ($i=1,2,...,n$). 
$\eta_i$ is the mean number of photons that would have been measured using the $i$th source-detector pair if the imaged object was removed. We make the standard assumption that $\eta_i$ was measured for all $i$ in an independent calibration experiment and is known to us. 

We construct $\bd{\Phi}$ using the approach in \cite{zhuang1994numerical}, where an entry $\phi_{ij}$ at the $i$th row and $j$th column of $\bd{\Phi}$ is equal to the length at which the line between the $i$th source-detector pair intersects the $j$th pixel/voxel (see Fig.~\ref{fig_geo}).
Each ray transverses at most $O(p^{1/D})$ pixels/voxels, where $p$ is the total number of pixels/voxels and $D=2,3$ for 2D/3D images, respectively. Accordingly, each row $\bd{\phi}^T_i$ in the system matrix has only $O(p^{1/D})$ non-zero entries. Henceforth, we absorb a reference linear attenuation coefficient into $\bd{\Phi}$ so that $\bd{x}$ is dimensionless.

\begin{figure}
\centering
\begin{minipage}{.45\textwidth}
\centering
\includegraphics[width=4cm]{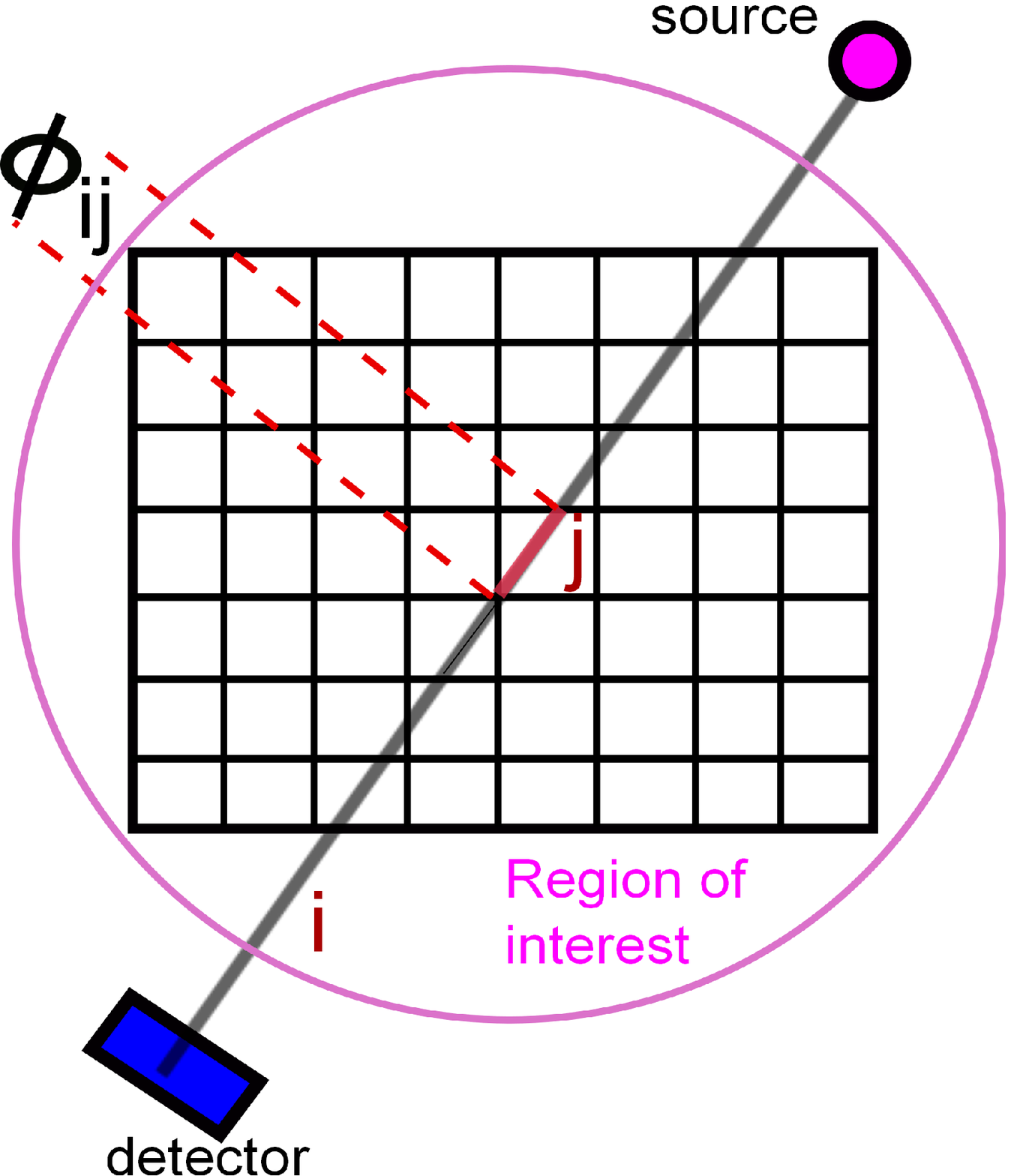}
\caption[]{Tomographic measurements in 2D. The total linear attenuation for the $i$th measurement is equal to the line integral over the imaged domain along the path defined by the $i$th source-detector pair. Several measurements are collected simultaneously using several detectors. The source and detectors are rotated to collect different views. Each line integral is modeled as $\bd{\phi}_i^T\bd{x}=\sum_j \phi_{ij}x_j$ where $x_j$ is the linear attenuation coefficient at the $j$th pixel in units of inverse length and $\phi_{ij}$ is the intersection length of the $i$th line with the $j$th pixel. If $p$ is the total number of pixels, then each line intersects at most $O(p^{1/2})$ pixels. }
\label{fig_geo}
\end{minipage}%
\hspace{0.2cm}
\begin{minipage}{.45\textwidth}
\centering
\includegraphics[width=3.5cm]{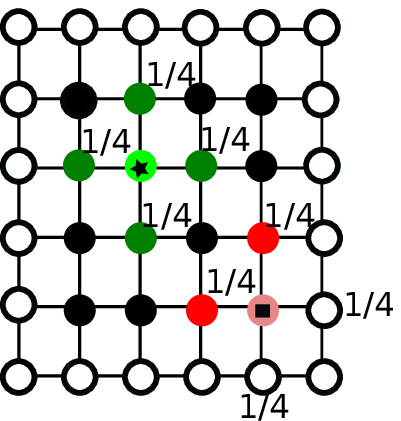}
\caption[]{An example for an adjacency graph defining a neighborhood $\Theta$ in \eqref{psi_diff}. Each node represents a pixel in the image and $\{j,k\}\in\Theta$ iff the $j$th and $k$th nodes are connected by an edge. In this example, the node colored in light green and marked by a star has 4 neighboring pixels shown in dark green. Each neighbor has a weight of $1/4$ when computing their average. The image is augmented to introduce a boundary of pixels with zero values, represented by white nodes. The latter are not included in the forward model and are merely introduced to provide the interpretation for the $1/4$ weight when averaging less than $4$ neighbors, e.g., consider the node colored in light red and marked by a square.} 
\label{fig_graph}
\end{minipage}
\end{figure}

\subsection{Background}
The method proposed in this paper is an extension of automatic relevance determination (ARD). For the convenience of readers unfamiliar with ARD we first provide background to explain the potential advantages of ARD motivating this work and discuss the key differences between ARD and other more common sparse estimation methods. Our aim here is to emphasize previous challenges in adapting ARD to the Poisson model in \eqref{poisson}, challenges which we aim to resolve with the proposed method.
\subsubsection{Maximum a Posteriori (MAP) Estimation} \label{typeI}
Statistical iterative methods for CT \cite{sonka2000handbook,lange,lange1995globally,OSullivan} cast the problem of image reconstruction as a penalized maximum likelihood estimation problem, given by
\begin{align}
\bd{x}^*=\arg\max_{\bd{x}} [\log p(\bd{y}|\bd{x})+ \log p(\bd{x}|\beta)], \label{type_I}
\end{align}
where $\log p(\bd{y}|\bd{x})$ is the log-likelihood function which represents the datafit. 
For the Poisson model in \eqref{poisson}, maximizing the log-likelihood is equivalent to minimizing the I-divergence \footnote{I-divergence between two vectors $\bd{p},\bd{q}\in\Reals_+^M$ is defined as $I(\bd{p}||\bd{q})=\sum_{j=1}^M p_j\log(\frac{p_j}{q_j})-p_j+q_j$ \cite{cizar}} between the actual measurements $\bd{y}$ and the predicted measurements $\bd{\eta}\odot\exp(-\bd{\Phi}\bd{x})$ \cite{OSullivan}, where $\odot$ denotes an elementwise product. 
$p(\bd{x}|\beta)$ in \eqref{type_I} is the prior probability distribution representing prior knowledge about $\bd{x}$ and $\log p(\bd{x}|\beta)$ is a penalty that promotes solutions with desired properties according to this prior knowledge, e.g., the $\ell_1$ penalty \cite{li2002}, or the total variation (TV) penalty \cite{sidky2,sidky}, which promote sparse solutions. $\beta$ is a tuning parameter that determines the tradeoff between the datafit and the prior, or the bias-variance tradeoff. The ``best'' value of $\beta$ is usually unknown and depends on the imaged object at hand.  To determine $\beta$, it is often required to solve \eqref{type_I} many times for different values of $\beta$, examine all reconstructed images and then select the value that best fits the needs of the application, based on some chosen figure of merit.  For many common choices of $p(\bd{x}|\beta)$  considered below, decreasing $\beta$ will put more weight on the datafit term and will result in increased variance and in noisier images. Increasing $\beta$ will put more weight on the penalty term, and will result in increased bias and increased artifacts in the image, e.g., staircase artifacts.  For a given $\beta$, the solution in \eqref{type_I} can be interpreted as the maximum a posteriori (MAP) solution, i.e., the maximizing point of the posterior distribution, which is given by Bayes' theorem
\begin{align}
p(\bd{x}|\bd{y})=p(\bd{y}|\bd{x})p(\bd{x})/p(\bd{y}). \label{Bayes}
\end{align}
The $\ell_1$ and TV penalties are non-differentiable and do not allow the use of standard gradient-based methods for solving \eqref{type_I}. An alternative to TV is the class of neighborhood roughness penalties (Gibbs priors) \cite{lange1995globally} which have the following form 
\begin{align}
 -\log p(\bd{x}|\beta)=\beta\sum_{\{j,k\} \in \Theta} w_{jk} \pi (x_j-x_k), \label{Gibbs}
\end{align}\\[-1em]
where $w_{jk}$ is a positive weight inversely proportional to the distance between the centers of the $j$th and $k$th pixels/voxels, $\Theta$ is a set of ordered pairs $\{j,k\}$ defining a neighborhood system and $\pi(x)$ is chosen to be a positive, even, and twice continuously differentiable function.  To preserve sharp edges in the image, $\pi(x)$ is often chosen such that $\pi(x) \propto |x|$ for large $|x|$ and can be viewed as a smooth version of the $\ell_1$ penalty. A common choice is \cite{holland1977robust,lange1995globally}
\begin{align}
\pi(x)=\delta^2[|x/\delta|-\log(1+|x/\delta|)] \label{Huber},
\end{align}
which approaches the quadratic $x^2$ for $\delta\rightarrow\infty$ and the linear function $\delta|x|$ as $\delta\rightarrow 0$. This type of penalties allows the use of standard gradient-based methods but comes with the price of having an additional tuning parameter $\delta$, so the search for the tuning parameters becomes even more cumbersome. Decreasing $\delta$ too much will result in noisy images but with sharp edges, and increasing $\delta$ too much will result in blurred images. 

\subsubsection{Automatic Relevance Determination (ARD)} \label{ARD}
A different approach to sparse estimation is ``automatic relevance determination'' (ARD) \cite{neal1995bayesian}. Originally, ARD has been proposed for \emph{Gaussian} likelihoods with the prior \cite{tipping2001sparse}   
\begin{align}
 p(\bd{x}|\bd{\gamma})=\cN(\bd{x}|\bd{0},\bd{\Gamma}), \qquad \bd{\Gamma}=\text{Diag}(\bd{\gamma}),  \label{prior_ARD}
\end{align}
where $\text{Diag}(\bd{\gamma})$ denotes a diagonal matrix with main diagonal $\bd{\gamma}$. In this approach sparsity is promoted by treating $\bd{x}$ as latent variables and seeking the newly introduced \emph{hyperparameters} $\bd{\gamma}\in \mathbb{R}_+ ^p$ that maximize the marginal likelihood $p(\bd{y}|\bd{\gamma})$ (also called the ``model evidence'' \cite{wipf2011latent})
\begin{align}
\bd{\gamma}^*=\max_{\bd{\gamma}}\log p(\bd{y}|\bd{\gamma}). \label{evidence_max}
\end{align}
Once $\bd{\gamma}^*$ is found, the estimate for $\bd{x}$ is given by the mean of the posterior distribution $p(\bd{x}|\bd{y},\bd{\gamma}^*)$. Interestingly, the solutions to \eqref{evidence_max} have many $\gamma_j\approx 0$ and the posterior distribution for the corresponding parameters $x_j$ becomes highly concentrated around zero\footnote{By employing Bayes' rule in \eqref{Bayes}, it can be easily shown that if the prior is concentrated around $x_i=0$ then the corresponding posterior distribution is also concentrated around $x_i=0$}. The mechanism that promotes sparsity of $\bd{\gamma}^*$ has been revealed by Wipf et al.  \cite{wipf2004perspectives,Wipf2008}, with the analysis specifically tailored for Gaussian likelihoods. 

The solution of \eqref{evidence_max} was originally addressed by Tipping in \cite{tipping2001sparse} using an expectation-maximization (EM) algorithm \cite{dempster1977maximum}.  In the E-step, the Gaussian posterior probability distribution $p(\bd{x}|\bd{y},\bd{\gamma})$ is computed for given hyperparameters $\bd{\gamma}$. In the M-step, $\bd{\gamma}$ are updated based on the mean and variances of the posterior distribution found in the E-step. 
As opposed to MAP estimation methods, which focus on the mode (maximum) of the posterior distribution, ARD methods also compute the covariance matrix (during the E step), or at least the marginal variances, which provide additional information. ARD-based algorithms are therefore computationally more demanding than MAP methods, mainly due to the expensive task of estimating the variances. Previous extensions of ARD  \cite{Wipf2008,seeger2011large} require the same type of operations and rely on approximations to the  variances \cite{seeger2011large} to maintain scalability. 

Despite computational complications, the ARD approach can be quite rewarding and excel in hard or ill-posed problems, as evidenced in several recent applications \cite{seeger_mag}.  A very significant advantage of ARD methods is their ability to \emph{automatically learn the balance between the datafit and the degree of sparsity}. This avoids the tuning of nuisance parameters, which is often nonintuitive, object-dependent, and requires many trials.  The posterior variances are key in determining this tradeoff. The variances can also be used for adaptive sensing/experimental design \cite{Danniel,seeger2011large}, which is outside the scope of this paper and will be the subject of future work.

\subsubsection{Difficulties in Previous ARD-based Methods for Poisson Models} \label{prev_ARD}
The ARD-EM algorithm \cite{tipping2001sparse} described in Sec.~\ref{ARD}  cannot be directly applied to the Poisson likelihood in \eqref{poisson} since the E-step is intractable due to the non-conjugacy\footnote{A prior is said to be conjugate to the likelihood if both the prior and posterior probability distributions are in the same family. The posterior is then given in closed-form.} of the prior in \eqref{prior_ARD} and the likelihood. 
Although we are not aware of any previously published work on ARD for Poisson models with Beer's law, a simple remedy to the non-conjugacy is the Laplace method, discussed in \cite{tipping2001sparse}. In that approach, the posterior is approximated during the E-step using a multivariate Gaussian distribution, with the mean equal to the MAP solution and the covariance matrix equal to the inverse of the Hessian matrix for the log-likelihood, evaluated at the MAP solution.  However, the resulting ``EM-like'' algorithm  is not based upon a consistent objective function, \emph{and is not guaranteed to converge}. In addition, it requires inversions of the Hessian matrix \emph{with computational complexity $O(p^3)$}, where $p$ is the total number of pixels/voxels. This is infeasible for x-ray CT, where $p$ could be very large, e.g., $p=10^8$ or higher. Instead, one can use more sophisticated methods \cite{seeger2011large,bekas2007estimator, Papandreou2011} to indirectly estimate the diagonal of the inverse, i.e., the marginal variances, since only these are required for updating the hyperparameters $\bd{\gamma}$ \cite{tipping2001sparse}. 
Note however that the work in \cite{seeger2011large,bekas2007estimator, Papandreou2011} does not consider  Laplace approximations, so these methods have not been tested for this case.   
Irrespective of how the variances are computed, the ``EM-like'' algorithm with the Laplace approximation is unsatisfactory from a principled perspective due to the lack of an objective function and due to the lack of convergence guarantees.  The objective is also desired in order to assess convergence in practice and in order to verify the correctness of the implementation.  \emph{In addition, we are faced with conceptual difficulties, since the analysis that revealed the sparsity promoting mechanism of ARD by Wipf et al. \cite{wipf2011latent,wipf2004perspectives,Wipf2008} is restricted to Gaussian likelihoods.}



\subsubsection{Previous ARD Method for Transmission Tomography Based on an Approximate Post-Log Gaussian Noise Model}
As mentioned in the introduction, it is reasonable to replace the model in \eqref{poisson} with an approximate post-log Gaussian noise model under moderate to high photon flux conditions.
An ARD algorithm for transmission tomography based on the Gaussian model and the EM algorithm in  \cite{tipping2001sparse} was proposed by Jeon et al. \cite{Danniel}.
It should be noted that the original ARD algorithm in \cite{tipping2001sparse} has a computational cost of $O(p^3)$ due to the computation of posterior variances, and therefore, it is not practical for large scale problems. In contrast, the algorithm proposed in \cite{Danniel} scales as $O(n p^{1/D})$ where $D=2,3$ for 2D and 3D images, respectively.
There are several disadvantages of the algorithm in \cite{Danniel}: (1) it does not enforce positivity of the solution; (2) due to the approximate estimation of posterior variances, the overall algorithm does not have convergence guarantees; (3) it requires one to choose the type and number of probing vectors  for estimating the variances \cite{Danniel}, which depends on the specific geometry of the system and on the imaged object; (4) in order to estimate the posterior variances it requires solving \emph{many} $p\times p$ matrix  equations, which is computationally expensive; (5) computing the objective function scales as $O(p^3)$, making this operation infeasible for large scale problems; (6) it is limited to proper priors and complete sparse representations.
\subsection{Contributions} \label{cont}
The contributions of this paper are summarized as follows. 
\begin{enumerate}
\item We present a new alternating minimization (AM) algorithm for transmission tomography, which extends automatic relevance determination (ARD) to Poisson noise models with Beer's law. The proposed algorithm is free of any tuning parameters and unlike prior ARD methods, also has a consistent objective function that can be evaluated in a computationally efficient way, in order to assess convergence in practice and to verify the correctness of the implementation.
\item We reveal the sparsity-promoting mechanism in the proposed ARD-based algorithm for Poisson likelihoods. This preserves insights similar to those gained from previous work \cite{wipf2011latent,wipf2004perspectives,Wipf2008} for Gaussian likelihoods.    
\item We derive a \emph{scalable} AM algorithm for transmission tomography, that is well adapted to parallel computing. We reduce the computational complexity considerably by simplifying the algorithm into parallel 1D optimization tasks (for each pixel/voxel) with the scalar mean and variance updates also done in parallel. 
\item We present a convergence analysis for the scalable AM algorithm and establish guarantees for global convergence\footnote{An algorithm is said to be ``globally convergent'' if for \emph{arbitrary starting points} the algorithm is guaranteed to converge to a point satisfying the necessary optimality conditions (local optima included) \cite{Luenberger}.}. To the best of our knowledge, this is the first \emph{ARD-based} algorithm for \emph{Poisson} models that has convergence guarantees. 
\item We study the performance of the algorithm on synthetic and real-world x-ray datasets. 
\item We compare the proposed algorithm to the Gaussian ARD method in \cite{Danniel} and demonstrate that the proposed algorithm leads to better performance under Poisson distributed data, even for high photon flux.
\end{enumerate}

\subsection{Outline of the Paper}
The rest of this paper is organized as follows. In Sec.~\ref{AM_ARD} we present a general overview of the proposed AM algorithm and explain how sparsity is promoted. In Sec.~\ref{fast_VARD} we present a scalable AM algorithm based on separable surrogates for the objective function, which is well adapted to parallel computing and feasible for large-scale problems. In Sec.~\ref{convergence} we present a detailed analysis of the convergence properties of the algorithm presented in Sec.~\ref{fast_VARD}.  In Sec.~\ref{connections} we make connections between the proposed algorithm and related previous methods. We also present an alternative view of the proposed ARD method that further explains how sparsity is promoted. In Sec.~\ref{gen_AM} we extend the proposed ARD framework to allow the modeling of sparsity in \emph{overcomplete} representations and the use of improper priors, while retaining the convergence properties studied in Sec.~\ref{convergence}. 
In Secs.~\ref{limits}--\ref{alter} we discuss model limitations and possible alternative optimization methods related to our work. 
In Sec.~\ref{results} we study the performance of the proposed scalable ARD algorithm on synthetic and real-world datasets (with a parallel implementation used in the experiments) and compare it to prior MAP estimation methods as well as to the ARD method in \cite{Danniel} which is based on a post-log Gaussian noise model. Conclusions are presented in Sec.~\ref{conc}.

\section{Variational ARD}\label{AM_ARD}
\subsection{Prior} \label{VARD_prior}
In many applications of transmission tomography sparsity is not present directly in the native basis but in some transform domain. Therefore,  we consider a generalized prior (similar to Seeger and Nickisch \cite{seeger2011large}) to address these cases 
\begin{align}
 p(\bd{x}|\bd{\gamma})=\cN(\bd{x}|\bd{0},(\bd{\Psi}^T\bd{\Gamma}^{-1}\bd{\Psi})^{-1}), \qquad \bd{\Gamma}=\text{Diag}(\bd{\gamma}), \label{prior}
\end{align}
where $\bd{\gamma}\in \mathbb{R}_+ ^p$ are newly introduced hyperparameters, and we assume a transformation 
\begin{align}
\bd{s}=\bd{\Psi}\bd{x}, \label{sparse_rep}
\end{align} 
where $\bd{\Psi}$ is known and $\bd{s}$ is approximately sparse. By choosing $\bd{\Psi}=\bd{I}$, the prior in \eqref{prior} reduces to \eqref{prior_ARD}. 
To provide some insight into the mechanism of the proposed algorithm we assume for now that $\bd{\Psi}$ is square and invertible. We shall present a more general approach in Sec.~\ref{gen_AM} to allow for cases where these assumptions are not met in practice.

Any scalable algorithm will rely on efficient matrix-vector multiplication $\bd{\Psi}\bd{x}$. This is achieved if $\bd{\Psi}$ is sparse or has some special structure, such as discrete Fourier or wavelet transforms.  In this work we focus on neighborhood penalties, which are commonly used for x-ray CT, the application explored in Sec.~\ref{results}. We first consider $\bd{\Psi}$ given by  
\begin{align} \label{psi_diff}
(\bd{\Psi})_{jk}=\psi_{jk}=\bigg\{
\begin{array}{lll}
&1    &\qquad \text{if} \quad j=k \\
&-1/N_j &\qquad \text{if} \quad \{j,k\} \in \Theta \\ 
&0    &\qquad \text{otherwise}
\end{array},
\end{align}
where $\Theta$ is the set of ordered pairs $\{j,k\}$ defining a neighborhood system and $N_j$ is the number of neighbors for the $j$th pixel/voxel. An example is shown in Fig.~\ref{fig_graph}.  
The neighborhood is usually chosen to be small enough such that $\bd{\Psi}$ will be sparse with $N_j+1$ non-zero elements in the $j$th row. For the choice in \eqref{psi_diff}, each $s_j$ in \eqref{sparse_rep} is equal to the difference between the $j$th pixel/voxel and the average of its neighboring pixels/voxels, i.e., $s_j=x_j-\sum_{k:  \{j,k\}\in\Theta}[x_k/N_j]$. Note that sparsity of these pixel/voxel differences implies piecewise smooth images.  
Lastly, we note that the choice in \eqref{psi_diff} implies Dirichlet boundary conditions, which is equivalent to assuming additional neighbors outside the image domain with values set to zero, such that the number of neighbors for each pixel/voxel is equal, i.e., $N_j=N$ (see Fig.~\ref{fig_graph}). These boundary conditions agree with physical constraints typical of many real experimental setups, where the attenuation is zero outside the region of interest.  We shall remove these restrictions in Sec.~\ref{gen_AM}.

\subsection{Variational View of the EM Algorithm for ARD} \label{ARD_EM}
We start with a known variational view \cite{neal1998view} of the EM algorithm which is used to solve \eqref{evidence_max} for Gaussian noise models. Here we introduce some of the notation used throughout this work, and discuss the type of difficulties encountered when trying to extend ARD to non-Gaussian noise models.  

We rewrite the model evidence as
\begin{align}
\log p(\mathbf{y}|\bd{\gamma}) =-\mathcal{F}[q(\mathbf{x}),\bd{\gamma}]+D_{KL}\big[q(\mathbf{x})||p(\bd{x}|\bd{y},\bd{\gamma}) \big], \label{evidence}
\end{align}
where
\begin{align}
&\mathcal{F}[q(\mathbf{x}),\bd{\gamma}]=\mathbb{E}_{q(\mathbf{x})} \log [q(\mathbf{x})/p(\bd{y},\bd{x}|\bd{\gamma}) ], \label{FVE} \\
&D_{KL}\big[q||p\big]= \E_{q(\bd{x})} \log [q(\bd{x})/p(\bd{x}|\bd{y},\bd{\gamma})]. \label{KL}
\end{align} 
$q(\bd{x})$ is the proposed variational distribution and $p(\bd{x}|\bd{y},\bd{\gamma})$ is the exact posterior distribution, $\mathcal{F}$ is the \emph{free variational energy} (\textbf{FVE}) or negative evidence lower bound (\textbf{ELBO}), $D_{KL}[q||p]$ is the Kullback-Leibler (KL) divergence, and $\E_q$ is the expected value with respect to $q(\bd{x})$. The EM algorithm can be viewed \cite{neal1998view} as minimizing the free variational energy (FVE) function $\mathcal{F}$ in \eqref{FVE} (or maximizing the ELBO) by alternating between updating $q(\bd{x})$ in the E-step and updating $\bd{\gamma}$ in the M-step, as summarized below
\begin{align}
&\mathrm{\textbf{E}:} \quad q^{(t+1)} =\arg\min_{q} \mathcal{F}[q(\mathbf{x}),\bd{\gamma}^{(t)}], \label{E_var} \\
&\mathrm{\textbf{M}:} \quad \bd{\gamma}^{(t+1)} =\arg\min_{\bd{\gamma}} \mathcal{F}[q^{(t+1)}(\mathbf{x}),\bd{\gamma}]. \label{M_var}
\end{align}
In the E-step, the optimal solution with zero KL divergence is $q^{(t+1)}=p(\bd{x}|\bd{y},\bd{\gamma}^{(t)})$, i.e., the exact posterior based on the previous $\bd{\gamma}^{(t)}$ values, which can be computed directly from \eqref{Bayes}. In the M-step, minimizing $\mathcal{F}$ with respect to $\bd{\gamma}$ can only increase the KL term, since it is nonnegative and overall the model evidence is also increased.

For \emph{non-Gaussian} likelihoods and the prior in \eqref{prior}, the E step in \eqref{E_var} is intractable, i.e., the expression for $p(\bd{x}|\bd{y},\bd{\gamma}^{(t)})$ is unknown and its direct calculation according to \eqref{Bayes} involves a high-dimensional integral that is prohibitive to compute. Instead, we are forced to \emph{restrict} $q(\bd{x})$ to a tractable family of distributions to approximate the posterior. However,  this results in an algorithm that is no longer guaranteed to increase the evidence \cite{gunawardana}. This occurs because the KL term in \eqref{evidence} is never zero and can either decrease or increase after updating $\bd{\gamma}$. Furthermore, a local maximum of the evidence with respect to $\bd{\gamma}$ is generally \emph{not} a local minimum of $\mathcal{F}$ (or a local maximum of the ELBO), except for degenerate cases \cite{gunawardana}. 

Therefore, when considering the above approach to extend ARD, we immediately encounter a \emph{conceptual} difficulty, since most previous theoretical studies of ARD \cite{wipf2011latent,wipf2004perspectives,Wipf2008} are based on the \emph{maximization of the evidence} and its specific closed-form expression for the Gaussian likelihood. It is not immediately clear how to extend the ARD framework to the Poisson noise model in \eqref{poisson} in a manner which will preserve the insights gained from \cite{wipf2011latent,wipf2004perspectives,Wipf2008} and will explain why sparsity is being promoted by the ARD algorithm. 

Perhaps surprisingly, we shall show that the variational algorithm discussed above still preserves the principles of ARD, despite the fact that it is not based on maximizing the evidence, but rather on maximizing the ELBO (lower bound for the evidence).  We call this framework ``Variational ARD'' (\textbf{VARD}).

\subsection{Alternating-Minimization Algorithm for Variational ARD} \label{Pres_AM_VARD}

As mentioned in Sec.~\ref{ARD_EM}, the true posterior distribution $p(\bd{x}|\bd{y},\bd{\gamma})$ corresponding to the Poisson model in \eqref{poisson} and the prior in \eqref{prior} is intractable, so we cannot solve the E-step in \eqref{E_var} exactly, and an approximation to the posterior should be used instead. Perhaps the most common approach is to use a free-form factorized posterior distribution (mean field variational Bayes) \cite{tzikas}, however this is also intractable here due to the lack of conditional conjugacy.
Instead, we restrict the posterior to a parametric multivariate Gaussian form. In addition, in order to reduce the computational complexity, we restrict ourselves to factorized distributions, i.e., \begin{align}
\mathcal{D} \triangleq \{h(\bd{x})\,\, |\,\ h(\bd{x})=\prod_{i=1}^p\cN(x_i|m_i,v_i), \quad \text{for some } m_i,v_i, \,\,i= 1,...,p\} \label{VG},
\end{align}
where $\cN(x_i|m_i,v_i)$ denotes a normal univariate distribution with mean $m_i$ and variance $v_i$. 
We propose to modify the EM algorithm in \eqref{E_var}--\eqref{M_var} by adding the constraint $q \in \mathcal{D}$ during the E-step in \eqref{E_var}. 
The resulting algorithm is defined by the following two steps 
\begin{align}
&\mathrm{\textbf{B}:} \quad (\bd{m}^{(t+1)},\bd{v}^{(t+1)}) =\arg\min_{\bd{m}\succeq 0 ,\bd{v} \succeq 0} \E_{q(\bd{x};\bd{m},\bd{v})}[-\log p(\bd{y}|\bd{x})] + D_{KL}[q(\mathbf{x};\bd{m},\bd{v})||p(\mathbf{x}|\bd{\gamma}^{(t)})], \label{backward} \\
&\mathrm{\textbf{F}:} \quad \bd{\gamma}^{(t+1)} =\arg\min_{\bd{\gamma}} D_{KL}[q(\mathbf{x};\bd{m}^{(t+1)},\bd{v}^{(t+1)})||p(\mathbf{x}|\bd{\gamma})], \label{forward}
\end{align}
which are repeated until convergence. Note that we have rewritten the FVE of \eqref{FVE} as $\mathcal{F}=\E_{q(\bd{x};\bd{m},\bd{v})}\big [-\log p(\bd{y}|\bd{x})\big] + D_{KL}[q(\mathbf{x};\bd{m},\bd{v})||p(\mathbf{x}|\bd{\gamma})]$. The first term in \eqref{backward} does not depend on $\bd{\gamma}$ and was therefore omitted in \eqref{forward}.  In the backward (B) step in \eqref{backward}, we update the approximation to the posterior $q$, which is reduced to only searching for the variational parameters $(\bd{m},\bd{v})$. In the forward (F) step in \eqref{forward}, we update the generative (forward) model by finding $\bd{\gamma}$.  
Note that due to the constraint $q\in\mathcal{D}$ in \eqref{VG}, this is no longer an EM algorithm, and therefore we distinguish between the ``E-step'' and the ``B-step,'' and similarly between the ``M-step'' and the ``F-step.''

Note that the posterior mean $\bd{m}$ is the main object of interest and provides the estimates for the linear attenuation coefficients, which must be nonnegative. Accordingly, a non-negativity constraint on $\bd{m}$ was added in \eqref{backward}. The function in \eqref{backward} is only defined for positive $\bd{v}$ values and using the extended-value function we assume $\mathcal{F}=\infty$ for any $v_j \leq 0$.

The AM algorithm formed by repeating the updates in \eqref{backward}--\eqref{forward} will reduce the FVE at each iteration (and increase the ELBO).   Importantly, as shown below, the AM algorithm leads to near-sparse solutions, i.e., many of the marginal posterior distributions will be highly concentrated around zero in the transform space defined by \eqref{sparse_rep}.  The roles of $\bd{v}$ and $\bd{\gamma}$ in promoting sparse solutions will be made clear in Sec.~\ref{explain}.

For the Poisson model in \eqref{poisson}, the prior in \eqref{prior} and the approximate posterior proposed in \eqref{VG}, the objective function in \eqref{backward} can be written as (dropping iteration number)
\begin{align}
&\mathcal{F}=\underbrace{\mathcal{F}_1(\bd{m},\bd{v})}_{=\E_{q(\bd{x};\bd{m},\bd{v})}[-\log p(\bd{y}|\bd{x})]}+ \qquad \underbrace{\mathcal{F}_2(\bd{m},\bd{\gamma})+\mathcal{F}_3(\bd{v},\bd{\gamma})}_{=D_{KL}[q(\mathbf{x};\bd{m},\bd{v})||p(\mathbf{x}|\bd{\gamma})]}, \label{F} \\
&\mathcal{F}_1(\bd{m},\bd{v})=\sum_{i=1}^n\bigg [y_i \sum_{j=1}^p \phi_{ij} m_j+\eta_i\exp( -\sum_{j=1}^p \phi_{ij} m_j+  \sum_{j=1}^p \phi^2_{ij}v_j /2) \bigg ], \label{F_I} \\
&\mathcal{F}_2(\bd{m},\bd{\gamma})=\frac{1}{2}\sum_{k=1}^p\gamma^{-1}_k (\sum_{j=1}^p \psi_{kj} m_j )^2, \label{F_II}  \\ 
&\mathcal{F}_3(\bd{v},\bd{\gamma})=\frac{1}{2}\sum_{k=1}^p\sum_{j=1}^p \gamma^{-1}_k\psi^2_{kj}v_j  
-\frac{1}{2}\sum_{j=1}^p\log v_j + \frac{1}{2}\sum_{k=1}^p\log\gamma_k,  \label{F_III}
\end{align}   
where $\phi_{ij}=(\bd{\Phi})_{ij}$ are the entries of the system matrix defined after \eqref{poisson}, and $\psi_{ij}=(\bd{\Psi})_{ij}$ are the entries of the $\bd{\Psi}$ matrix defined in \eqref{prior}. In deriving the last term in \eqref{F_III}, we assumed that $\bd{\Psi}$ is square and non-singular\footnote{In deriving \eqref{F_III} we used $\log|(\bd{\Psi}^T\bd{\Gamma}^{-1}\bd{\Psi})^{-1}|=-\log|\bd{\Psi}^T\bd{\Gamma}^{-1}\bd{\Psi}|=\log|\bd{\Gamma}|+C=\sum_k\log\gamma_k+C$, where $C$ is a finite constant and can be excluded. Note that  $(\bd{\Psi}^T\bd{\Gamma}^{-1}\bd{\Psi})$ has to be non-singular in order for this expression to be well-defined. The second equality with $C<\infty$ only holds when $\bd{\Psi}$ is square and non-singular.}. We provide a generalization in Sec.~\ref{gen_AM} which removes this restriction and allows any $\bd{\Psi}$. 

Recalling the discussion in Sec.~\ref{model} about the sparse structure of $\bd{\phi}_i$, we note that the objective function in \eqref{F}--\eqref{F_III} has a computational complexity of $O(np^{1/D})$ \footnote{For convenience we restate the definitions here: $n$ is the number of measurements, $p$ is the total number of pixels/voxels, and $D=2,3$ for 2D/3D images, respectively.} which is feasible to compute in order to assess convergence in practice. In contrast, the original ARD \cite{tipping2001sparse} or other extensions \cite{Wipf2008,seeger2011large} require $O(p^3)$ operations to evaluate the objective which is prohibitive for large scale problems where $p$ is very large, so the objective can only be approximated, e.g., via the stochastic estimator of Papandreou and Yuille \cite{Papandreou2011}.  

The solution to the F-step in \eqref{forward} is obtained by solving  
$\nabla_{\bd{\gamma}}D_{KL}[q||p]=\bd{0}$ which yields
\begin{align}
\bd{\gamma}^{(t+1)}=[\bd{\Psi}\bd{m}^{(t+1)}]\odot [\bd{\Psi}\bd{m}^{(t+1)}] +[\bd{\Psi}\odot\bd{\Psi}]\bd{v}^{(t+1)}, \label{sol_forward}
\end{align}
where $\odot$ denotes elementwise product. It is straightforward to verify that \eqref{sol_forward} is indeed a minimizing point of \eqref{forward} (for finite $\bd{\gamma}$).    
For fixed $\bd{\gamma}$, the FVE objective function is jointly convex with respect to $(\bd{m},\bd{v})$, which implies that any local minimum of the objective in \eqref{backward} is also a global minimum. 
Note that the FVE function is not jointly convex with respect to $(\bd{m},\bd{v},\bd{\gamma})$, so local minima are possible. In a similar way, the original ARD-EM algorithm \cite{tipping2001sparse} is also only guaranteed to find a \emph{local} maximum of the evidence \cite{Wu}.

\subsection{Discussion} \label{explain}
The B-step in \eqref{backward} is well understood from a Bayesian variational inference perspective, and it is equivalent to minimizing the KL divergence between the proposed posterior $q(\bd{x})$ and the true posterior $p(\bd{x}|\bd{y},\bd{\gamma})$, as described in Fig.~\ref{Fig_Back}.

However, the F-step in \eqref{forward} is somewhat puzzling, as it is unclear why it promotes sparse solutions, especially since using a Gaussian prior for MAP estimation does not lead to sparse solutions.   To explain how sparsity is promoted, consider the $\bd{s}$-space defined in \eqref{sparse_rep}, where the image has a sparse representation. 
Note that the KL divergence in \eqref{forward} is invariant under the transformation in \eqref{sparse_rep}. The prior and posterior distributions become
\begin{align}
p(\bd{s}|\bd{\gamma})=\cN(\bd{s}|\bd{0},\bd{\gamma}), \qquad  q(\bd{s})=\cN(\bd{s}|\bd{\mu},\bd{\Sigma}), \qquad 
\bd{\mu}=\bd{\Psi}\bd{m},  \qquad \bd{\Sigma}=\bd{\Psi}\text{Diag}(\bd{v})\bd{\Psi}^T, \label{prior_s} 
\end{align}
where $\bd{\Psi}$ is assumed to be invertible.  Next, we make use of the following representation for a factorized Student's-t distribution  \cite{wipf2004perspectives}
\begin{align}
\text{t}(\bd{s}|\nu) \triangleq {\prod}_{j=1}^p\text{t}(s_j|\nu)={\prod}_{j=1}^p\max_{\gamma_i \geq 0}\cN(s_j|0,\gamma_j) \varphi(\gamma_j|\nu), \label{student}
\end{align}
where $\text{t}(s_j|\nu)=\Gamma[(\nu+1)/2)](1+s^2_j/\nu)^{-(\nu+1)/2}/[\sqrt{\nu\pi}\,\Gamma(\nu/2)]$ is a univariate Student's-t distribution with $\Gamma(\cdot)$ denoting the Gamma function, and $\varphi(\gamma_j|\nu)=\gamma_j^{-\nu/2}\exp(-\nu/2\gamma_j)$. To obtain \eqref{student}, one exploits Fenchel duality  \cite{Boyd}
$\log t(\sqrt{s_j}|\nu)=\max_{\lambda_j\leq 0}(\lambda_j s_j-f^*(\lambda_j))$ where $f^*(\lambda_j)$ is the Fenchel dual of $\log t(\sqrt{s_j}|\nu)$, and then uses the monotonically increasing transformation $\lambda_j=-1/2\gamma_j$ to obtain $\log t(\sqrt{s_j}|\nu)=\max_{\gamma_j\geq 0}(-s_j/2\gamma_j+f^*(-1/2\gamma_j))$.  By exponentiating both sides of the dual representation for $\log t(s_i|\nu)$ one obtains \eqref{student} \cite{wipf2004perspectives}. For $\nu \rightarrow 0$,  $\text{t}(s_i|\nu) \propto 1/|s_i|$ with sharp spikes at $s_i=0$. 
Note in \eqref{student} that $\nu \rightarrow 0$ implies $\varphi\rightarrow 1$, and the multivariate t-distribution $t(\bd{s}|\nu)$ becomes an envelope for the ARD prior $p(\bd{s}|\bd{\gamma})$ in \eqref{prior_s}. This is illustrated in Fig.~\ref{Fig_For} using a 2D example with $\bd{s}\in\mathbb{R}^2$, which depicts the proposed posterior $q(\bd{s})$ and the prior $p(\bd{s}|\bd{\gamma})$, with the latter ``trapped'' inside the level sets of $\text{t}(\bd{s}|\nu)$ due to \eqref{student}. One can see in Fig.~\ref{Fig_For} how selecting the widths $\bd{\gamma}$ of the Gaussian prior $p(\bd{s})$ to minimize the KL divergence between $q(\bd{s})$ and $p(\bd{s})$ promotes the shrinking of $\gamma_k$ for some $k$ (minimizing KL divergence between the two distributions is illustrated as maximizing the overlap between the level-set ellipsoids corresponding to the majority of the probability mass of these distributions). In the example of Fig.~\ref{Fig_EM}, the posterior has a much lower mean and variance along $s_2$ than along $s_1$, and the F-step will further shrink $\gamma_2$. The corresponding parameter $s_2$ will be shrunk in the sense that the marginal prior $p(s_2|\gamma_2)=\cN(s_2|0,\gamma_2)$ will be more concentrated around zero, and so will the true posterior $p(s_2|\bd{y})$ (see \eqref{Bayes}). In the following B-step, the \emph{approximate} posterior $q$ will also concentrate further around zero due to the zero-forcing property of the KL divergence \cite{Bishop} (recall that in the B-step we minimize KL-divergence between $q$ and the true posterior). Note that repeating the B and F steps is required to obtain sparse solutions. 

\vspace{-0.5em}
\begin{figure}[H]
\centering
  \subfigure[Backward Step]{%
      \includegraphics[width=0.4\columnwidth]{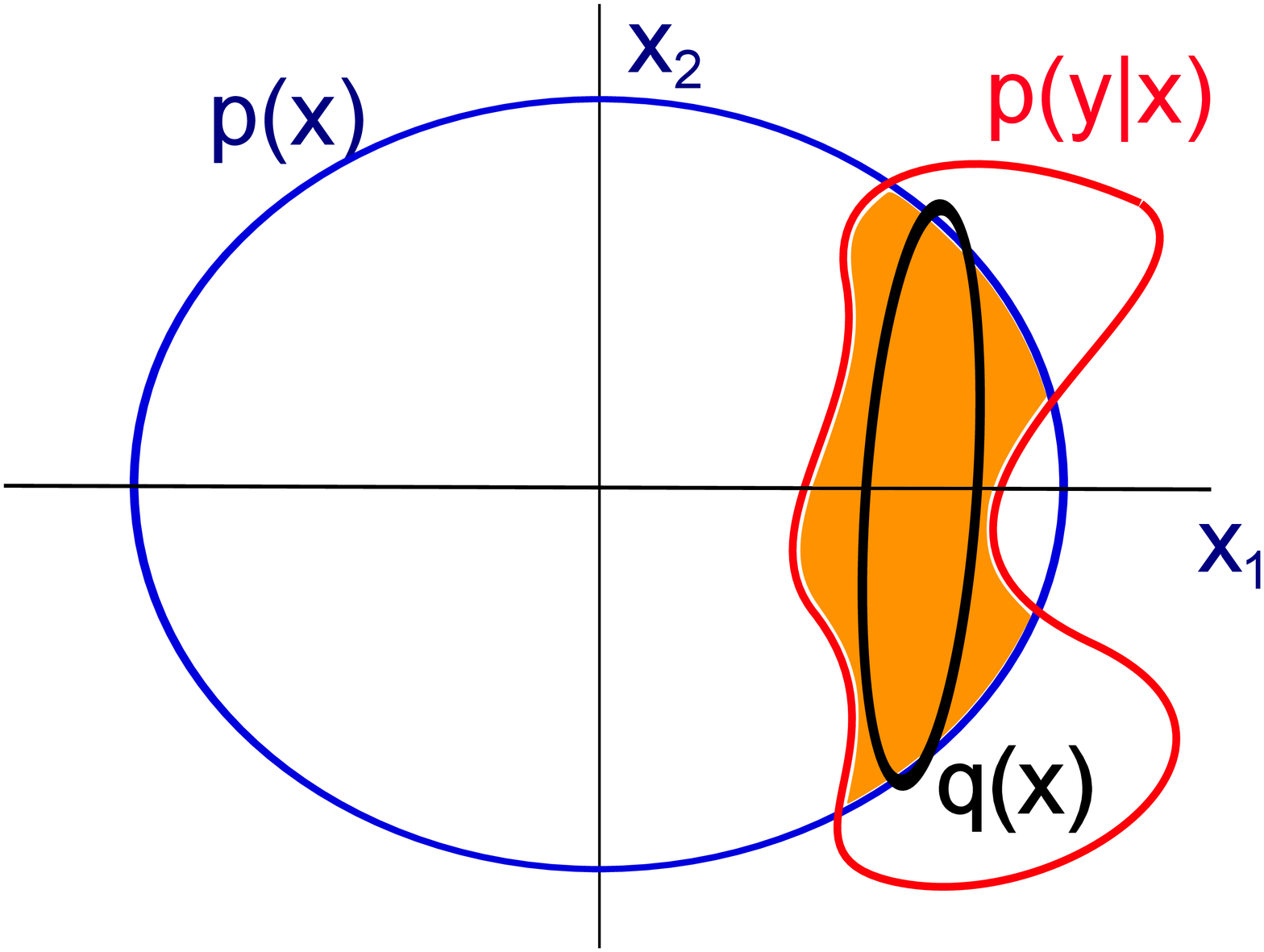}
      \label{Fig_Back}} %
    \subfigure[Forward Step]{%
      \includegraphics[width=0.4\columnwidth]{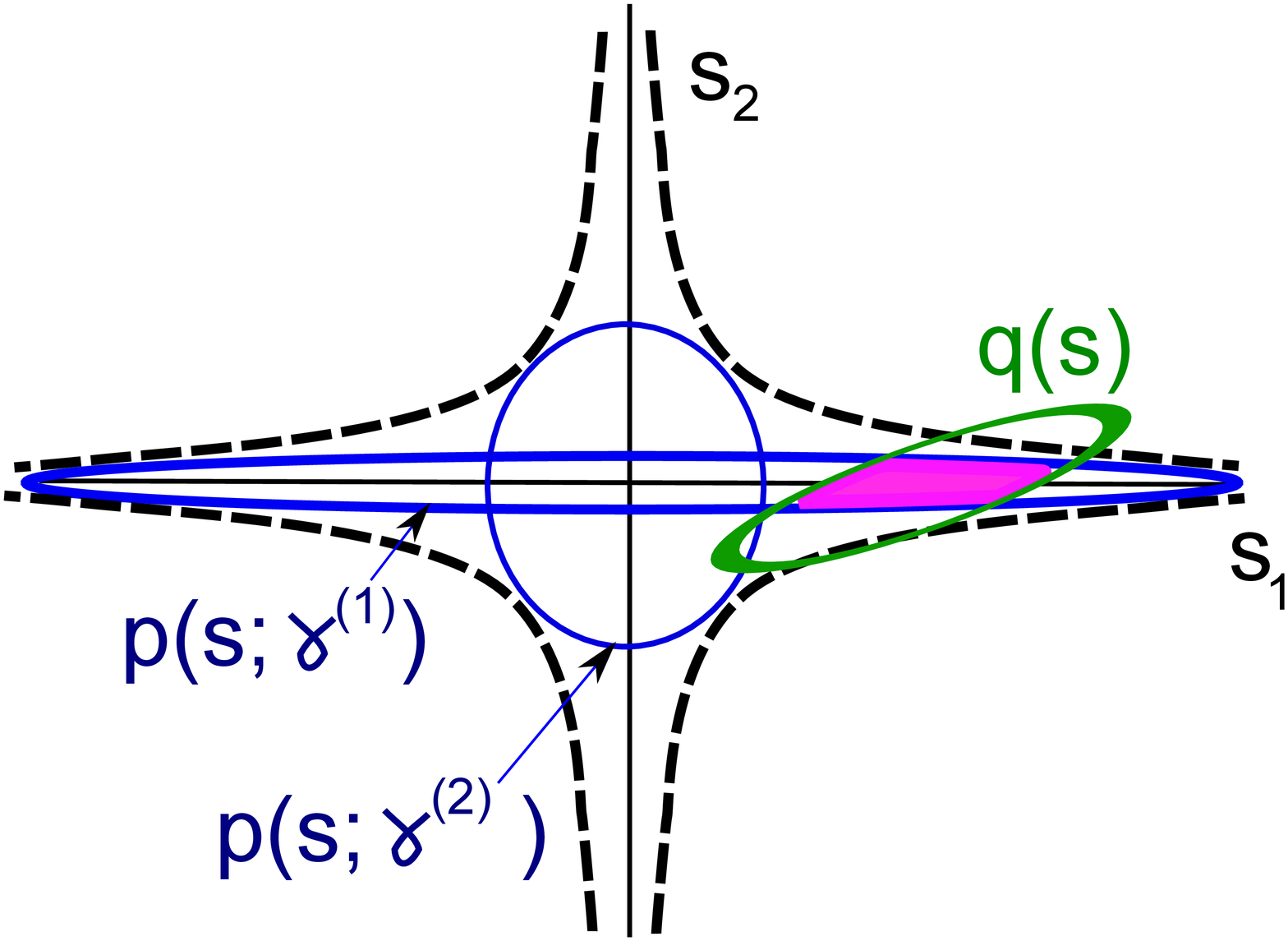} 
      \label{Fig_For}} %
\caption{Description of the AM algorithm proposed in \eqref{backward}--\eqref{forward}. (a) Backward step: level sets for the likelihood $p(\bd{y}|\bd{x})$ (red curve), prior $p(\bd{x})$ (blue curve) and the proposed posterior $q(\bd{x})$ (black curve); true posterior $p(\bd{x}|\bd{y})$ is concentrated in the orange region; proposed posterior $q(\bd{x})$ is chosen to minimize the KL divergence with $p(\bd{x}|\bd{y})$. (b) Forward step: a level set for the student's t distribution (dashed black curves) forming an envelope for the level set of the zero mean Gaussian prior $p(\bd{s})$ (blue curves); $p(\bd{s}|\bd{\gamma}^{(1)})$ and $p(\bd{s}|\bd{\gamma}^{(2)})$ correspond to two different choices for $\bd{\gamma}$ where $p(\bd{s}|\bd{\gamma}^{(2)})$ assigns equal variance along $s_1$ and $s_2$ and $p(\bd{s}|\bd{\gamma}^{(1)})$ assigns small variance along $s_2$ and large variance along $s_1$. Choosing $p(\bd{s}|\bd{\gamma}^{(1)})$ maximizes the overlap between the prior $p(\bd{s})$ and the proposed posterior $q(\bd{s})$ (leading to lower KL divergence). $s_2$ is shrunk in the sense that its exact and proposed posterior distributions are further concentrated around smaller values due to the concentration of the prior about zero. Note that Fig.~(a) is in the $\bd{x}$ domain whereas Fig.~(b) is in the $\bd{s}$ domain.}\label{Fig_EM}
\end{figure}

The description of the mechanism responsible for promoting sparsity was inspired by the work of Wipf et al.\cite{wipf2004perspectives}, which uses similar arguments. However, that work relies entirely on the evidence function $p(\bd{y}|\bd{\gamma})$ in \eqref{evidence} since a Gaussian noise model was considered.  In contrast, our proposed approach is based on the free variational energy (FVE) function and applies more generally, even to cases where the expression for the evidence is unknown and an algorithm that increases the evidence is unavailable. This is important for non-Gaussian noise models, for which the evidence function is intractable, so we are forced to work with the FVE function instead.  
Another important difference between our work and that of Wipf et al.\cite{wipf2004perspectives} is that we show how sparsity is promoted \emph{during the iterations of the algorithm}, whereas in \cite{wipf2004perspectives} it has been shown that a solution to \eqref{evidence_max} is sparse.  
Note also that \cite{wipf2004perspectives} considers the original ARD prior in \eqref{prior_ARD}, whereas we consider the more general case in \eqref{prior}.

\section{Parallel AM Algorithm} \label{fast_VARD}
Here we present a modified AM algorithm, which uses an alternative B-step. The rationale behind the modified algorithm lies in \emph{decoupling} the variational parameters $m_j, v_j$ of each pixel/voxel. This reduces the algorithm into simple 1D optimization tasks that can be computed in parallel, leading to high scalability. First, note that \eqref{F_I} and  \eqref{F_II} are not additively separable with respect to the individual parameters $m_j, v_j$. We shall therefore introduce a surrogate function, making use of the following lemma. 
\begin{lemma}[The convex decomposition lemma \cite{OSullivan}] \label{lemma} 
Let $f(\bd{x})$ be a convex function with $\bd{x}\in\Reals^p$ and let  $\widehat{\bd{x}}\in\Reals^p$ be a reference point. Let $r_j \geq 0$ such that $\sum_{j=1}^p r_j=1$.  Then 
\begin{align}
f(\bd{x})\leq \sum_{j=1}^p r_j f(\widehat{\bd{x}}+\frac{x_j-\widehat{x}_j}{r_j}\bd{e}_j), \label{lemma_eq}
\end{align}\\[-3.5ex]
where $\bd{e}_j$ denotes a column vector with all zeros except 1 at the $j$th entry. \textbf{Equality} in \eqref{lemma_eq} is obtained if $x_j=\widehat{x}_j$ for all $j$  or if $f$ is affine.
\end{lemma}

\begin{proof}
The proof readily follows from Jensen's inequality \cite{Boyd}
\begin{align}
f(\bd{x})=f(\widehat{\bd{x}}+\sum_{j=1}^p (x_j-\widehat{x}_j)\bd{e}_j)=f\big(\widehat{\bd{x}}+\sum_{j=1}^p r_j\frac{x_j-\widehat{x}_j}{r_j}\bd{e}_j\big) \leq \sum_{j=1}^p r_j f(\widehat{\bd{x}}+\frac{x_j-\widehat{x}_j}{r_j}\bd{e}_j), 
\end{align}
where we added and subtracted a reference point $\widehat{\bd{x}}$  and then applied Jensen's inequality to $g(\sum_j r_j [\bd{\nu}_j/r_j])$ where $\bd{\nu}_j\triangleq(x_j-\widehat{x}_j)\bd{e}_j$ and $g(\bd{x})\triangleq f(\widehat{\bd{x}}+\bd{x})$. Equality is obtained iff $\bd{\nu}_j/r_j$ is independent of $j$ or $f$ is affine. The former is obtained if $x_j=\widehat{x}_j$,\, $\forall j$.
\end{proof}

We will also use a minor (but useful) extension of lemma~\ref{lemma} by O'Sullivan and Benac \cite{OSullivan}, where we introduce a dummy variable denoted $x_0$ with a corresponding weight $r_0\geq 0$, such that $f(x_0,\bd{x})=f(\bd{x})$. Accordingly, $\bd{x}$ and $\bd{r}$ in \eqref{lemma_eq} are replaced by $(x_0, \bd{x})$ and $(r_0, \bd{r})$, respectively  ($\bd{r}$ is the collection of all $r_j$'s), and the summation in \eqref{lemma_eq} over $j$ from $1$ to $p$ is replaced by a summation from $0$ to $p$.  This modification leads to the requirement that $\sum_{j=1}^p r_j+r_0=1$ which allows the freedom to choose $\sum_{j=1}^p r_j\leq 1$ instead of strict equality as in Lemma~\ref{lemma} (which of course requires that $r_0=1-\sum_{j=1}^p r_j$). We choose to fix the dummy variable $x_0=\hat{x}_0$ so that the term corresponding to $j=0$ on the right hand side of \eqref{lemma_eq} is just a constant and will not be involved in the following optimization. 

Before we proceed, we make several definitions to simplify the following derivations,
\begin{align}
&p^{(t)}_i \triangleq \sum_j \phi_{ij} m_j^{(t)}, \qquad \tilde{p}^{(t)}_i \triangleq \sum_j\phi^2_{ij}v^{(t)}_j, \label{p}\\
&\mu^{(t)}_i \triangleq \mathbb{E}_{q(\bd{x};\bd{m}^{(t)},\bd{v}^{(t)})} \big [\eta_i\exp(-\bd{\phi}_i^T\bd{x}  ) \big ] = \eta_i \exp  ( \tilde{p}^{(t)}_i/2)\exp (-p^{(t)}_i), \label{mu} \\[1.5ex]
&b^y_j \triangleq \sum_i   \phi_{ij}y_i, \qquad  b^{(t)}_j \triangleq \sum_i   \phi_{ij}\mu^{(t)}_i, \qquad \tilde{b}^{(t)}_j \triangleq \sum_i   \phi^2_{ij}\mu^{(t)}_i/2, \label{b}
\end{align}
where $y_i$, $\bd{\phi}_i$ and $\eta_i$ are defined in \eqref{poisson}. To distinguish between quantities associated with the mean and variance, we denote any quantities associated with the variance using a tilde. 
In \eqref{p}, $p^{(t)}_i$ is the projection of the posterior mean $\bd{m}^{(t)}$ along the $i$th line, which we shall call ``mean-type'' projection; $\tilde{p}^{(t)}_i$ is a projection of the  posterior variance $\bd{v}^{(t)}$ along the $i$th line using an elementwise squared system matrix, which we term a ``variance-type'' projection. In \eqref{mu}, $\mu^{(t)}_i$ is the predicted Poisson rate for the $i$th measurement, based on the expectation with respect to posterior distribution $q$ computed at iteration $t$. In \eqref{b}, $b^y_j$ is the back-projection of measurements $\bd{y}$ to the $j$th pixel; $b^{(t)}_j$ and $\tilde{b}^{(t)}_j$ are the back-projections of the predicted Poisson rate in \eqref{mu} to the $j$th pixel, using the backward (adjoint) operator and the squared backward operator, respectively, which we call ``mean-type'' and ``variance-type'' backprojections. 
 
By noting that \eqref{F_I} can be written as $\mathcal{F}_1=\sum_i \mathcal{F}_{1i}$ and that $\mathcal{F}_{1i}$ is jointly convex with respect to $(\bd{m},\bd{v})$,  we apply Lemma~\ref{lemma}  to each $\mathcal{F}_{1i}$ separately.
Using \eqref{lemma_eq} with its minor extension (mentioned after Lemma~\ref{lemma}) and $\bd{x}=(\bd{m},\bd{v})$, $\widehat{\bd{x}}=(\bd{m}^{(t)},\bd{v}^{(t)})$, and choosing a dummy variable $m_0$ which is set to zero, we obtain the bound $\mathcal{F}_1\leq \mathcal{S}^{(t)}_1(\bd{m},\bd{v})$ at iteration $t$ with $\mathcal{S}^{(t)}_1$ defined as
\begin{align}
\mathcal{S}^{(t)}_1(\bd{m},\bd{v}) =\sum_i \sum_j y_i  \phi_{ij}m_j+\mu^{(t)}_i r_{ij} \exp{\big[-\frac{\phi_{ij}}{r_{ij}}(m_j-m^{(t)}_j)\big]}+\mu^{(t)}_i \tilde{r}_{ij}\exp{\big[ \frac{\phi^2_{ij}}{\tilde{r}_{ij}}(v_j-v^{(t)}_j)/2  \big]} , \label{S1_temp}
\end{align}
where we made use of \eqref{p}--\eqref{b} and $r_{ij}$ corresponds to $m_j$ while $\tilde{r}_{ij}$ corresponds to $v_j$ (for the $i$th measurement). 
Note that the summation over $j$ in  \eqref{S1_temp} includes the dummy variable $m_0$ with $m_0=m_0^{(t)}=0$ for any $t$. 
Also note that \eqref{S1_temp} is separable with respect to the components $m_j$ and $v_j$. 
According to the extended Lemma~\ref{lemma} we are free to choose $r_{ij}, \tilde{r}_{ij}$ as long as \\[-1.2em]
\begin{align}
\sum_j r_{ij}+\sum_j \tilde{r}_{ij}\leq 1, \label{sum_qr}
\end{align}
which must be satisfied for each $i$. We make the following choice for $r_{ij}$ and $\tilde{r}_{ij}$
\begin{align}
r_{ij}=\frac{\phi_{ij}}{Z_1}, \qquad \qquad   \tilde{r}_{ij}=\frac{\phi^2_{ij}}{2Z_1}, \qquad \qquad  Z_1=\max_i \sum_j \phi_{ij}+\phi^2_{ij}/2, \label{qr}
\end{align}
which can be easily verified to satisfy the condition in \eqref{sum_qr}. 
Substituting \eqref{qr} into \eqref{S1_temp}, and using the definitions in Eq.~\eqref{b}, we obtain
\begin{align}
\mathcal{S}^{(t)}_1(\bd{m},\bd{v}) = \sum_j b^y_j m_j+\frac{b^{(t)}_j}{Z_1} \exp{\big[-Z_1(m_j-m^{(t)}_j)\big]} +\frac{\tilde{b}^{(t)}_j}{Z_1}   \exp{\big[ Z_1 (v_j-v^{(t)}_j)  \big]} . \label{S1}
\end{align}
The particular choice in \eqref{qr} is very useful since the sum over $i$ in \eqref{S1_temp} needs to be done only once instead of repeating it for each update of $m_j$ during the search for the minimum of $\mathcal{S}_1$.

We derive a separable bound for \eqref{F_II}, $\mathcal{F}_2\leq \mathcal{S}^{(t)}_2(\bd{m},\bd{v})$, 
by noting that $x^2$ is a convex function and using lemma~\ref{lemma} again to obtain
\begin{align}
\mathcal{S}^{(t)}_2(\bd{m},\bd{v})=
\frac{1}{2}\sum_k\frac{1}{\gamma^{(t)}_k} \sum_j s_{kj} \big [(d^{(t)}_k)^2+2d^{(t)}_k\frac{\psi_{kj}}{s_{kj}} (m_j-m^{(t)}_j) +\frac{\psi^2_{kj}}{s^2_{kj}} (m_j-m^{(t)}_j)^2\big ], \label{S2_temp}
\end{align}
where $d^{(t)}_{k}\triangleq\sum_j \psi_{kj} m^{(t)}_j$ (also defined in \eqref{d_f}). Similarly to \eqref{sum_qr}, we require $\sum_j s_{kj} \leq 1$ using the dummy variable extension.  One choice for $s_{kj}$ is  
\begin{align}
s_{kj}=\frac{|\psi_{kj}|}{Z_2}, \qquad Z_2=\max_k \sum_j |\psi_{kj}|. \label{s}
\end{align}
Substituting \eqref{s} into \eqref{S2_temp}, we obtain 
\begin{align}
\mathcal{S}^{(t)}_2(\bd{m},\bd{v})= h^{(t)}+\sum_j [f^{(t)}_j\,(m_j-m^{(t)}_j)+\,g^{(t)}_j\,(m_j-m^{(t)}_j)^2], \label{S2}
\end{align}
where we introduced the following definitions
\begin{align}
&d^{(t)}_{k}\triangleq\sum_j \psi_{kj} m^{(t)}_j, \qquad \qquad f^{(t)}_j \triangleq \sum_k \psi_{kj}d^{(t)}_k/ \gamma^{(t)}_k, \label{d_f} \\
&g^{(t)}_j\triangleq Z_2\sum_k |\psi_{kj}|/2 \gamma^{(t)}_k, \qquad h^{(t)}\triangleq \sum_k (d^{(t)}_k)^2/2\gamma^{(t)}_k,  \label{g_h}  \\
&\xi^{(t)}_j\triangleq \sum_k\psi^2_{kj}/\gamma^{(t)}_k. \label{xi} 
\end{align}
$d^{(t)}_k$ in \eqref{d_f} is the estimated $k$th element in the sparse representation for the object based on the posterior mean $\bd{m}^{(t)}$, e.g., if $\bd{\Psi}$ is chosen according to \eqref{psi_diff} (assuming sparsity in the pixel/voxel-difference space) then $d^{(t)}_k=m^{(t)}_k-\sum_{j: \{j,k\}\in\Theta}[m^{(t)}_j/N]$ where $\Theta$ defines the neighborhood system and $N$ is the number of neighbors. In deriving \eqref{g_h} we made use of $\sum_j s_{kj}=1$.

Combining \eqref{S1},\eqref{S2}, and \eqref{F_III} (the latter is already separable), we obtain a surrogate function $\mathcal{S}^{(t)}$ that is a global majorizer of $\mathcal{F}$, i.e., $\mathcal{F}(\bd{m}, \bd{\gamma}) \leq \mathcal{S}^{(t)}(\bd{m},\bd{v})$,  given by
\begin{align}
&\mathcal{S}^{(t)}(\bd{m},\bd{v}) =\sum_j \mathcal{S}^{(t)}_{m_{j}}(m_j)+\sum_j\mathcal{S}^{(t)}_{v_j}(v_j), \label{surogate} \\
&\mathcal{S}^{(t)}_{m_j}(m_j)=b^y_j m_j+\frac{b^{(t)}_j}{Z_1} \exp{\big[-Z_1(m_j-m^{(t)}_j)\big]} + f^{(t)}_j\,(m_j-m^{(t)}_j)+\,g^{(t)}_j\,(m_j-m^{(t)}_j)^2, \label{surogate1} \\
&\mathcal{S}^{(t)}_{v_j}(v_j)= \frac{\tilde{b}^{(t)}_j}{Z_1}   \exp{\big[ Z_1 (v_j-v^{(t)}_j)  \big]}+\xi^{(t)}_j v_j/2 -\log v_j/2.  \label{surogate2}
\end{align}   
Note that the surrogate function $\mathcal{S}$ in \eqref{surogate} is separable with respect to both the components of $\bd{m}$ and $\bd{v}$.  This yields a modified B-step where $\mathcal{S}$ is minimized by simultaneous (parallel) 1D optimization tasks, each with respect to a different component $m_j$ or $v_j$. The iterations are defined as
\begin{align}
&m^{(t+1)}_j=\argmin_{m_j\geq 0} \mathcal{S}^{(t)}_{m_j}(m_j), \label{min_S_m} \\
&v^{(t+1)}_j=\argmin_{v_j \geq 0}\mathcal{S}^{(t)}_{v_j}(v_j). \label{min_S_v}
\end{align}
Combining the modified B-step in \eqref{min_S_m}--\eqref{min_S_v} with the F-step in \eqref{sol_forward}, we obtain the parallel VARD  algorithm described in Algorithm 1. 
\begin{algorithm}
\caption{: PAR-VARD (VARD with Separable Surrogates)} \label{alg1}
\begin{algorithmic}[1]
\State Initialize $\bd{m}^{(1)},\bd{v}^{(1)},\bd{\gamma}^{(1)}$
\State Compute $\bd{b}^y$ in \eqref{b} \hfill  \% backprojection of data \%    
\For  {t=1 to N} \hfill \% AM iterations \%  
\State $\bd{p}^{(t)}\leftarrow\bd{\Phi}\bd{m}^{(t)}$  \hfill\%``mean-type ''projections \%
\State $\tilde{\bd{p}}^{(t)}\leftarrow(\bd{\Phi}\odot \bd{\Phi})\bd{v}^{(t)}$ \hfill \% ``variance-type'' projections \% 
\State $\bd{\mu}^{(t)}\leftarrow \bd{\eta}\odot \exp ( \tilde{\bd{p}}^{(t)}/2)\odot\exp (-\bd{p}^{(t)})$  
\State Compute $\bd{g}^{(t)}$, $\bd{\xi}^{(t)}$ and $\bd{f}^{(t)}$ defined in \eqref{d_f}--\eqref{xi} 
\State $\bd{b}^{(t)}\leftarrow\bd{\Phi}^T\bd{\mu}^{(t)}$  \hfill \% ``mean-type'' backprojections of $\bd{\mu}^{(t)}$ \%  
\State $\tilde{\bd{b}}^{(t)}\leftarrow(\bd{\Phi}\odot \bd{\Phi})^T\bd{\mu}^{(t)}$ \hfill \% ``variance-type'' backprojections of $\bd{\mu}^{(t)}$ \%
\For {j=1 to p} \hfill \% executed in parallel \%
\State $m_j^{(t+1)}\leftarrow \argmin_{m_j \geq 0} \mathcal{S}^{(t)}_{m_j}(m_j;m^{(t)}_j,b^{(t)}_j,f_j^{(t)},g_j^{(t)})$ \hfill \% see \eqref{surogate1} \%
\State $v_j^{(t+1)}\leftarrow \argmin_{v_j \geq 0} \mathcal{S}^{(t)}_{v_j}(v_j;v^{(t)}_j,\tilde{b}^{(t)}_j,\xi_j^{(t)})$ \hfill \% see \eqref{surogate2} \%
\EndFor
\State $\bd{\gamma}^{(t+1)} \leftarrow[\bd{\Psi}\bd{m}^{(t)}]\odot [\bd{\Psi}\bd{m}^{(t)}]+[\bd{\Psi}\odot\bd{\Psi}]\bd{v}^{(t)}$ \hfill \% see \eqref{sol_forward} \%
\EndFor
\end{algorithmic}
\end{algorithm}
Each iteration requires a ``mean-type'' forward-projection $\bd{p}^{(t)}$ and a ``variance type'' forward-projection  $\tilde{\bd{p}}^{(t)}$ (lines 4-5 in Algorithm~1) which can be computed in parallel. It also requires a ``mean-type'' backprojection $\bd{b}^{(t)}$ and a ``variance-type'' backprojection
 $\tilde{\bd{b}}^{(t)}$ (lines 8-9 in Algorithm~1), which can also be computed in parallel. 
Each iteration also requires solving two types of 1D optimization problems, one for the mean components $m_j$  and one for the variance components $v_j$ (lines 11-12 in Algorithm~1). 
The 1D optimization tasks for all pixels/voxels can be done in parallel and the scalar updates for mean and variance can also be done in parallel.
Note that despite the latter parallelization, $\bd{b}^{(t)}$ and $\bd{\tilde{b}}^{(t)}$ provide information about both the mean and variance at the previous iteration which is shared during the mean and variance updates in \eqref{min_S_m}--\eqref{min_S_v} (lines 11-12 in Algorithm~1). Also note that these subiterations for $(\bd{m},\bd{v})$  are themselves iterations of another alternating minimization algorithm (lines 4-12 of Algorithm~1). Here we chose to use only one subiteration, but any number can be used instead. 

A considerable advantage of this algorithm is that it allows fast parallel computing on multi-core computers, or on dedicated hardware such as field-programmable gate arrays (FPGA). Another advantage is that the solution to the constrained problem in \eqref{min_S_m} can be obtained by first solving the unconstrained problem using simple gradient-based methods; if $m_j<0$ then the solution to the constrained problem must be $m_j=0$, otherwise the solution is the same. This simple thresholding for $m_j$ leads to the optimal solution only in one-dimensional minimization problems, and the same procedure for the original B-step in \eqref{backward} can lead to sub-optimal solutions or even lack of convergence. 

The computational complexity of a forward/backprojection is $O(np^{1/D})$, where $n$ is the number of measurements, $p$ is the total number of pixels/voxels and $D=2,3$ for 2D/3D images, respectively (recall the discussion about the sparsity of $\bd{\Phi}$ in Sec.~\ref{model}).   The computational complexity of the operations $\bd{\Psi}\bd{m}$ and $(\bd{\Psi}\odot\bd{\Psi})\bd{v}$ is $O(p(N+1))$, where $N$ is the number of neighboring pixels/voxels defined by the prior (recall the discussion on the sparsity of $\bd{\Psi}$ in Sec.~\ref{VARD_prior}). Typically, $N\approx 3^D-1$, $n \ge  p$, and $p^{1/D}\gg N$, so the forward/backprojection cost dominants the cost of $\bd{\Psi}\bd{m}$ and $(\bd{\Psi}\odot\bd{\Psi})\bd{v}$.  
The details regarding the implementations of the 1D optimization tasks (lines 11-12 in Algorithm~1) are given in Appendix~\ref{apx_C}.

\section{Optimality Conditions and Convergence Analysis} \label{convergence}
First we provide the necessary optimality conditions for the minimization problems presented in Secs.~\ref{AM_ARD}--\ref{fast_VARD} as well as a few definitions in order to simplify the following analysis. We then present a detailed study of the convergence properties of the parallel AM algorithm derived in Sec.~\ref{fast_VARD} (Algorithm~1). We shall refer to this algorithm as \textbf{``PAR-VARD''} (parallel VARD).

\subsection{Necessary Optimality Conditions}
\begin{proposition}
\label{Prop_KKT_AM}
\begin{enumerate}
\item The necessary Karush-Kuhn-Tucker (KKT) optimality conditions \cite{Boyd} for $(\bd{m},\bd{v})$ to be a solution to \eqref{backward}  are given by  
\begin{align}
&b^y_j +\sum_k\psi_{kj}d_k/\gamma^{(t)}_k - b_j \,\,\,  
\bigg \{ \begin{array}{lll} 
&\hspace{-1em}=0    &\forall j \text { s.t. } m_j>0 \\
&\hspace{-1em} \geq 0    &\forall j \text{ s.t. } m_j=0
\end{array},  \label{KKT1} \\
&\tilde{b}=\xi^{(t)}_j-1/v_j, \quad \forall j, \label{KKT2}
\end{align}
where $b^y_j, b_j, \tilde{b}_j$ are defined by \eqref{b} with $m_j^{(t+1)}$ and $v_j^{(t+1)}$ replaced by $m_j$ and $v_j$, respectively; $d_k$ is defined by $d^{(t)}_k$ in \eqref{d_f} with $m^{(t)}_j$ replaced by $m_j$;  $\xi^{(t)}_j$ is defined in  \eqref{xi}. 
\item The necessary KKT condition for $\bd{\gamma}$ to be a solution to \eqref{forward} is given by
\begin{align}
\bd{\gamma}=(\bd{\Psi}\bd{m}^{(t+1)}) \odot (\bd{\Psi}\bd{m}^{(t+1)})+(\bd{\Psi}\odot\bd{\Psi})\,\bd{v}^{(t+1)}. \label{KKT3}
\end{align} 
\item The necessary KKT conditions for $(\bd{m},\bd{v},\bd{\gamma})$ to be a minimizer of the FVE objective in \eqref{F}-\eqref{F_III} subject to $\bd{m} \succeq 0$  are given by
\begin{align}
&b^y_j +f_j - b_j \,\,\,  
\bigg \{ \begin{array}{lll} 
&\hspace{-1em}=0    &\forall j \text { s.t. } m_j>0 \\
&\hspace{-1em} \geq 0    &\forall j \text{ s.t. } m_j=0
\end{array},  \label{KKT_G1} \\
&\tilde{b}=\xi_j-1/v_j, \quad \forall j, \label{KKT_G2} \\
&\bd{\gamma}=(\bd{\Psi}\bd{m}) \odot (\bd{\Psi}\bd{m})+(\bd{\Psi}\odot\bd{\Psi})\,\bd{v}, \label{KKT_G3}
\end{align}
where $\xi_j$ is defined as $\xi^{(t)}_j$ in \eqref{xi} with $\gamma_j^{(t)}$ replaced by $\gamma_j$; $f_j$ is defined as $f^{(t)}_j$ in \eqref{d_f} with $d_k^{(t)}$ and $\gamma_k^{(t)}$ replaced by $d_k$ and $\gamma_k$, respectively (see also definitions after \eqref{KKT2}). 
\end{enumerate}
\end{proposition}
\begin{proof} 
Follows directly from the KKT conditions $\nabla_{m_j}\mathcal{F} \geq 0$ $\forall j$ s.t. $m_j=0$, $\nabla_{m_j}\mathcal{F} = 0$ $\forall j$ s.t. $m_j>0$, $\nabla_{\bd{v}}\mathcal{F}=\bd{0}$, $\nabla_{\bd{\gamma}}\mathcal{F}=\bd{0}$, and by plugging in \eqref{F}--\eqref{F_III}. For (1) substitute $\bd{\gamma}=\bd{\gamma}^{(t)}$ and use only KKT conditions for $(\bd{m},\bd{v})$; For (2) substitute  $(\bd{m},\bd{v})=(\bd{m}^{(t+1)},\bd{v}^{(t+1)})$ and use KKT condition for $\bd{\gamma}$.
\end{proof}
\\
{\color{maroon} Remark 1}: note that $v_j>0$ due to the $-\log v_j$ term in \eqref{F_III}. {\color{maroon} Remark 2}: note that \eqref{KKT3} is the F-step update in \eqref{sol_forward} so it is satisfied at the end of each iteration of PAR-VARD.  
\newpage
\begin{proposition}
\label{Prop_KKT_surrogate}
\begin{enumerate}
\item The necessary KKT conditions for $m_j$ to be a solution to \eqref{min_S_m} are 
\begin{align}
&b^y_j  -b^{(t)}_j \exp{\big[-Z_1(m_j-m^{(t)}_j)\big]} + f^{(t)}_j+2g^{(t)}_j\,(m_j-m^{(t)}_j) \, 
\bigg \{ \begin{array}{lll} 
&\hspace{-1em}=0    &\forall j \text { s.t. } m_j>0 \\
&\hspace{-1em} \geq 0    &\forall j \text{ s.t. } m_j=0
\end{array}, 
\label{m_KKT}
\end{align}
where $b^y_j$, $b^{(t)}_j$, $f_j^{(t)}$, $g^{(t)}_j$  and $Z_1$ are defined in  \eqref{b},\eqref{d_f}, \eqref{g_h} and \eqref{qr}, respectively.  
\item The necessary KKT condition for $v_j$ to be a solution to \eqref{min_S_v} is  
\begin{align}
\tilde{b}^{(t)}_j\exp{\big[Z_1(v_j-v^{(t)}_j)\big]} + \xi^{(t)}_j/2-1/2v_j=0, \label{v_KKT}
\end{align}
where $\tilde{b}^{(t)}, \xi^{(t)}_j$ are defined in \eqref{b} and \eqref{xi}. 
\end{enumerate}
\end{proposition}
\begin{proof}
Conditions \eqref{m_KKT}--\eqref{v_KKT} follow directly from the KKT conditions: $\nabla_{m_j}\mathcal{S}_{m_j} \geq 0$ $\forall j$ s.t. $m_j=0$, $\nabla_{m_j}\mathcal{S}_{m_j} = 0$ $\forall j$ s.t. $m_j >0$, $\nabla_{v_j}\mathcal{S}_{v_j}=0$ $\forall j$, and by plugging in \eqref{surogate1}--\eqref{surogate2}. 
\end{proof}
\\
{\color{maroon} Remark}: Note that $v_j>0$ due to the $-\log v_j$ term in \eqref{surogate2}.

\subsection{Definitions}
To simplify the convergence study in Sec.~\ref{sec_conv}, we introduce a few definitions concerning the solutions to \eqref{min_S_m}--\eqref{min_S_v}.  
\begin{definition}
Let $\bar{m}^{(t)}_j$ be a solution to \eqref{min_S_m} but without the non-negativity constraint.  The corresponding solution to \eqref{min_S_m} with the non-negativity constraint can be expressed as 
\begin{align}
m^{(t+1)}_j =\bigg[\bar{m}^{(t)}_j-\frac{1}{Z_1}\log\big (\frac{\beta^{(t)}_j}{b^{(t)}_j} \big)\bigg ]_+, \label{m_update1}
\end{align}
where $[x]_+=\max [0,x]$ and $\beta^{(t)}_j$ is defined as 
\begin{align}
\beta^{(t)}_j \triangleq b^y_j+f^{(t)}_j+2g^{(t)}_j(\bar{m}^{(t)}_j-m^{(t)}_j). \label{beta_m}
\end{align}
\end{definition}
{\color{maroon} Remark}: note from \eqref{b} that $b^{(t)}_j\geq 0$, and from \eqref{m_KKT} and \eqref{beta_m} it follows that $\beta^{(t)}_j \geq 0$,  so the logarithm in \eqref{m_update1} is well defined.  
\begin{definition}
The solution  to \eqref{min_S_m} can alternatively be expressed as
\begin{align}
m^{(t+1)}_j = m^{(t)}_j-\frac{1}{Z^{(t)}_j}\log\big (\frac{\beta^{(t)}_j}{b^{(t)}_j} \big), \label{m_update}
\end{align}\\[-2em]
where we define
\begin{align}
Z^{(t)}_j \triangleq \left\{ \begin{array}{ll}
         Z_1 & \mbox{for $\bar{m}^{(t)}_j \geq 0$} \\
         \frac{1}{m^{(t)}_j}\log\big (\frac{\beta^{(t)}_j}{b^{(t)}_j} \big) & \mbox{for $\bar{m}^{(t)}_j < 0$ and \, $m^{(t)}_j>0$} \\
         \infty & \mbox{for $\bar{m}^{(t)}_j <0$ and  \, $m^{(t)}_j=0$}. \end{array} \right.  \label{Z_tilde_m}
\end{align}
\end{definition}
{\color{maroon} Remark 1}: introducing $Z^{(t)}_j$ in \eqref{m_update} has the same effect as thresholding in \eqref{m_update1}. {\color{maroon} Remark 2}: note that since $m^{(t+1)}\ge \bar{m}^{(t)}$ we have $Z^{(t)}_j \geq Z_1$. 
We also introduce a similar definition for the variance. 
\begin{definition}
The solution to \eqref{min_S_v} can be expressed as
\begin{align}
v^{(t+1)}_j = v^{(t)}_j+\frac{1}{Z_1}\log\big (\frac{\tilde{\beta}^{(t)}_j}{\tilde{b}^{(t)}_j} \big) , \label{v_update}
\end{align}\\[-1.5em]
where  $\tilde{\beta}^{(t)}_j$ is defined as
\begin{align}
\tilde{\beta}^{(t)}_j \triangleq \frac{1}{2v^{(t+1)}_j}-\xi^{(t+1)}_j/2.  \label{beta_v}
\end{align}
\end{definition}
{\color{maroon} Remark}: since $\tilde{b}^{(t)}_j \geq 0$ and $\xi^{(t)}_j \geq 0$, it follows from \eqref{v_KKT} that $v^{(t+1)}_j > 0$. 

\begin{definition}
The I-divergence between two vectors $\bd{p},\bd{q}\in\Reals_+^M$ is defined as
\begin{align}
I(\bd{p}||\bd{q})=\sum_{j=1}^M p_j\log(\frac{p_j}{q_j})-p_j+q_j. \label{I_divergence}
\end{align}
We define $\log(\frac{0}{0})=0$. {\color{maroon} Remark}: the I-divergence is nonnegative. 
\end{definition}
\begin{definition}
The Itakura–-Saito (IS) divergence between two vectors $\bd{p},\bd{q}\in\Reals_+^M$ is defined as
\begin{align}
D_{IS}(\bd{p}||\bd{q})=\sum_k[\log(p_k/q_k)+q_k/p_k-1]. \label{IS_divergence}
\end{align}
We define $\log(\frac{0}{0})=0$, and $\frac{0}{0}=1$. {\color{maroon} Remark}: $D_{IS}$ is nonnegative. 
\end{definition}

\subsection{Convergence Analysis} \label{sec_conv}
For simplicity we shall denote $\mathcal{F}^{(t)} \triangleq \mathcal{F}(\bd{m}^{(t)},\bd{v}^{(t)},\bd{\gamma}^{(t)})$. 
\\[-0.5em]
\begin{corollary} \label{col}
A fixed point of the iterates of PAR-VARD satisfies the necessary KKT conditions in \eqref{KKT_G1}--\eqref{KKT_G3} and also \eqref{KKT1}--\eqref{KKT3}.
\end{corollary}
\begin{proof}
Substituting a fixed point $m_j=m^{(t)}_j$  and $\gamma_k=\gamma^{(t)}_j$  into \eqref{m_KKT} and \eqref{KKT1} reduces \eqref{m_KKT} to \eqref{KKT_G1} and \eqref{KKT1}. Substituting a fixed point $v_j=v^{(t)}_j$  and $\gamma_k=\gamma^{(t)}_k$  into  \eqref{v_KKT} and \eqref{KKT2} reduces \eqref{v_KKT} to  \eqref{KKT_G2} and \eqref{KKT2}. Recall the definitions after \eqref{KKT2} and \eqref{KKT_G3}. The condition in \eqref{KKT_G3} is obtained by substituting $\bd{m}^{(t+1)}=\bd{m}$ and $\bd{v}^{(t+1)}=\bd{v}$ into \eqref{KKT3}.  
\end{proof}
\\[-0.5em]

The following theorem is due to Degirmenci and O'Sullivan. 
\begin{theorem} 
\label{theorem_dec}
The objective function of \eqref{F}--\eqref{F_III}  decreases monotonically during the iterations of PAR-VARD, and the decrease between subsequent iterations is lower bounded by
\begin{align}
&\mathcal{F}^{(t)}-\mathcal{F}^{(t+1)} \geq I\bigg (\frac{\bd{\beta}^{(t)}}{\bd{Z}^{(t)}}\big|\big|\frac{\bd{b}^{(t)}}{\bd{Z}^{(t)}}\bigg)+I\bigg (\frac{\tilde{\bd{\beta}}^{(t)}}{Z_1}\big|\big|\frac{\tilde{\bd{b}}^{(t)}}{Z_1}\bigg) +\sum_j g^{(t)}_j\bigg[\frac{1}{Z^{(t)}_j}\log\bigg (\frac{\beta^{(t)}_j}{b^{(t)}_j} \bigg)\bigg]^2 +\frac{1}{2}D_{IS}(\bd{\gamma}^{(t)}||\bd{\gamma}^{(t+1)}),  \label{B_diff}
\end{align}
where $\bd{\beta}^{(t)}$, $\tilde{\bd{\beta}}^{(t)}$, and $\bd{Z}^{(t)}$ are vectors with components $\beta^{(t)}_j$, 	
$\tilde{\beta}^{(t)}_j$, and $Z^{(t)}_j$, defined in \eqref{beta_m},\eqref{beta_v} and \eqref{Z_tilde_m}, respectively. $\bd{b}^{(t)}$, $\tilde{\bd{b}}^{(t)}$ are vectors with components $b^{(t)}_j$ and $\tilde{b}^{(t)}_j$ defined in \eqref{b}, and $g^{(t)}_j$ is defined in \eqref{g_h} ($g^{(t)}_j\geq 0$). Dividing two vectors should be interpreted as done elementwise. 
\end{theorem}
\begin{proof}
The proof is very technical and therefore moved to Appendix~\ref{apx_A}. 
\end{proof}
\\
Several properties follow from Theorem~\ref{theorem_dec} as stated in Theorem~\ref{theorem_connect}.
\\
\begin{theorem} \label{theorem_connect}
Assume  that $\mathcal{F}^{(1)}$ is finite. Then \\
(a) The I-divergence  $I(\frac{\bd{\beta}^{(t)}}{\bd{Z}^{(t)}}||\frac{\bd{b}^{(t)}}{\bd{Z}^{(t)}})$ converges to zero. \\
(b) The I-divergence  $I(\frac{\tilde{\bd{\beta}}^{(t)}}{Z_1}||\frac{\tilde{\bd{b}}^{(t)}}{Z_1})$ converges to zero. \\[0.3em]
(c) The set of limit points of the mean iterates $\bd{m}^{(t)}$ is a connected set. \\[0.3em]
(d) The set of limit points of the variance iterates $\bd{v}^{(t)}$ is a connected set.
\end{theorem}
\begin{proof}
(a)-(b): From theorem~\ref{theorem_dec} we have 
\begin{align}
&\mathcal{F}^{(1)}-\mathcal{F}^{(t)} =\sum_{s=1}^{t-1} \big [\mathcal{F}^{(s)}-\mathcal{F}^{(s+1)}\big ] \geq\sum_{s=1}^{t-1} I\bigg (\frac{\bd{\beta}^{(s)}}{\bd{Z}^{(s)}}\big|\big|\frac{\bd{b}^{(s)}}{\bd{Z}^{(s)}}\bigg)
+\sum_{s=1}^{t-1} I\bigg (\frac{\tilde{\bd{\beta}}^{(s)}}{Z_1}\big|\big|\frac{\tilde{\bd{b}}^{(s)}}{Z_1}\bigg) \nonumber \\
& +\sum_{s=1}^{t-1}\sum_j g^{(s)}_j\bigg[\frac{1}{Z^{(s)}_j}\log\bigg (\frac{\beta^{(s)}_j}{b^{(s)}_j} \bigg)\bigg]^2 +\sum_{s=1}^{t-1}\frac{1}{2}D_{IS}(\bd{\gamma}^{(s)}||\bd{\gamma}^{(s+1)}). \label{temp_proof}
\end{align}
As $t\rightarrow \infty$ all term on the right-hand side of \eqref{temp_proof} must go to zero since they are nonnegative and have a finite positive upper bound due to the fact that $\mathcal{F}^{(1)}$ is finite and due to the monotonically decreasing and positive sequence $\mathcal{F}^{(t)}$.
\\
(c)-(d): Connectedness of the limit set of the iterates follows from convergence of the I-divergences to zero which implies by Pinsker's inequality \cite{cover}[Lemma 12.6.1] that $\|\bd{\beta}^{(t)}-\bd{b}^{(t)}\|_1 \rightarrow 0$ and $\|\tilde{\bd{\beta}}^{(t)}-\tilde{\bd{b}}^{(t)}\|_1 \rightarrow 0$. This together with \eqref{m_update} and \eqref{v_update} implies that $m^{(t+1)}_j-m^{(t)}_j\rightarrow 0$ $\forall j$, $v^{(t+1)}_j-v^{(t)}_j\rightarrow 0$ $\forall j$. 
\end{proof}
\\ 
\bigskip
We are now able to state the guarantees for global\footnote{Please see footnote~4 on page~6.} convergence of PAR-VARD. 

\begin{theorem}[Global Convergence Theorem] \label{theorem_global}
Let $\{\bd{z}^{(t)}\}_{t=0}^\infty=\{(\bd{m}^{(t)},\bd{v}^{(t)},\bd{\gamma}^{(t)})\}_{t=0}^\infty$ be the sequence of iterates generated by Algorithm~\ref{alg1}. Let the solution set $\Gamma$ be the set of points $(\bd{m},\bd{v},\bd{\gamma})$ which satisfy the KKT conditions in \eqref{KKT_G1}--\eqref{KKT_G3}. Assume there exists $\bd{z}^{(1)}$ such that $\mathcal{F}^{(1)}=\mathcal{F}(\bd{z}^{(1)})$ is finite. Then \\ 
(a) The iterates are contained in a compact set. \\
(b) For $\bd{z}^{(t)}\in \Gamma ,  \mathcal{F}^{(t)}\geq \mathcal{F}^{(t+1)}$. \\
(c) For $\bd{z}^{(t)}\notin \Gamma , \mathcal{F}^{(t)}>\mathcal{F}^{(t+1)}$. \\
(d) The point-to-set mapping defined by Algorithm~\ref{alg1} is closed. \\
(e) All limit points of the iterates are in the solution set $\Gamma$ and $\mathcal{F}^{(t)}\rightarrow \mathcal{F}(\bd{m},\bd{v},\bd{\gamma})$ as $t\rightarrow \infty$ for some $(\bd{m},\bd{v},\bd{\gamma}) \in \Gamma$. 
\end{theorem}
\begin{proof}
(a) By assumption, there is an initial guess $\bd{z}^{(1)}$ such that $\mathcal{F}^{(1)}=\mathcal{F}(\bd{z}^{(1)})$ is finite. From  the $\mathcal{F}$-monotonicity of the sequence $\{\bd{z}^{(t)}\}_{t=1}^\infty$, as asserted by theorem~\ref{theorem_dec}, and the positivity constraints incorporated into the algorithm, the iterates $\{\bd{z}^{(k)}\}_{k=1}^\infty$ are contained in a sub-level set of $\mathcal{F}$ given by $\{\bd{z} \, :  \bd{z} \succeq 0, \mathcal{F}(\bd{z})\leq \mathcal{F}^{(1)} \}$. To show that this set is compact we only need to show that  $\lim_{\|\bd{z}\|_2\rightarrow\infty} \mathcal{F}(\bd{z}) =\infty$. 
Using \eqref{F}--\eqref{F_III}, it can be verified that this condition is satisfied for any combination of $m_j\rightarrow \infty$, $v_j\rightarrow \infty$, $\gamma_k\rightarrow \infty$. \\
(b) Follows from \eqref{B_diff}. \\
(c) Equivalently, we show that if $\mathcal{F}^{(t)}\leq \mathcal{F}^{(t+1)}$ then $\bd{z}^{(t)}\in \Gamma$.  Assume that $\mathcal{F}^{(t)}\leq \mathcal{F}^{(t+1)}$, then from \eqref{B_diff} we must have $\mathcal{F}^{(t)}= \mathcal{F}^{(t+1)}$ and all I divergences in \eqref{B_diff} must equal zero (as they are nonnegative). This implies that  $\beta^{(t)}_j\rightarrow b_j^{(t)}$ and $\tilde{\beta}_j^{(t)}\rightarrow \tilde{b}_j^{(t)}$, $\forall j$. From \eqref{m_update} and \eqref{v_update} it follows that  $m^{(t+1)}_j\rightarrow m_j^{(t)}$, and $v_j^{(t+1)} \rightarrow v^{(t)}_j$ so $\bd{z}^{(t)}$ is a fixed point. It then follows from corollary~\ref{col} that $\bd{z}^{(t)}$ satisfies the KKT conditions, i.e., $\bd{z}^{(t)} \in \Gamma$   \\
(d) We use a result by Gunawardana and Byrne \cite{gunawardana} which is restated here for convenience. 
\begin{proposition}[Proposition~7 in \cite{gunawardana}]
\label{prop_guna}
Given a real-valued continuous function $f$ on $A\times B$, define the point-to-set map $F : A \rightarrow B$ by 
\vspace{-0.25ex}
\begin{align}
F(a) = \argmin_{b\in B} f(a,b)= \{b: f(a,b)\leq f(a,b'), \,\, \forall b'\in B \}.
\end{align}
Then, the point-to-set mapping  $a\rightarrow F(a)$ is closed at $a=a'$ if $F(a')$ is nonempty.
\end{proposition}
\\ 
The surrogate functions in \eqref{surogate1}--\eqref{surogate2} are continuous and convex, and they also satisfy $\mathcal{S}^m_j\ge 0$ and $\mathcal{S}^v_j\geq 0$ for any $m_j, v_j \geq 0$, thus insuring the existence of the solutions to \eqref{min_S_m}--\eqref{min_S_v}. From Proposition~\ref{prop_guna} it follows that the mapping $(\bd{m}^{(t)},\bd{v}^{(t)}) \rightarrow (\bd{m}^{(t+1)},\bd{v}^{(t+1)})$ defined by \eqref{min_S_m}--\eqref{min_S_v} (lines 11-12 of Algorithm~1) is a \emph{closed} point-to-set mapping. Performing several iterations of \eqref{min_S_m}--\eqref{min_S_v} also constitutes a closed mapping, since it is a composition of closed mappings. The mapping $\bd{\gamma}^{(t)}\rightarrow \bd{\gamma}^{(t+1)}$ defined by \eqref{sol_forward} (line 15 of Algorithm~1) is a continuous and therefore closed point-to-point mapping. Finally, the composition of both closed mappings $(\bd{m}^{(t)},\bd{v}^{(t)},\bd{\gamma}^{(t)}) \rightarrow (\bd{m}^{(t+1)},\bd{v}^{(t+1)},\bd{\gamma}^{(t+1)})$ is also closed.    
\\
(e) Follows from (a)--(d) and Zangwill's generalized convergence theorem \cite{Zangwill}.
\end{proof}

\section{Different Perspectives on VARD and Connections to Prior Art} \label{connections}
\subsection{A Different View of VARD} \label{new_view}
We present an alternative view of the proposed VARD framework, which will provide further insight into the sparsity-promoting mechanism and will later enable extensions of the AM algorithm. 
The minimum of the objective function in \eqref{F_I}-\eqref{F_III} with respect to $\bd{\gamma}$ is given by 
\begin{align}
\gamma_k^*=(\sum_j\psi_{kj}m_j)^2+\sum_j\psi^2_{kj}v_j. \label{gamma_opt}
\end{align}
Substituting \eqref{gamma_opt} into \eqref{F_I}-\eqref{F_III} we obtain an alternative formulation for VARD
\begin{align}
&\min_{\bd{m}\succeq 0 ,\bd{v} \succeq 0} \E_{q(\bd{x}|\bd{m},\bd{v})}\big [-\log p(\bd{y}|\bd{x})\big] +\frac{1}{2}\sum_k \log(\mu^2_k+\sigma^2_k)-h[q(\bd{x};\bd{m},\bd{v})], \label{backward2}  \\
&\mu_k=\sum_j\psi_{kj}m_j, \qquad \sigma^2_k =\sum_j\psi^2_{kj}v_j, \qquad h[q(\bd{x};\bd{m},\bd{v})]=\frac{1}{2}\sum_j \log v_j, \label{tilde_mv}
\end{align}
where the explicit expression for $\E_{q(\bd{x}|\bd{m},\bd{v})}\big [-\log p(\bd{y}|\bd{x})\big]$ is given by \eqref{F_I}, and $h[q(\bd{x};\bd{m},\bd{v})]$ in \eqref{tilde_mv} denotes the differential entropy of the distribution $q$. 
$\bd{\mu}$ and $\bd{\sigma}$ in \eqref{tilde_mv} are the posterior mean and standard deviation (std) in the transform domain $\bd{s}=\bd{\Psi}\bd{m}$, respectively.

Note in \eqref{backward2} that the term $\sum_k\log(\mu^2_k+\sigma^2_k)$ promotes sparsity of both the posterior mean and standard deviation in the transform domain. In fact, without the entropy term $h[q]=\sum_j \log v_j/2$ in \eqref{backward2}, any $(\bd{\mu},\bd{\sigma})$ with at least one pair of components $(\mu_k,\sigma_k)=(0,0)$ would be a global minimum of \eqref{backward2}, with $2^p$ possible combinations. The entropy term serves as a barrier preventing solutions with $v_j=0$, thus avoiding these unwanted minima. Note that the concave term $\sum_k\log(\mu^2_k+\sigma^2_k)$ in \eqref{backward2} is not separable with respect to the components of $\bd{m}$ and $\bd{v}$, and the convex decomposition approach used in Sec.~\ref{fast_VARD} is not applicable here.  

\subsection{Connections to Reweighted $\ell_2$ Algorithms} \label{connections_L2}
The ARD approach is related to the reweighted $\ell_2$ algorithm \cite{chartrand2008iteratively} and the comparison between the two for Gaussian likelihoods has been presented in \cite{wipf2010iterative}.  
It is natural to consider the reweighted $\ell_2$ algorithm for the Poisson likelihood as well, which consists of the following two steps
\begin{align}
&\bd{x}^{(t+1)}=\argmin_{\bd{x}\succeq 0}  -\log p(\bd{y}|\bd{x}) +\frac{1}{2} \sum_k (\bd{\Psi}\bd{x})^2_k/\gamma^{(t)}_k, \label{re_ell2_step1} \\
& \gamma^{(t+1)}_k=(\bd{\Psi}\bd{x}^{(t+1)})^2_k+\epsilon, \label{re_ell2_step2}
\end{align} 
where $\epsilon > 0$ is a fixed parameter that needs to be tuned. Note that we cannot choose $\epsilon=0$ since some entries in $\bd{\Psi}\bd{x}$ are zero and this will result in infinite weights for the penalty in \eqref{re_ell2_step1}.
For an easy comparison of \eqref{re_ell2_step1}--\eqref{re_ell2_step2} to the VARD algorithm in \eqref{backward}--\eqref{forward}, we provide here again the update equations for VARD with the posterior variance fixed 
\begin{align}
&\bd{m}^{(t+1)}=\min_{\bd{m}\succeq 0} \E_{q(\bd{x}|\bd{m},\bd{v})}\big [-\log p(\bd{y}|\bd{x})\big]+\frac{1}{2}\sum_k (\bd{\Psi}\bd{m})^2_k/\gamma^{(t)}_k, \label{VARD_re_ell2_step1} \\
& \gamma^{(t+1)}_k=(\bd{\Psi}\bd{m}^{(t+1)})^2_k+(\bd{\Psi}\text{Diag}[\bd{v}^{(t+1)}]\bd{\Psi}^T)_k. \label{VARD_re_ell2_step2}
\end{align} 
It can be seen that in \eqref{re_ell2_step1} the data-fidelity term is the negated log-likelihood, whereas for VARD in \eqref{VARD_re_ell2_step1} it is the expectation of this term with respect to the approximate posterior $q$. A key difference between the two algorithms is that in  \eqref{re_ell2_step2} a tuning parameter $\epsilon$ is required, whereas in \eqref{VARD_re_ell2_step2} the variances serve as ``tuning parameters'' with one parameter for each pixel/voxel, and are automatically learned from the data during the B-step.  

Next, we make another comparison between the two approaches. We rewrite the reweighted $\ell_2$ algorithm in \eqref{re_ell2_step1}--\eqref{re_ell2_step2} as 
$\bd{x}^{(t+1)}=\min_{\bd{x}\succeq 0} Q(\bd{x},\bd{\gamma}^{(t)})$ and $\bd{\gamma}^{(t+1)}=\min_{\bd{\gamma}} Q(\bd{x}^{(t+1)},\bd{\gamma})$,  where $Q$ is defined as 
\begin{align}
& Q(\bd{x},\bd{\gamma})=-\log p(\bd{y}|\bd{x})+\frac{1}{2}\sum_k [(\bd{\Psi}\bd{x})^2_k+\epsilon]/\gamma_k+\frac{1}{2}\sum_k\log \gamma_k. \label{re_ell2_obj1} 
\end{align}
Minimizing \eqref{re_ell2_obj1} with respect to $\bd{\gamma}$ by solving $\nabla_{\bd{\gamma}}Q=0$ results in \eqref{re_ell2_step2} with the superscript $(t+1)$ removed. Substituting this solution into \eqref{re_ell2_obj1}, we obtain the equivalent MAP estimation problem
\begin{align}
\min_{\bd{x}\succeq 0}-\log p(\bd{y}|\bd{x})+\frac{1}{2}\sum_k\log[(\bd{\Psi}\bd{x})^2_k+\epsilon]. \label{re_ell2_obj}
\end{align}
The result in \eqref{re_ell2_obj} can be compared to the alternative formulation for VARD given in \eqref{backward2}--\eqref{tilde_mv}.
Note that small values of $\epsilon$ in \eqref{re_ell2_obj} introduce many local minima, since any solution with at least one component $(\bd{\Psi}\bd{x})_k=0$ will result in a large negative penalty. In this case, any standard gradient-based algorithm will be easily trapped at a sub-optimal solution. On the other hand, large values of $\epsilon$ avoid local minima but do not promote sparse solutions. The correct choice for the parameter $\epsilon$ is therefore critical. 

To facilitate the comparison between reweighted $\ell_2$ and  PAR-VARD, we propose to modify the former using the separable surrogate approach as in Sec.~\ref{fast_VARD}, and derive the corresponding surrogate function.  
The surrogate function for the log-likelihood is given by 
\begin{align}
-\log p(\bd{y}|\bd{x})\leq \sum_j\mathcal{S}^{mle}_j(x_j;x^{(t)}_j) \triangleq \sum_j\bigg\{b^y_j x_j+\frac{b^{(t)}_j}{Z_1} \exp{\big[-Z_{mle}(x_j-x^{(t)}_j)\big]} \bigg\}, \label{S_mle}
\end{align}
where $b^{(t)}_j$ and $b^y_j$ are computed according to \eqref{b} and \eqref{mu} with $\tilde{p}^{(t)}_j=0$ plugged in, and 
\begin{align}
Z_{mle}=\max_i \sum_j \phi_{ij}. \label{Z_mle}
\end{align}
$\mathcal{S}^{mle}_j(x_j)$  in \eqref{S_mle} was derived in \cite{OSullivan} and can be obtained as a special case of the PAR-VARD surrogate in \eqref{surogate1} when $\bd{m}$ is replaced by $\bd{x}$, $\bd{v}\rightarrow 0$, and $\bd{\gamma}\rightarrow \infty$, with consequently $f_j^{(t)},g_j^{(t)}\rightarrow 0$. Adding the surrogate function for the quadratic penalty in \eqref{re_ell2_step1}, we obtain the update equation
\begin{align}
&x^{(t+1)}_j=\arg\min_{x \geq 0} \,\,\mathcal{S}_j^x(x;x_j^{(t)}), \label{SL2_update}\\
&\mathcal{S}_j^x(x;x_j^{(t)})=\mathcal{S}_j^{mle}(x;x_j^{(t)})+ f^{(t)}_j\,(x-x^{(t)}_j)+\,g^{(t)}_j\,(x-x^{(t)}_j)^2, \label{S_L2}
\end{align}
where $f^{(t)}_j, g^{(t)}_j$ are defined in \eqref{d_f}--\eqref{g_h} with $\gamma_j^{(t)}$ given by \eqref{re_ell2_step2} and $m_j^{(t)}$ replaced by $x_j^{(t)}$. The reweighted $\ell_2$ algorithm is then described in Algorithm~2. To the best of our knowledge, this algorithm has not been previously reported in the literature and can be considered as another contribution of this paper. 
\begin{algorithm}[t]
\caption{: Reweighted $\ell_2$ with Separable Surrogates} \label{alg2}
\begin{algorithmic}[1]
\State Initialize $\bd{x}^{(1)}, \bd{\gamma}^{(1)}$
\State Compute $\bd{b}^y$ in \eqref{b} \hfill\% backprojection of data \%
\For {t=1 to N} \hfill \% AM iterations \%
\State $\bd{p}^{(t)}\leftarrow\bd{\Phi}\bd{x}^{(t)}$  \hfill \% projections \%
\State $\bd{\mu}^{(t)}\leftarrow \bd{\eta}\odot\exp (-\bd{p}^{(t)})$  
\State Compute $\bd{g}^{(t)}$,  and $\bd{f}^{(t)}$  \hfill \% defined in \eqref{d_f}--\eqref{g_h} with $m^{(t)}$ replaced by $x^{(t)}$ \%
\State $\bd{b}^{(t)}\leftarrow\bd{\Phi}^T\bd{\mu}^{(t)}$  \hfill \% backprojections of $\bd{\mu}^{(t)}$ \%  
\For {j=1 to p} \hfill \% executed in parallel \%
\State $x_j^{(t+1)}\leftarrow \argmin_{x\geq 0} \mathcal{S}^x(x;x^{(t)}_j,b^{(t)}_j,f_j^{(t)},g_j^{(t)})$ \hfill \% see \eqref{SL2_update}--\eqref{S_L2} \% 
\EndFor
\State $\bd{\gamma}^{(t+1)} \leftarrow[\bd{\Psi}\bd{x}^{(t)}]\odot [\bd{\Psi}\bd{x}^{(t)}]+\bd{1}\epsilon$ \hfill \% $\bd{1}$ is the vector of all ones \% 
\EndFor
\end{algorithmic}
\end{algorithm}

\subsection{Comparison to Maximum Likelihood Estimation (MLE) and Maximum a Posteriori (MAP) Estimation} \label{MLE}
Next, we compare VARD to the AM algorithm for maximum likelihood estimation (MLE) in \cite{OSullivan} and the AM algorithm for penalized maximum likelihood estimation (MAP) in \cite{Keesing,Soysal}. The model  in \cite{OSullivan} is beyond the scope of our work, but it can be reduced to the special case of a monoenergetic model with no background events given in \eqref{poisson} (See also Sec.~\ref{limits} for a more detailed discussion). Although the derivations in this section can be found in \cite{OSullivan,Keesing,Soysal} we present them here to allow for an easier comparison to the proposed VARD algorithm.

The surrogate function for MLE is given by \eqref{S_mle}. The minimizer of this surrogate is given in closed form, leading to the update equation
\begin{align}
x^{(t+1)}_j=\bigg [x^{(t)}_j+\frac{1}{Z_{mle}}\log\big(\frac{b^{(t)}_j}{b^y_j} \big)\bigg ]_+, \label{MLE_update}
\end{align}
with $b^{(t)}_j$ and $b^y_j$ computed according to \eqref{b} and \eqref{mu} with $\tilde{p}^{(t)}_j=0$ plugged in, and $Z_{mle}$ is defined in \eqref{Z_mle}. 
The resulting algorithm is described in Algorithm~3. 
\begin{algorithm}[t] 
\caption{: {\color{blue} MLE}/ MAP with Separable Surrogates} \label{alg3}
\begin{algorithmic}[1]
\State Initialize $\bd{x}^{(1)}$
\State Compute $\bd{b}^y$ in \eqref{b} \hfill \% backprojection of data \%    
\For {t=1 to L} \hfill \% convex decomposition iterations \%
\State $\bd{p}^{(t)}\leftarrow\bd{\Phi}\bd{x}^{(t)}$  \hfill \% projections \%
\State $\bd{\mu}^{(t)}\leftarrow \bd{\eta}\odot\exp (-\bd{p}^{(t)})$  \hfill \% predicted mean for measurements  \%
\State $\bd{b}^{(t)}\leftarrow\bd{\Phi}^T\bd{\mu}^{(t)}$  \hfill\% backprojections of $\bd{\mu}^{(t)}$ \%
\For {j=1 to p} \hfill \% executed in parallel \%
\State {\color{blue} $x^{(t+1)}_j=\big [x^{(t)}_j+\log(b^{(t)}_j/b^y_j)/Z_{mle}\big ]_+$     \hfill \% only for MLE \%}
\State  or $x_j^{(t+1)}\leftarrow \argmin_{x\geq 0} \mathcal{S}^{map}(x;x^{(t)}_j,b^{(t)}_j)$ \hfill \% only for MAP (see \eqref{S_MAP}) \% 
\EndFor
\EndFor
\end{algorithmic}
\end{algorithm}

Using the same approach, it is also possible to add a penalty to the log-likelihood and obtain an AM algorithm for MAP estimation \cite{Keesing,Soysal}. In the following comparisons, we shall use the smooth edge preserving neighborhood penalty given by \eqref{Gibbs}--\eqref{Huber}. The surrogate function for this penalty involving the $j$th pixel is given by
\begin{align}
\mathcal{S}_j^{pen}(x;x_j^{(t)})=\beta\delta^2\sum_{k:\,\{j,k\}\in\Theta}\frac{1}{N_j}\bigg [\big|x^{(t)}_j-x^{(t)}_k+N_j(x-x^{(t)}_j)\big|/\delta-\log\big(\big|x^{(t)}_j-x^{(t)}_k+N_j(x-x^{(t)}_j)\big|/\delta\big)\bigg ],
\label{S_Huber}
\end{align}
where $N_j$ is the number of neighbors of the $j$th pixel.   
The updates for the MAP estimate are then given by 
\begin{align}
&x^{(t+1)}_j=\arg\min_{x \geq 0} \,\,\mathcal{S}_j^{map}(x;x_j^{(t)}), \label{MAP_update} \\
&\mathcal{S}_j^{map}(x;x_j^{(t)})=\mathcal{S}_j^{mle}(x;x_j^{(t)})+\mathcal{S}_j^{pen}(x;x_j^{(t)}), \label{S_MAP}
\end{align}
where $\mathcal{S}_j^{mle}$ and $\mathcal{S}_j^{pen}$ are given by \eqref{S_mle} and \eqref{S_Huber}, respectively. See Algorithm~\ref{alg3}.

\subsection{Connections to Previous Variational Inference Methods}
The approximation of the posterior using a Gaussian distribution has been studied by Challis and Barber \cite{challis2011concave}. Although they did not consider  the Poisson likelihood in \eqref{poisson} specifically, their approach is the same as our B-step in \eqref{backward} if the posterior is chosen to be factorized and if $\bd{\gamma}$ are \emph{fixed} and \emph{known}. We emphasize that this is very different from the AM algorithm proposed here, which alternates between the B-step and the F-step where $\bd{\gamma}$ are also updated. In fact, a single B-step does not result in sparse solutions. In order to obtain sparse solutions, the B and F steps need to be repeated. This is similar in principle to the original EM-ARD algorithm \cite{tipping2001sparse}, where a single E-step would not result in sparse solutions. Also, the parallel algorithm we propose in Sec.~\ref{fast_VARD} is completely new. An alternative approach to \cite{challis2011concave} was proposed by Khan et al.\cite{Khan} but it requires the inversion of a $p\times p$ matrix ($p$ is the number of pixels/voxels) so it is not feasible for large scale problems. 
We note that the method by Seeger and Nickisch \cite{seeger2011large}  cannot be applied here for approximating the posterior since it is restricted to super-Gaussian likelihoods and the Poisson likelihood we consider (see \eqref{poisson}) is not super-Gaussian. 

\section{Generalization of VARD to Improper Priors and Overcomplete Representations} \label{gen_AM}
Building on Sec.~\ref{new_view}, we extend VARD to cases when the prior is improper or when one wishes to use a non-square $\bd{\Psi}$ matrix in \eqref{prior}. Recall that when deriving the AM algorithm in Sec.~\ref{Pres_AM_VARD}, we assumed a proper prior (the matrix $\bd{\Psi}^T\bd{\Gamma}^{-1}\bd{\Psi}$ is non-singular) and that  $\bd{\Psi}$ is a non-singular square matrix (see footnote~5). These assumptions are violated when the Dirichlet boundary conditions in Sec.~\ref{VARD_prior} cannot be used.  For example, in cone beam CT, the object extends beyond the region of interest along one spatial direction, and the Dirichlet boundary conditions are violated in that direction.
It is well known that without boundary conditions, the neighborhood penalties of the form in \eqref{Gibbs} correspond to a singular $\bd{\Psi}$ matrix in \eqref{prior} \cite{Banerjee}, leading to an improper prior. 
The case when $\bd{\Psi}$ is non-square arises when one wishes to use an overcomplete sparse representation, which is shown in Sec.~\ref{results} to yield better image quality. 
Even if $\bd{\Psi}$ is chosen such that the prior is proper, introducing a non-square $\bd{\Psi}$ in the AM algorithm of \eqref{backward}--\eqref{forward} results in a considerable complication of the $\bd{\gamma}$ update (F-step). The objective includes $\log|(\bd{\Psi}^T\bd{\Gamma}^{-1}\bd{\Psi})^{-1}|$ which can no longer be reduced to $\sum_k\log\gamma_k$ ($\bd{\Gamma}=\text{Diag}(\bd{\gamma})$) if $\bd{\Psi}$ is non-square. Taking the gradient of this expression with respect to $\bd{\gamma}$ yields $ \bd{\gamma}^{-2}\odot\text{Diag}[\bd{\Psi}(\bd{\Psi}^T\bd{\Gamma}^{-1}\bd{\Psi})^{-1}\bd{\Psi}^T]$. Both the objective and the gradient have a computational complexity of $O(p^3)$ ($p$ =number of pixels/voxels) which is prohibitive for large-scale problems. 

To avoid the above complications, we use a different approach, where we start with the objective function in \eqref{backward2}--\eqref{tilde_mv}, regardless of how it was derived, where $\bd{\Psi}$ can be square but singular or non-square. This can be justified by the insights gained from Secs.~\ref{new_view}--\ref{connections_L2}. 
The objective in \eqref{backward2}--\eqref{tilde_mv} can be minimized by using an AM algorithm for minimizing the objective function in \eqref{F}--\eqref{F_III}. The latter objective can no longer be interpreted as the free variational energy (FVE) in \eqref{FVE}, even when the prior is proper, since we replaced  $\log|(\bd{\Psi}^T\bd{\Gamma}^{-1}\bd{\Psi})^{-1}|$ with  $\sum_k\log\gamma_k$. Since the replaced term is independent of $\bd{m}$ and $\bd{v}$, minimizing the objective in \eqref{F}--\eqref{F_III} with respect to $(\bd{m},\bd{v})$  during the B-step of the AM algorithm is still equivalent to minimizing the KL divergence between the proposed posterior $q$ and the true posterior\footnote{Note that the posterior distribution defined by \eqref{Bayes} can be well-defined even though the prior is improper.}. 
The F-step, on the other hand, can no longer be interpreted as minimizing the KL divergence between $q$ and the prior, as in \eqref{forward}. Nevertheless, the promotion of sparsity is still asserted by the equivalent formulation in \eqref{backward2}--\eqref{tilde_mv}, where the  $\sum_k\log(\mu^2_k+\sigma^2_k)$ term promotes sparsity in the transform domain, and it is also supported by numerical experiments in Sec.~\ref{results}.
It is important to note that although the interpretation of the objective function during the F-step is different, it is still the same objective as before, given by \eqref{F}--\eqref{F_III}, so the \emph{convergence analysis in Sec.~\ref{convergence} still applies}.
   
To illustrate the type of priors that the above extension allows, we give an example for 2D-images where $\bd{\Psi}$ is a difference matrix given by
\vspace{-2ex}
\begin{align}
\bd{\Psi}=[\bd{\Psi}^h; \bd{\Psi}^v], \qquad &(\bd{\Psi}^{h,v})_{jk}=\bigg\{
\begin{array}{lll}
&1    &\text{if} \quad j=k \\
&-1 &\text{if} \quad  \{j,k\}\in\Theta^{h,v} \\ 
&0    &\text{otherwise}
\end{array}, 
\label{psi_diff1}
\end{align}
where $[\, ; \,]$ denotes row-wise concatenation, and
$\Theta^h$, $\Theta^v$ define pairs of neighboring pixels located along the same horizontal or vertical lines in the image, with corresponding matrices $\bd{\Psi}^h, \bd{\Psi}^v$, respectively. We assign the same weight to horizontal and vertical pixel-differences, so $\bd{\gamma}$ is replaced by $[\bd{\gamma};  \bd{\gamma}]$. It is also possible to  incorporate differences between additional neighbors by concatenating additional $\bd{\Psi}$ matrices for these additional neighbors in \eqref{psi_diff1}, and it is also easily extendible to 3D images. To understand the difference between the above construction and the original one in \eqref{psi_diff}, we write the sparsifying penalty from the objective function in \eqref{backward2} for both cases with only 2 neighbors per pixel for simplicity
\begin{align}
&\mathcal{P}_1=\sum_{j=1}^p \log \bigg\{\sqrt{\bigg (m_j-\frac{(m_{j_1}+m_{j_2})}{2}\bigg)^2+ v_j+\frac{1}{4}v_{j_1}+\frac{1}{4}v_{j_2}}\bigg\}, \label{P1} \\
&\mathcal{P}_2=\sum_{j=1}^p \log \bigg\{\sqrt{(m_j-m_{j_1})^2+ (v_j+v_{j_1}) + (m_j-m_{j_2})^2+(v_j+v_{j_2})} \bigg\} ,  \label{P2}
\end{align} 
where $\mathcal{P}_1$ corresponds to the original construction in \eqref{psi_diff} and $\mathcal{P}_2$ corresponds to the construction in \eqref{psi_diff1}. In both cases, $j_1$ and $j_2$ denote the indexes of the horizontal and vertical neighbors of the $j$th pixel, respectively. For $\mathcal{P}_1$ we have $\{j,j_1\}\in\Theta$ and $\{j,j_2\}\in\Theta$, so both neighbors are included in the same neighborhood defined by $\Theta$. For $\mathcal{P}_2$ we have  $\{j,j_1\}\in\Theta^h$ and $\{j,j_2\}\in\Theta^v$, using two different neighborhoods. In the limit of the variances going to zero, \eqref{P2} reduces to the logarithm of the isotropic total variation (TV) penalty, i.e., $\mathcal{P}_2$ promotes small differences between neighbors equally in each direction. In contrast, using the same limit in the argument of the logarithmic function in \eqref{P1} reduces it to differences between a pixel and the \emph{average} of its neighbors, so $\mathcal{P}_1$ promotes piecewise smoothness.  

\section{Model Limitations} \label{limits}
The algorithm presented in this paper is designed to solve a specific ARD problem. 
As with all models, the model assumed in Sec.~\ref{model} is only approximate, capturing some but not all of the underlying physics.
Considering a simplified physical model also facilitated our theoretical study which could potentially apply more broadly to applications other than the one we consider to motivate our work. 
Next, we discuss the limitations of the forward model presented in Sec.~\ref{model} when applying it to x-ray CT.

The actual data in X-ray transmission tomography are typically not Poisson distributed due to either system non-linearities or fundamental physical considerations. 
For example, the solid state detectors currently used by modern CT scanners are energy-integrating detectors, not photon-counting detectors.
The signal statistics for energy-integrating detectors are studied by Whiting et al. \cite{Whiting1,Whiting2}.
As mentioned in Sec.~\ref{sec_intro}, there is a growing interest in photon-counting detectors, which are already used in clinical studies \cite{Photon_Count}, and it is possible that this type of detectors will be integrated into commercial CT systems in the near future. 
Various researchers have developed algorithms for models that account for various physical phenomena such as the poly-energetic nature of the source \cite{OSullivan}, background events due to scatter \cite{Erdogan,OSullivan}, and electronic noise, which are beyond the scope of this work. A review on modelling the physics in iterative algorithms can be found in Nuyts et al. \cite{nuyts}.

Another simplification we have used is in the construction of the system matrix $\bd{\Phi}$ which was based on line length intersections. This approach does not take into account the detector pixel size or the focal spot size of the source. More realistic approaches for computing the system matrix can be found in the work by Basu and DeMan \cite{Basu,DeMan}. 

Finally, we would like to point to the limitations of the specific priors presented in Sec.~\ref{VARD_prior} and Sec.~\ref{gen_AM} which use the $\bd{\Psi}$ matrix defined in \eqref{psi_diff} and \eqref{psi_diff1}, respectively. 
As mentioned in Sec.~\ref{gen_AM}, in the case of cone-beam CT, the object often extends beyond the region of interest in one spatial direction, so the Dirichlet boundary conditions assumed implicitly in \eqref{psi_diff} and \eqref{psi_diff1} cannot be applied in this direction. In this case, the extension of VARD presented in Sec.~\ref{gen_AM} can be used to choose a different $\bd{\Psi}$ matrix which does not implicitly assume any boundary conditions.

\section{Alternative Algorithms} \label{alter}
In developing the parallel AM algorithm of Sec.~\ref{fast_VARD} we adopted a particular optimization transfer technique for the B-step based on separable surrogates that were derived via the convex decomposition lemma (see lemma~\ref{lemma}). There are several alternative approaches in the literature for CT reconstruction that could be used for the B-step as well. Two important examples are the separable paraboloidal surrogates (SPS) algorithm \cite{Erdogan} and the iterative coordinate descent (ICD) algorithm \cite{bouman1996,Yu}.  
Each different choice of an optimization method for the B-step would result in a different VARD algorithm, which could be compared to the corresponding MAP algorithm using the same optimization method. Here we consider only one possible choice and exploring the others is left for future research. 



\section{Numerical Results}  \label{results}
In the following section we study the performance of the proposed VARD algorithm for x-ray CT. 
We consider both simulated and real data for x-ray CT with source intensity levels ($\eta$ values in \eqref{poisson}) which are typical of clinical settings. The geometry considered in the experiments is shown in Fig.~\ref{fig_fan}.
\begin{figure}[h]
\centering
\includegraphics[width=7cm]{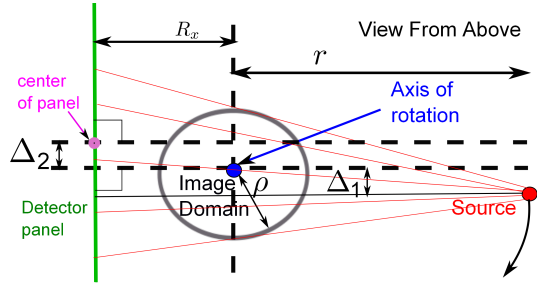}
\caption{Fan-beam geometries with flat detector panels used for all results. For simulated data we used $\rho=100\,mm$, $\Delta_1=0\,mm$, $\Delta_2=0.2\,mm$, $r=R_x=400\,mm$. For  experimental data $\rho=58\,mm$, $\Delta_1=1.2\,mm$, $\Delta_2=0\,mm$, $r=484.6\, mm$, $R_x=290.2\,mm$.}  
\label{fig_fan}
\end{figure}

We compare the following algorithms: (1) The proposed PAR-VARD algorithm described in Algorithm~\ref{alg1}; (2) Maximum likelihood estimator (MLE) from \cite{OSullivan} described in Algorithm~\ref{alg3}; (3) Maximum a Posteriori (MAP) estimator from \cite{Keesing,Soysal} described in Algorithm~\ref{alg3}; (4) The reweighted $\ell_2$ algorithm described in Algorithm~\ref{alg2}. For simplicity, we shall slightly abuse notation and refer to PAR-VARD as VARD. 
For reweighted $\ell_2$ and VARD, we use the prior in \eqref{prior} with two choices for $\bd{\Psi}$ given by \eqref{psi_diff} and \eqref{psi_diff1}, which we refer to as ``complete'' (C) and ``over-complete'' (O) representations, respectively. In \eqref{psi_diff} we choose $\Theta$ to correspond to a neighborhood of 1 horizontal and 1 vertical neighbors. For MAP, we use a sum of two penalties constructed by using \eqref{Gibbs}--\eqref{Huber}, one for the horizontal neighbor and one for the vertical neighbor with Dirichlet boundary conditions (consistent with \eqref{psi_diff} and \eqref{psi_diff1}).  
All algorithms were run for 2000 iterations to ensure convergence (a ``flat'' objective).  The images were initialized by setting all pixel values to zero. For VARD, all posterior variances $\bd{v}$ were initialized to $1$. For VARD and reweighted $\ell_2$,  the hyperparameters $\bd{\gamma}$ were initialized to $100$ to make the datafit term dominant in the first iteration.
All algorithms were implemented in MATLAB and ran on a desktop PC with 4 Intel Xeon E7 CPU's (8 cores each) running a Linux operating system.  
Sample code is available online \cite{website}. 

\subsection{Simulated Data} \label{sec_sim}
The simulated data were generated for a Shepp-Logan phantom with Poisson noise according to the model in \eqref{poisson} and source intensities $\eta=10^4,10^5$, which are typical of clinical systems. We use an image resolution of $256\times 256$, which for the geometry specified in Fig.~\ref{fig_fan} required $1372$ view-angles (full scan around the object) and $512$ detectors for each view ($702,464$ measurements in total). The required angular and detector resolutions were calculated according to classical sampling theory as detailed in \cite{natterer1986computerized}. 

For MAP we had to check different values for two tuning parameters $(\beta,\delta)$ in \eqref{Gibbs}--\eqref{Huber}, and perform many trials. For reweighted $\ell_2$ we had to check different values for $\epsilon$ in \eqref{re_ell2_step2} and run multiple trials as well.  The normalized root-mean-square errors (NRMSE) in the reconstructions obtained using these methods with different values of tuning parameters are shown in Tables~\ref{table_RMSE1}--\ref{table_RMSE2}, with source intensity  set to $\eta=10^5$ and $\eta=10^4$, respectively, and compared to VARD. 
The normalization in the NRMSE is done using the $\ell_2$ norm of the vectorized ground truth image.
One can see that VARD  with the over-complete representation (VARD-O) yields NRMSE that is comparable to the best result with the alternative methods but \emph{without} any tuning, and using only a \emph{single} reconstruction. 
By comparing Tables~1(a),(c) and Tables~2(a),(c), one can see that both reweighted $\ell_2$ and MAP are significantly more sensitive to the choice of tuning parameters when the noise level is higher. Both methods result in much higher NRMSE than VARD when the tuning parameters are not properly chosen.  It is important to note that in practice the object is unknown (NRMSE cannot be computed) so the tuning parameters are chosen \emph{by either examining the reconstructed image for each and every trial}, or using cross-validation, which are very time consuming. 
The lowest NRMSE is obtained using reweighted $\ell_2$ but the differences in the image compared to VARD are visually very minor, as shown below.
 
\begin{table}[h]
\centering
\subtable[MAP]{
\centering
\begin{tabular}{|c|c|c|c|c|}
\hline 
 & $\delta=10^{-5}$ & $\delta=10^{-4}$ & $\delta=10^{-3}$ & $\delta=10^{-2}$ \\ 
\hline 
$\beta=10^5$ & 1.79 \% & 1.52 \% & 0.65 \% & 2.10 \% \\ 
\hline 
$\beta=10^6$ & 1.51 \% & {\color{red} 0.54} \% & 1.16 \% & 9.26 \% \\ 
\hline 
$\beta=10^7$ & 0.55 \% & 0.96 \% & 7.22 \% & 40.01 \% \\ 
\hline 
$\beta=10^8$ & 1.61 \% & 11.5 \% & 43.69 \% & {\color{red} 64.59} \% \\ 
\hline 
\end{tabular}
}
\subtable[MLE/VARD]{
\centering
\begin{tabular}{|c|c|}
\hline 
MLE & 1.63  \% \\ 
\hline 
VARD-C & 0.85 \% \\ 
\hline 
{\color{red} VARD-O} & {\color{red} 0.68} \% \\ 
\hline
\end{tabular} 
}
\subtable[Reweighted $\ell2$]{
\centering
\begin{tabular}{|c|c|c|c|c|}
\hline 
& $\epsilon=10^{-8}$ &$\epsilon=10^{-6}$ & $\epsilon=10^{-4}$ & $\epsilon=10^{-2}$ \\ 
\hline 
Reweighted $\ell_2$-C &  11.03 \% &  0.66 \% & 1.32 \% & 1.81  \% \\ 
\hline 
Reweighted $\ell_2$-O & {\color{red} 2.88} \%  & {\color{red} 0.41} \%  & 0.99 \% &  1.79  \% \\ 
\hline 
\end{tabular}
}
\caption{Normalized root-mean-square errors (NRMSE) for reconstructions from synthetic data with Poisson noise using different methods and different values of tuning parameters. RMSE is normalized relative to the $\ell_2$ norm of the vectorized ground truth image. C=``Complete'' and O=``over-complete'' indicate the use of $\bd{\Psi}$ in \eqref{psi_diff} and \eqref{psi_diff1}, respectively. Here $\eta=10^5$. The standard deviation (std) of the NRMSE over 50 different noise realizations was negligible and therefore was not taken into account, e.g., for VARD-O, std $\approx 0.015\%$. } 
\label{table_RMSE1}
\end{table}
\begin{table}[h]
\centering
\subtable[MAP]{
\centering
\begin{tabular}{|c|c|c|c|c|}
\hline 
 & $\delta=10^{-5}$ & $\delta=10^{-4}$ & $\delta=10^{-3}$ & $\delta=10^{-2}$ \\ 
\hline 
$\beta=10^5$ & 5.29 \% & 3.01 \% & 1.51 \% & 9.30 \% \\ 
\hline 
$\beta=10^6$ & 2.99 \% & {\color{red} 1.36} \% & 7.27 \% & 40.00 \% \\ 
\hline 
$\beta=10^7$ & 1.85 \% & 11.60 \% & 43.69 \% & 64.59 \% \\ 
\hline 
$\beta=10^8$ &  21.94 \% &  58.19 \% &   73.94 \% &  {\color{red} 76.32} \% \\ 
\hline 
\end{tabular}
}
\subtable[MLE/VARD]{
\centering
\begin{tabular}{|c|c|}
\hline 
MLE & 5.64 \% \\ 
\hline 
VARD-C & 2.45 \% \\ 
\hline 
{\color{red} VARD-O} & {\color{red} 1.76} \% \\ 
\hline 
\end{tabular}}
\subtable[Reweighted $\ell2$]{
\centering
\begin{tabular}{|c|c|c|c|c|c|}
\hline 
& $\epsilon=10^{-8}$ &$\epsilon=10^{-6}$ & $\epsilon=10^{-4}$ & $\epsilon=10^{-2}$ \\ 
\hline 
Reweighted $\ell_2$-C & 46.2 \%  & 3.54 \%  & 2.49 \% & 5.23 \% \\ 
\hline 
Reweighted $\ell_2$-O & {\color{red} 21.34} \% &  1.5 \%  & {\color{red} 1.16 }\% & 4.80  \%  \\ 
\hline 
\end{tabular}
}
\caption{Same as Table~\ref{table_RMSE1} but for $\eta=10^4$. }
\label{table_RMSE2}
\end{table}

Some of the reconstructed images using these methods are shown in Fig.~\ref{fig_shepp}. For MAP and reweighted $\ell_2$ we present the results for two different choices of tuning parameters, one of which is the best choice we could find, and the other is a bad choice.  One can see in Figs.~\ref{fig_shepp_map2} and~\ref{fig_re_l2_eps2} how significant artifacts occur when the tuning parameters are not chosen correctly. In contrast, the VARD reconstruction in Fig.~\ref{fig_shepp_VARD} does not have any significant artifacts and the image quality is comparable to the best results obtained with the alternative methods (Figs.~\ref{fig_shepp_map1},\ref{fig_re_l2_eps1}).  
The reconstructed hyperparameters $\bd{\gamma}$ are shown in Fig.~\ref{fig_shepp_xsi} where one can see that $\bd{\gamma}$ has large values at discontinuities and very low values in regions where the attenuation levels are constant. Recall that VARD with the prior in \eqref{psi_diff1} promotes sparsity in the pixel-difference domain. Figure~\ref{fig_shepp_xsi} clearly indicates that VARD has learned the correct support of the object in the pixel-differences domain. The posterior standard deviation (std) is shown in Fig.~\ref{fig_shepp_var} and has a similar structure to $\bd{\gamma}$. Note that one should be careful when 

\begin{figure}[H]
\begin{center}
\subfigure[Truth]{%
\includegraphics[width=6cm]{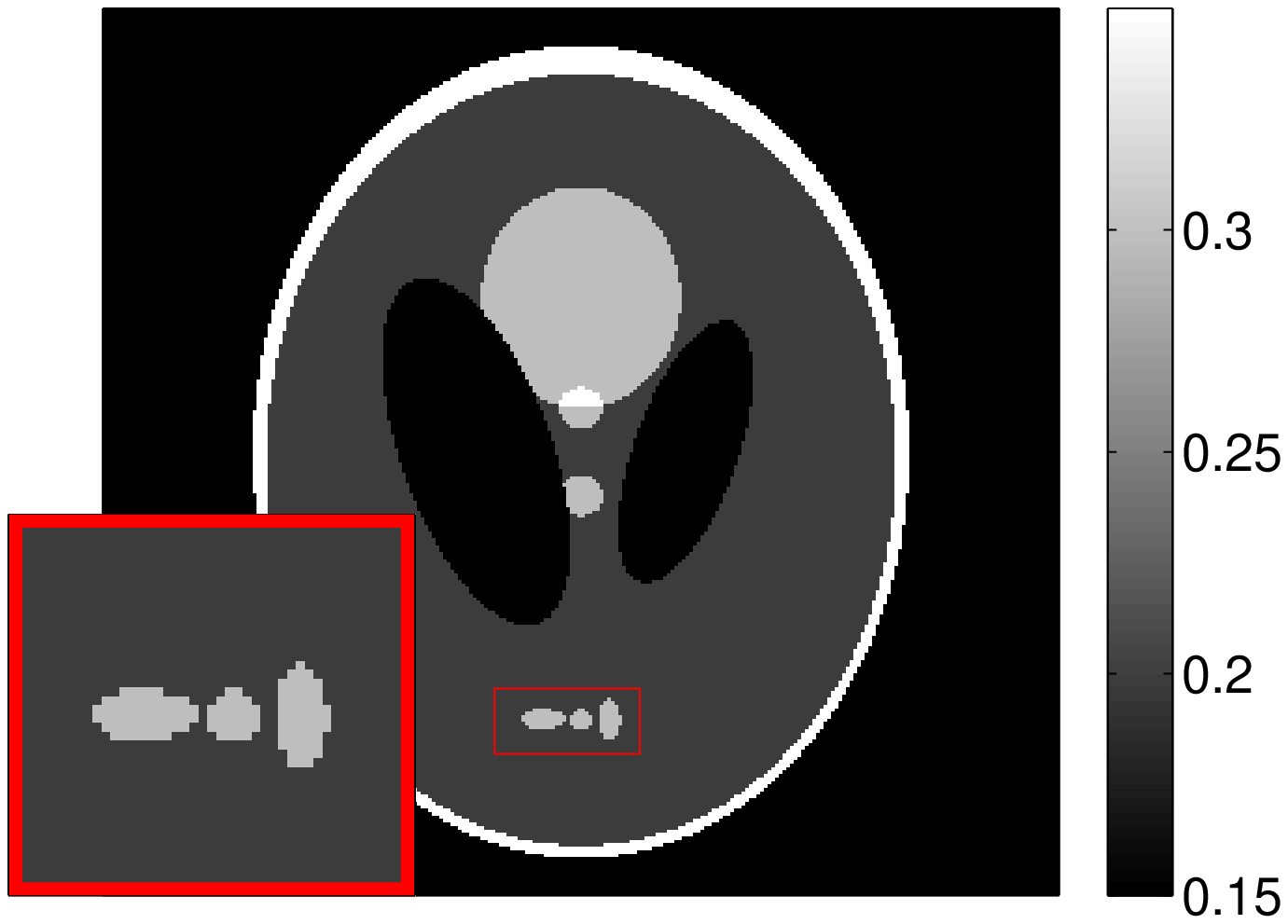}
\label{fig_shepp_truth}} %
\subfigure[MLE (NRMSE=1.63\%)]{%
\includegraphics[width=6cm]{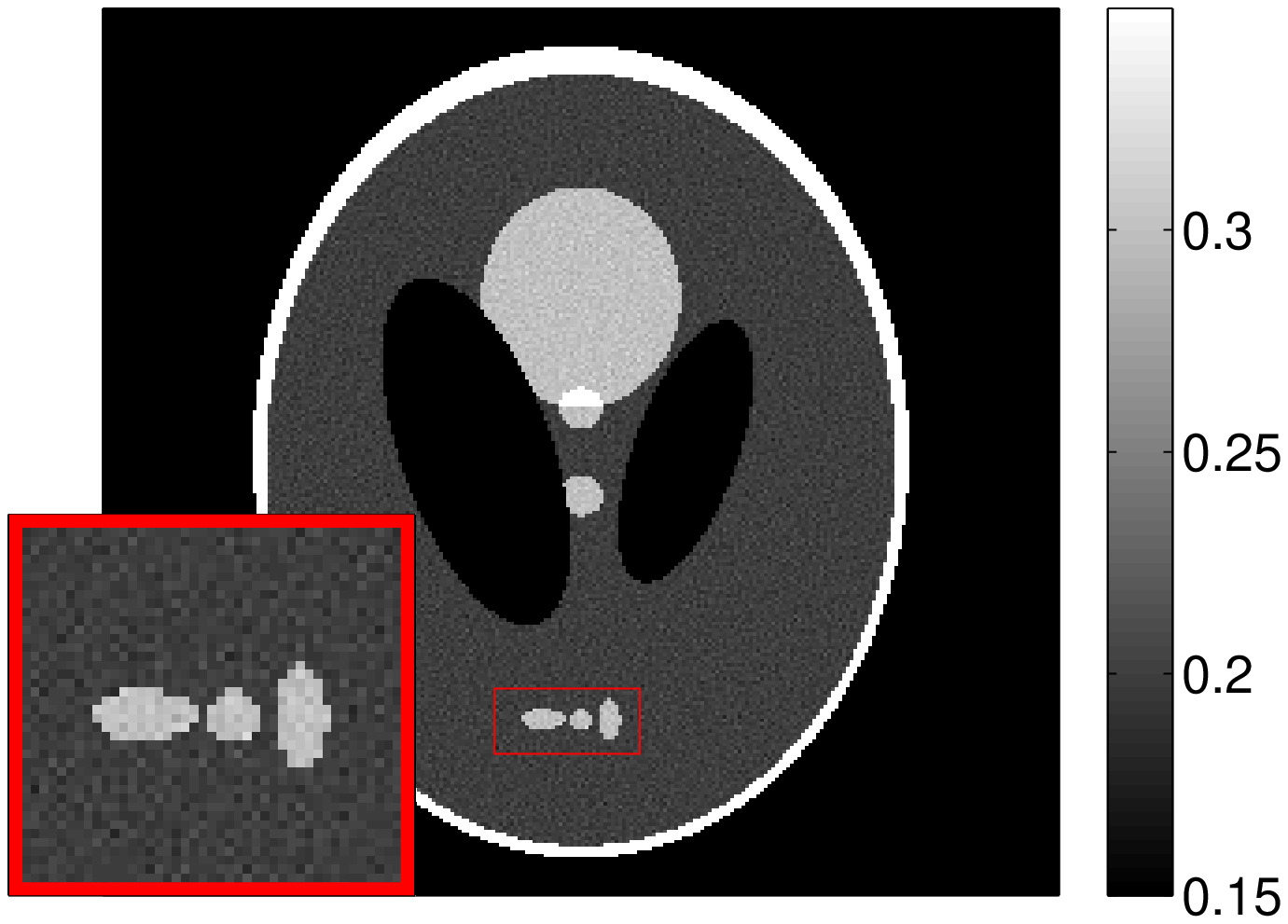} 
\label{fig_shepp_mle}} %
\\
\subfigure[MAP, $\beta=10^6$ (NRMSE=0.5\%)]{%
\includegraphics[width=6cm]{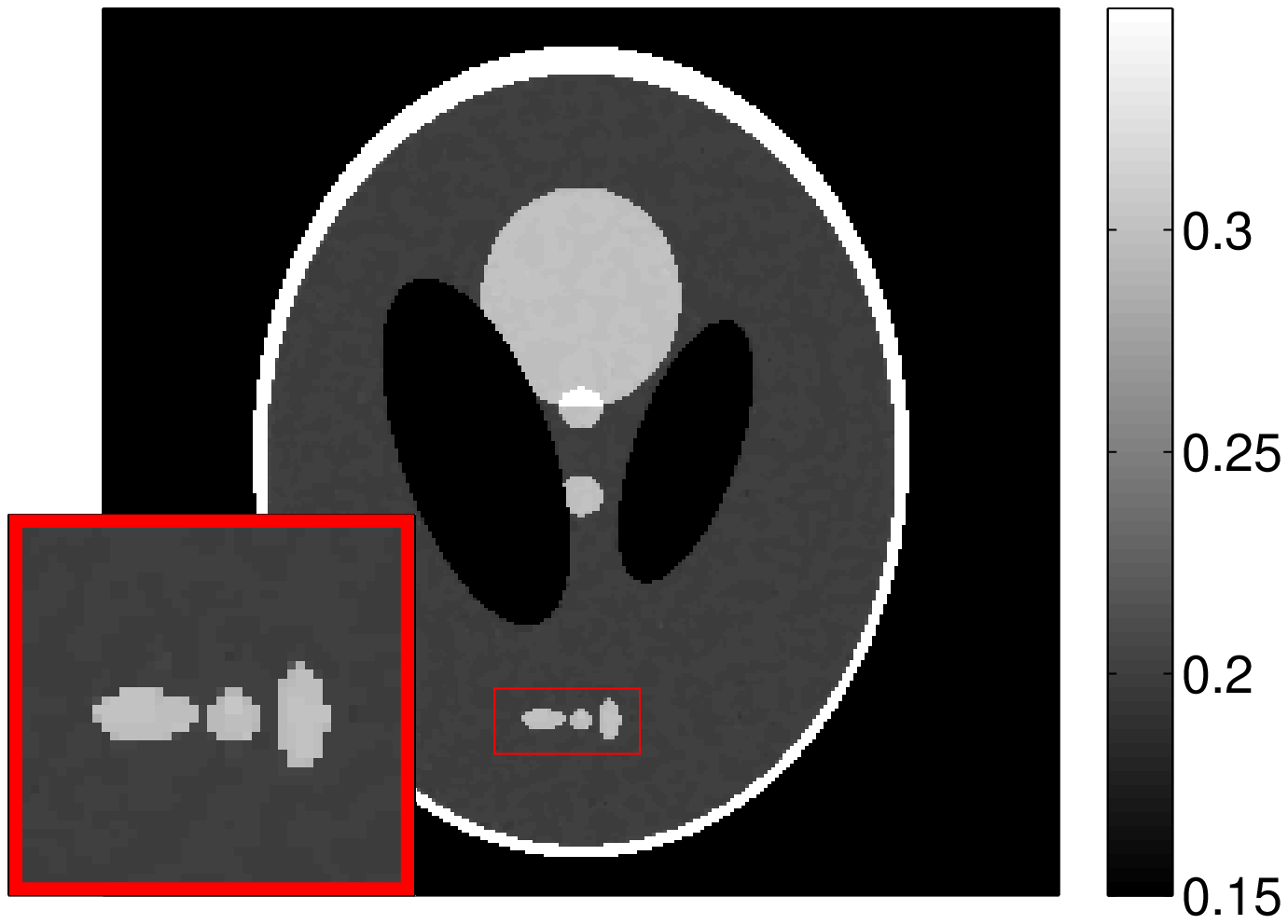} 
\label{fig_shepp_map1}} %
\subfigure[MAP, $\beta=10^7$ (NRMSE=7.2\%)]{%
\includegraphics[width=6cm]{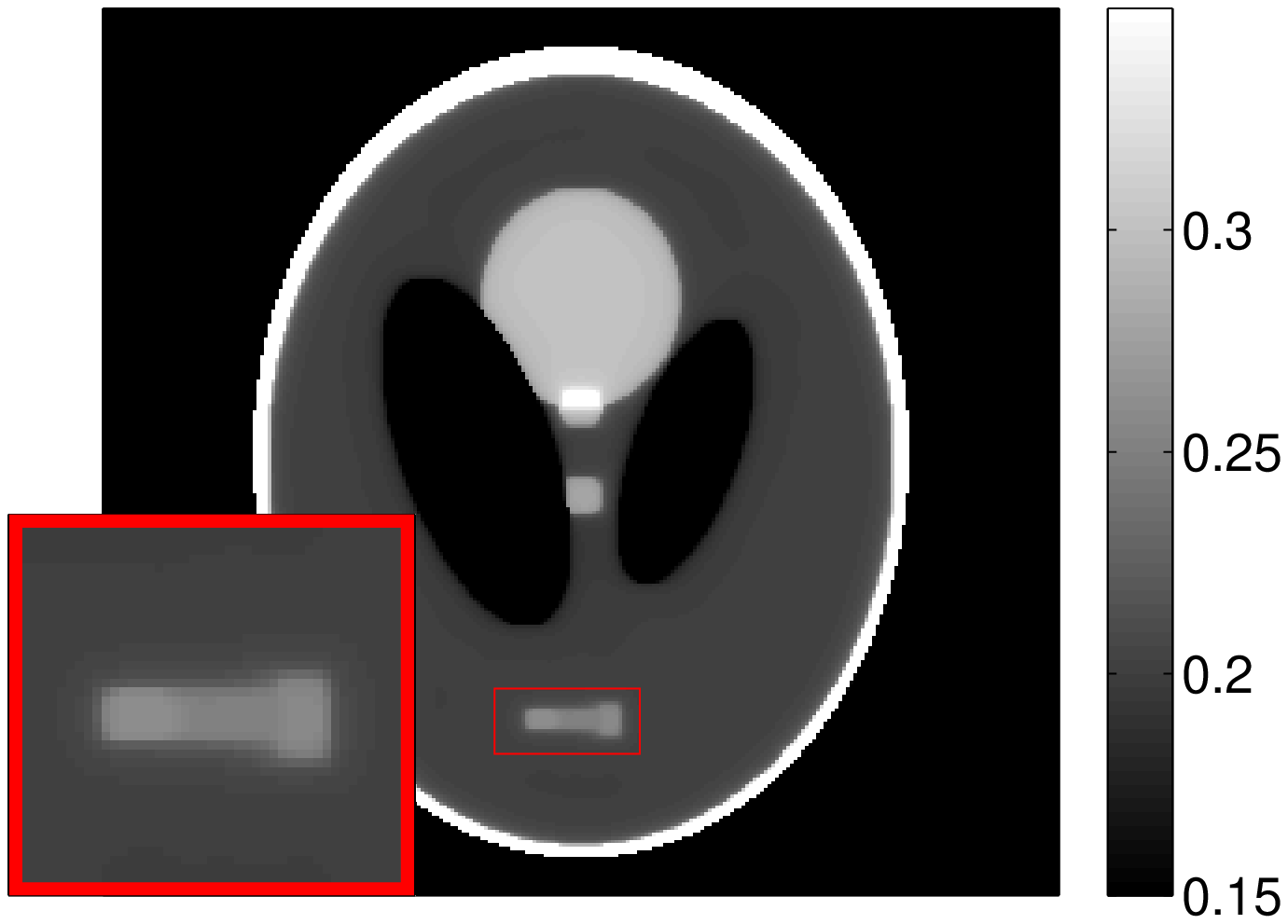} 
\label{fig_shepp_map2}} %
\\
\hspace{-1.5em}
\subfigure[Re-$\ell_2$-O, $\epsilon=10^{-6}$ (NRMSE=0.41\%)]{%
\includegraphics[width=6.3cm]{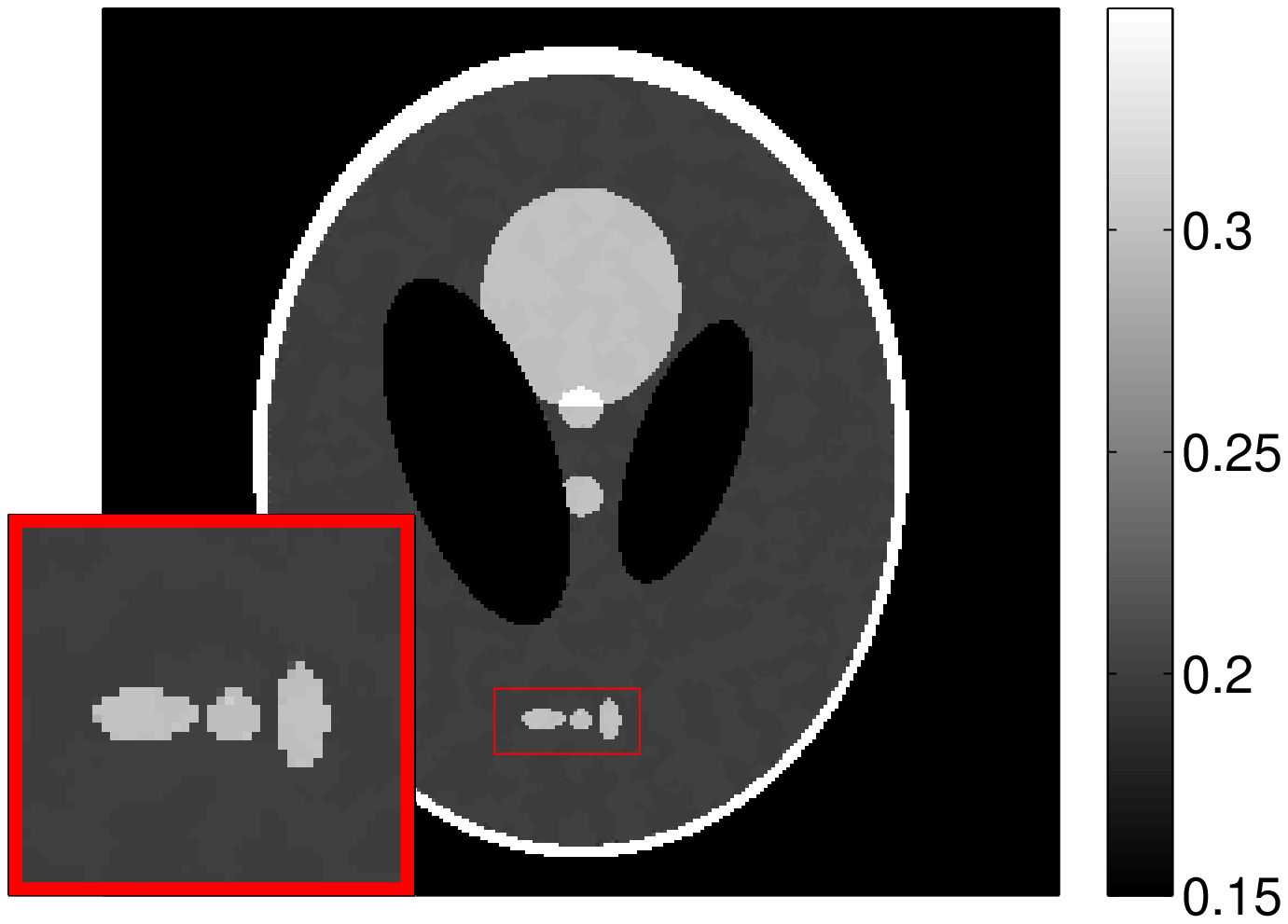}
\label{fig_re_l2_eps1}} %
\subfigure[Re-$\ell_2$-O, $\epsilon=10^{-8}$ (NRMSE=2.88\%)]{%
\includegraphics[width=6cm]{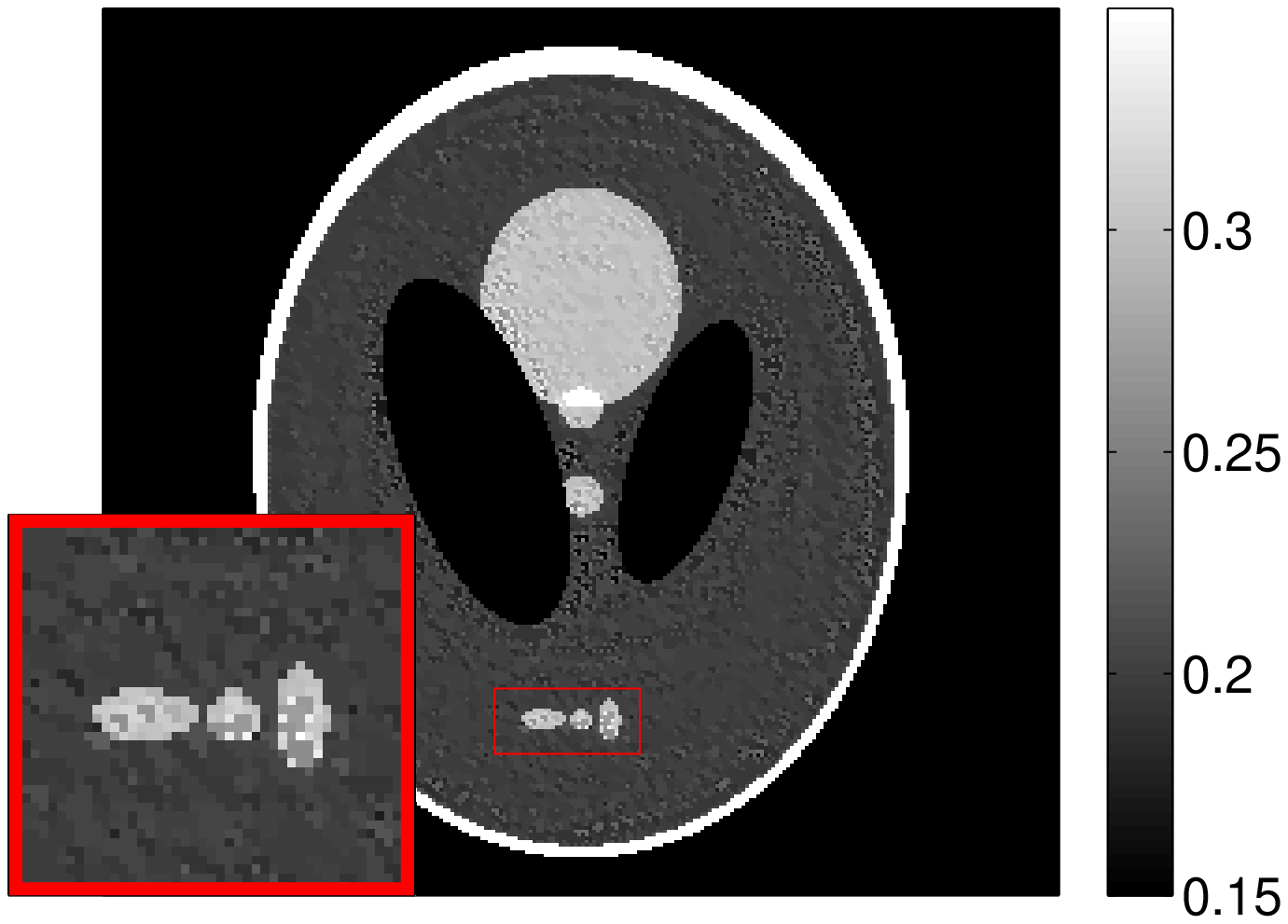}
\label{fig_re_l2_eps2}} %
\\
\subfigure[VARD-O (NRMSE=0.68\%)]{%
\includegraphics[width=5.8cm]{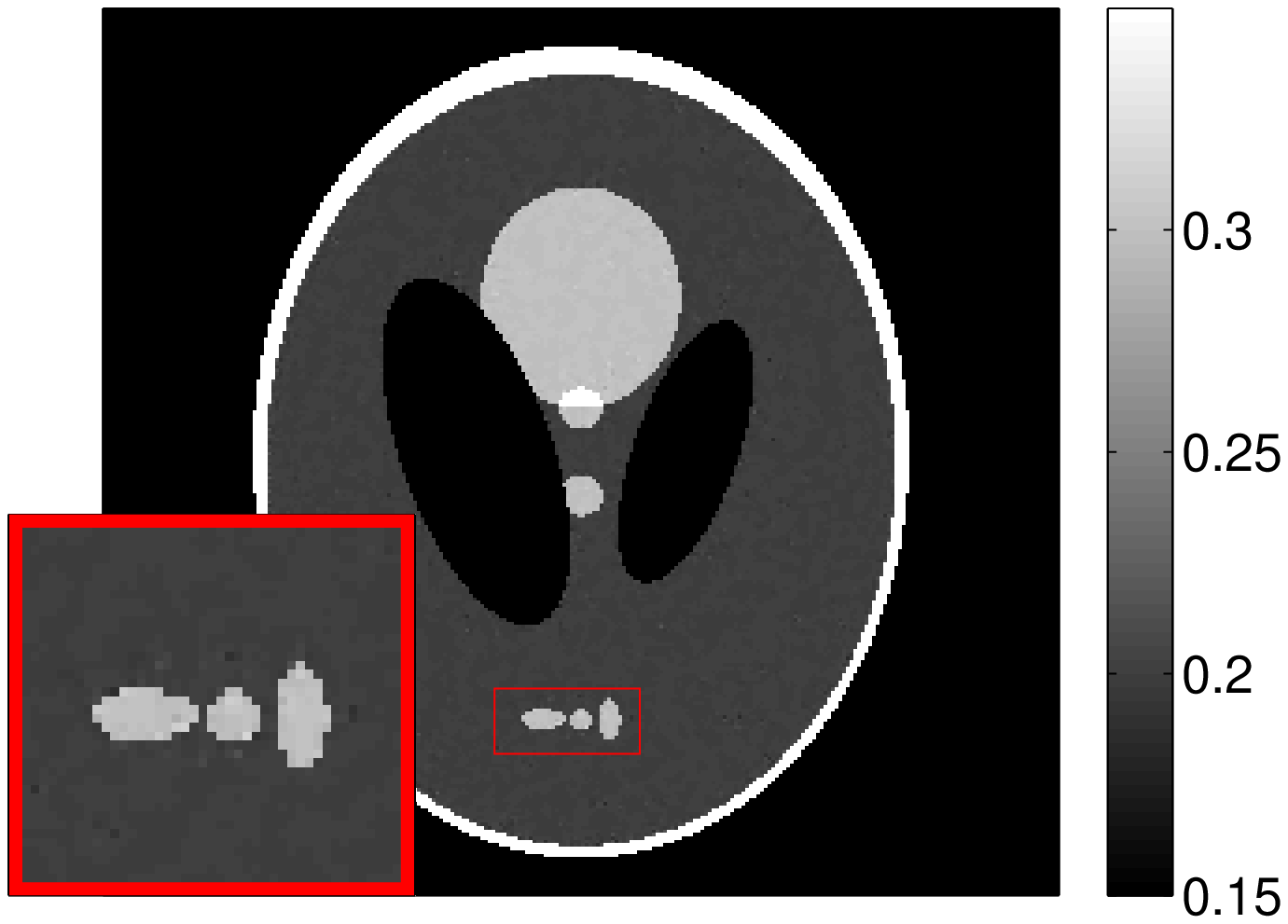}
\label{fig_shepp_VARD}} %
\,
\subfigure[Objective functions]{%
\includegraphics[width=6.2cm]{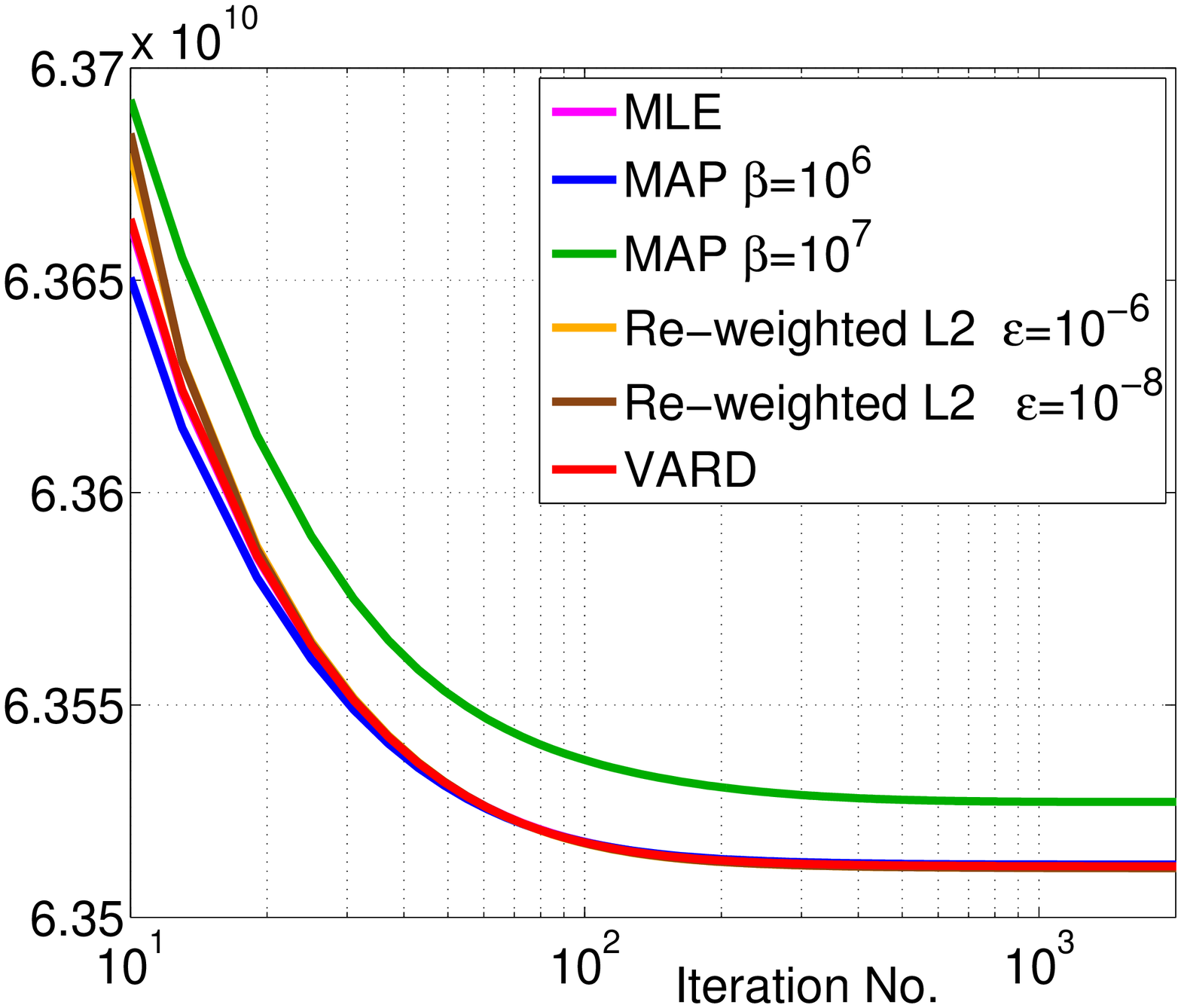} 
\label{fig_shepp_objective}} %
\end{center}
\caption{Reconstructions of attenuation images (relative to water) from synthetic data with Poisson noise using different methods. All images have resolution $256\times 256$. (a) Truth;  (b) Maximum likelihood (MLE) estimate; (c) Maximum a posteriori (MAP) estimate using the smooth edge-preserving neighborhood penalty with $\beta=10^6$ and $\delta=10^{-4}$; (d) same as Fig.~(c) but with $\beta=10^7$ and $\delta=10^{-3}$; (e) Reweighted $\ell_2$-O with $\epsilon=10^{-6}$; (f) same as (e) but with $\epsilon=10^{-8}$ ; (g) Posterior mean using VARD-O;  (h) Objective functions for all methods (note that each algorithm is optimizing a different objective). All images are displayed using the  window $[0.15, 0.35]$, even though the phantom includes values in $[0, 1]$, in order to enhance the differences between the different methods.  A zoom-in view of the region inside the small red frame is shown in the inset to the left in figures (a)--(g). All algorithms except MLE use a neighborhood penalty with  one horizontal and one vertical neighbor per pixel. All algorithms were run till 2000 iterations to ensure convergence. We recommend viewing the figures with at least 150\% zoom-in.}  
\label{fig_shepp}
\end{figure}

\noindent interpreting the posterior std since the support of the posterior is $(-\infty,\infty)$ (see \eqref{VG}) and one std from the mean can extend to negative attenuation values.  In Sec.~\ref{real_dat} we show a useful interpretation of the posterior variances.

Figure~\ref{fig_shepp_objective} presents the objective function for VARD-O (\eqref{F}--\eqref{F_III}) vs. number of iterations, as well as the objectives for all the other methods. The objective is seen to decrease monotonically, as predicted by theory, and can be used to assess convergence in practice, as well as to verify the correctness of the implementation. It can also be seen that all methods share similar convergence rates. It is important to note that each algorithm is optimizing a \emph{different} objective function so to compare convergence rates one needs to compare the slopes of the objectives and not the objective values. 

In Table~\ref{table_ARD} we compare VARD to the ARD method from \cite{Danniel} (SBL) which is based on an approximate post-log Gaussian noise model. Details about SBL can be found in Appendix~\ref{apx_B}. 
After convergence, we set all the negative pixel values in the SBL image to zero in order to obtain a physically meaningful result (since SBL does not impose positivity). From Table~\ref{table_ARD} one can see that SBL leads to higher NRMSE than VARD-C. For the highest photon flux (lowest noise level) with $\eta=10^5$, SBL \cite{Danniel} is comparable to VARD-C, but as the photon flux gets lower ($\eta=10^4$ and $\eta=10^3$), the difference between SBL and VARD-C becomes more noticeable. Since SBL and VARD-C use the same prior, the results suggest that the approximate Gaussian noise model is quite effective for the highest photon flux but becomes less effective as the flux decreases, as expected. Importantly, SBL has a notably higher RSME than VARD-O, especially for lower flux levels, which suggests that the over-complete prior in \eqref{psi_diff1} used by VARD-O provides another advantage over SBL. Note that SBL does not allow the use of the over-complete prior due to the assumptions used in deriving the M-step \cite{Danniel}, which are similar to the ones used for the F-step in Sec.~\ref{Pres_AM_VARD} but were later lifted in Sec.~\ref{gen_AM}.
We initialized SBL in the same way as VARD,  and SBL took 100 iterations to converge. Note that each iteration of SBL has a much higher computational cost than an iteration of VARD (see details in Sec.~\ref{sec_comp}). To illustrate the differences in computational costs, we note that VARD took about 20 minutes to run whereas SBL took 5 hours!  (SBL was run on CPU's, like VARD, and not on a GPU as in \cite{Danniel}).

\begin{table}[h]
\centering
\begin{tabular}{|c|c|c|c|}
\hline
 & $\eta=10^5$ & $\eta=10^4$ & $\eta=10^3$ \\ 
\hline 
SBL &  0.97 \% & 3.15 \% & 9.82 \% \\ 
\hline
VARD-C & 0.85 \% & 2.45 \% & 7.35 \%  \\ 
\hline 
VARD-O & 0.68 \% & 1.76 \% & 5.2 \%  \\ 
\hline 
\end{tabular}
\caption{Comparison of the NRMSE obtained using VARD and SBL \cite{Danniel} for different photon flux levels. 
}
\label{table_ARD}
\end{table}

There are several additional advantages to using VARD instead of SBL: (1) SBL does not enforce positivity of the solution, whereas in VARD, it is built into the algorithm. Note that setting all the negative pixel values to zero after SBL converges might be sub-optimal. (2) Generally, the straightforward computation of the posterior variances requires $O(p^3)$ operations  and $O(p^2)$ of memory ($p$ is the number of pixels/voxels) which is infeasible for large scale problems so approximations are needed \cite{Danniel}. However, these approximations result in the loss of convergence guarantees for the overall SBL algorithm. In contrast, VARD assumes a posterior distribution that while inexact naturally leads to lower computational costs and guarantees convergence. (3) In SBL, one is required to choose the type and number of probing vectors for approximating the variances, which are system and object dependent.  (4) Computing the objective function of SBL scales as $O(p^3)$, making this operation infeasible (in this example, we could not compute the SBL objective  within reasonable time), whereas computing the objective for VARD scales as $O(n p^{1/D})$ with $n$ being the number of measurements and $D=2,3$ for reconstruction of 2D/3D images, respectively. The objective provides another tool for checking the implementation and can be used to assess convergence.
%

\begin{figure}[h]
\begin{center}
\subfigure[VARD hyperparameters $\bd{\gamma}$ (std)]{%
\includegraphics[width=6cm]{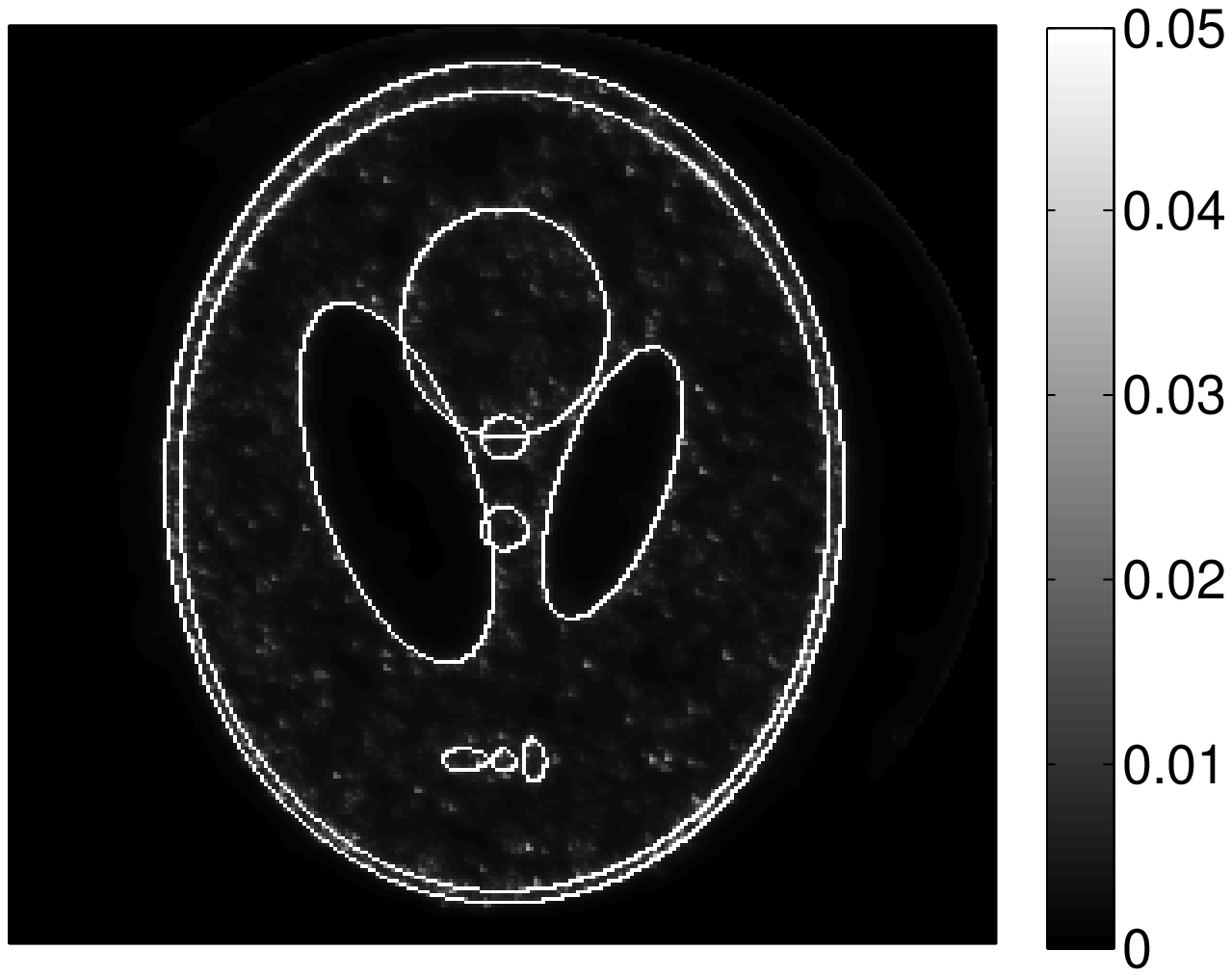}
\label{fig_shepp_xsi}} %
\subfigure[VARD posterior std]{%
\includegraphics[width=6.15cm]{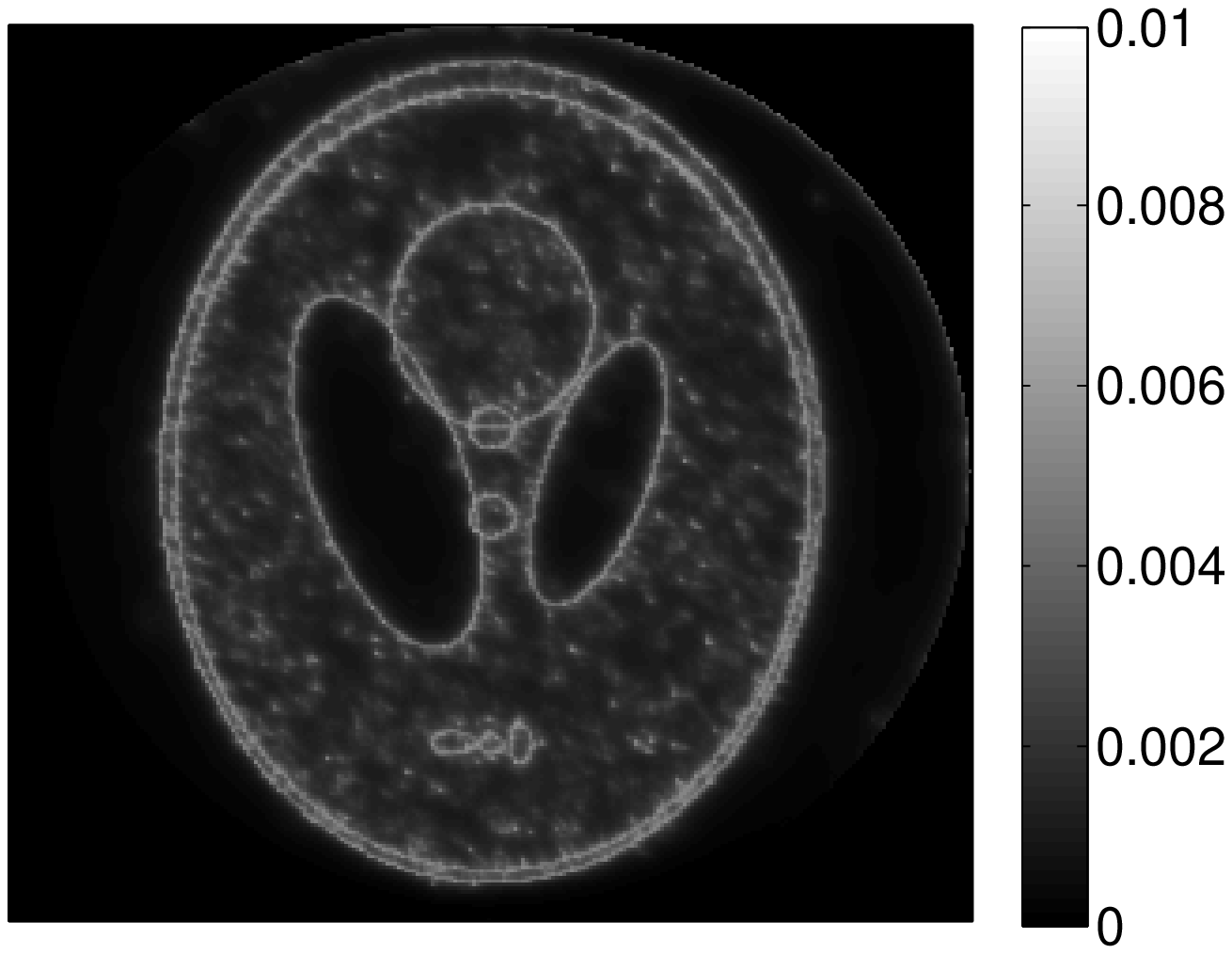} 
\label{fig_shepp_var}} %
\end{center}
\caption{Reconstructed standard deviations (std) using VARD on synthetic data. (a) squared root of hyperparameters $\bd{\gamma}$; (b) squared root of posterior variances $\bd{v}$.} 
\label{fig_shepp_variance}
\end{figure}
\subsection{Real Data} \label{real_dat}

Next we present reconstructions for a dataset acquired at an experimental laboratory at Duke University with the geometry specified in Fig.~\ref{fig_fan}. The number of views was $360$ (full scan around the object) with view-angle resolution of $1^\circ$. For each view $1781$  detector pixels were used. The total number of measurements was $641,160$. Additional details about the experimental setup are provided in Appendix~\ref{apx_D}. We measured a 2D cross section of an acrylic glass cylinder full of cylindrical holes with different diameters. In this case, the image for the ground truth is unknown. However, a good estimate of the truth is based on the specifications from the manufacturer of this phantom. The attenuation  of the acrylic glass relative to water is $1.5$, and for the inside of the rods it is $0$ (air).   Figure~\ref{fig_inv_rods} presents reconstructions using different methods. Image resolution is $256\times 256$. For MAP and reweighted $\ell_2$ we show two examples of  reconstructions, one for the best choice of the tuning parameter that we could find, and one with significant artifacts for a bad choice of tuning parameters.  
Again, one can see how VARD produces images which are comparable in quality to the best trial with MAP or reweighted $\ell_2$ but without any need for repeated trials.  
Figure~\ref{fig_inv_rods_variance} shows the reconstructed posterior variances and hyperparameters, where one can see again how VARD learns the correct support of the object in the pixel-differences domain.

Finally, in Figs.~\ref{fig_duke}--\ref{fig_duke2} we present an example which illustrates the robustness of VARD to missing data and also demonstrates a useful interpretation of the posterior variances. In this example we scanned a ``DUKE'' letters object suspended in mid-air using the same experimental system. The letters are made from VeroBlue material with a linear attenuation 
\begin{figure}[H]
\centering
\subfigure[MLE]{%
\includegraphics[width=6.3cm]{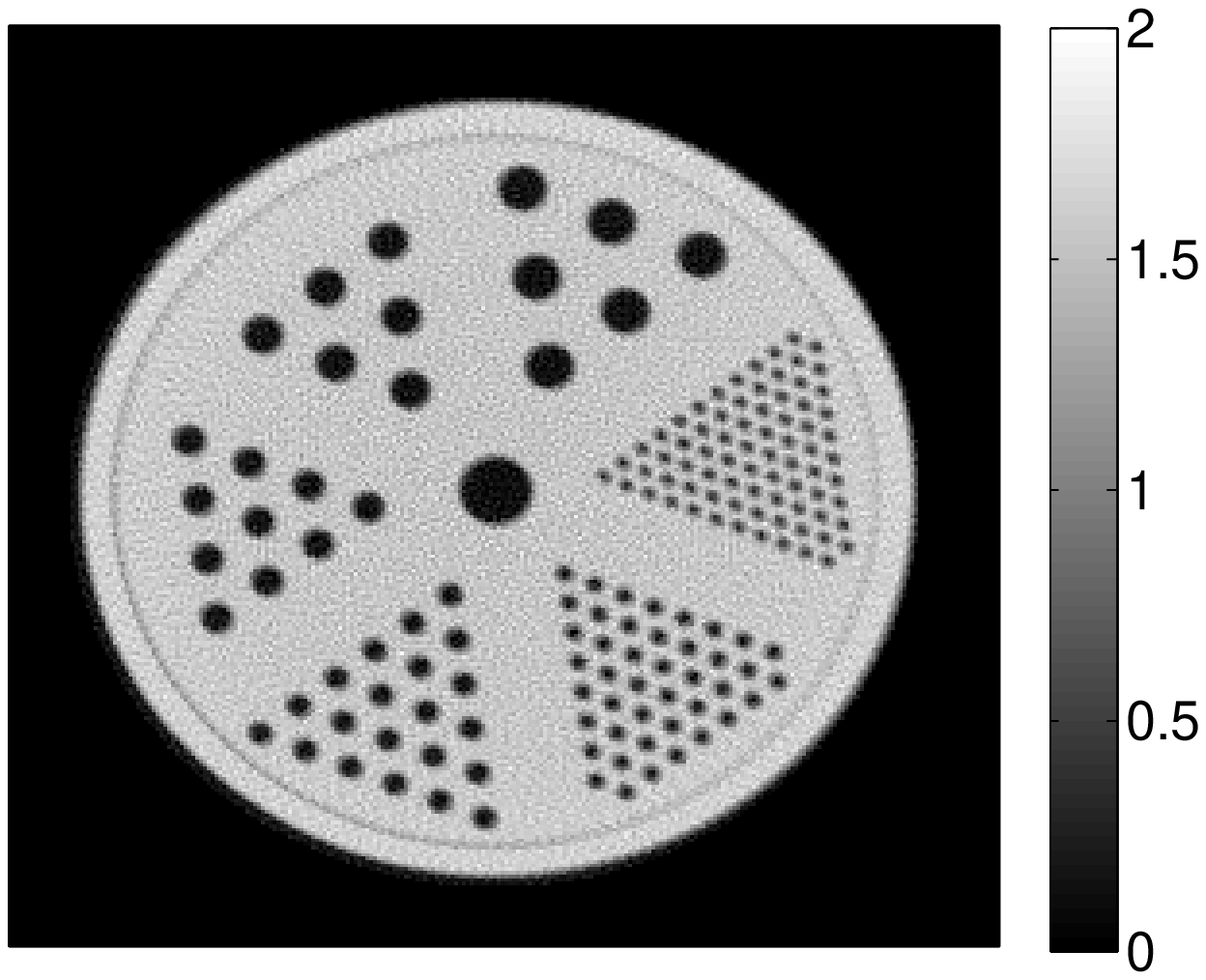}
\label{fig_rods_mle}} %
\subfigure[VARD-O]{%
\includegraphics[width=6.3cm]{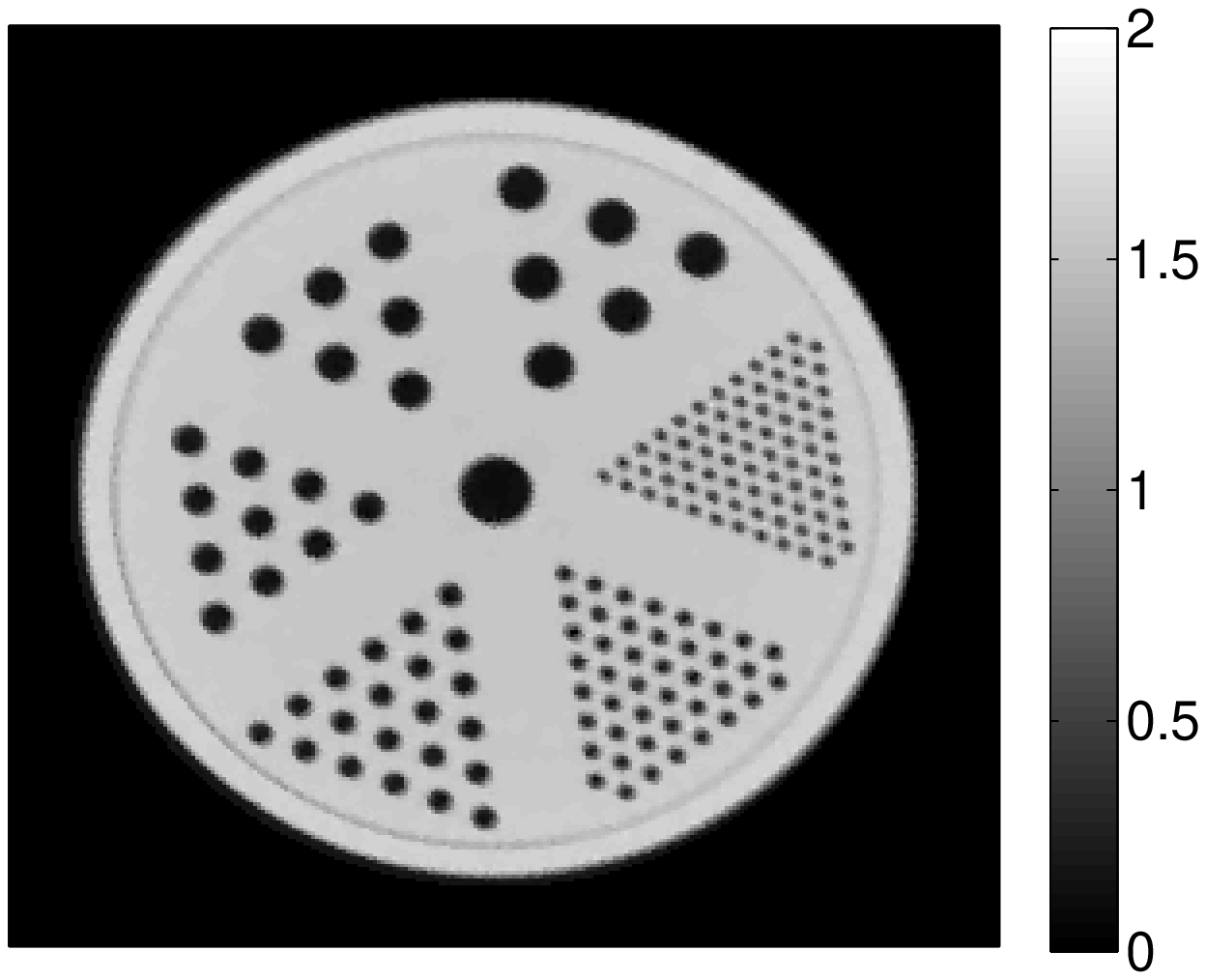} 
\label{fig_rods_VARD}} %
\\
\subfigure[MAP, $\beta=10^3$, $\delta=0.1$]{%
\includegraphics[width=6.3cm]{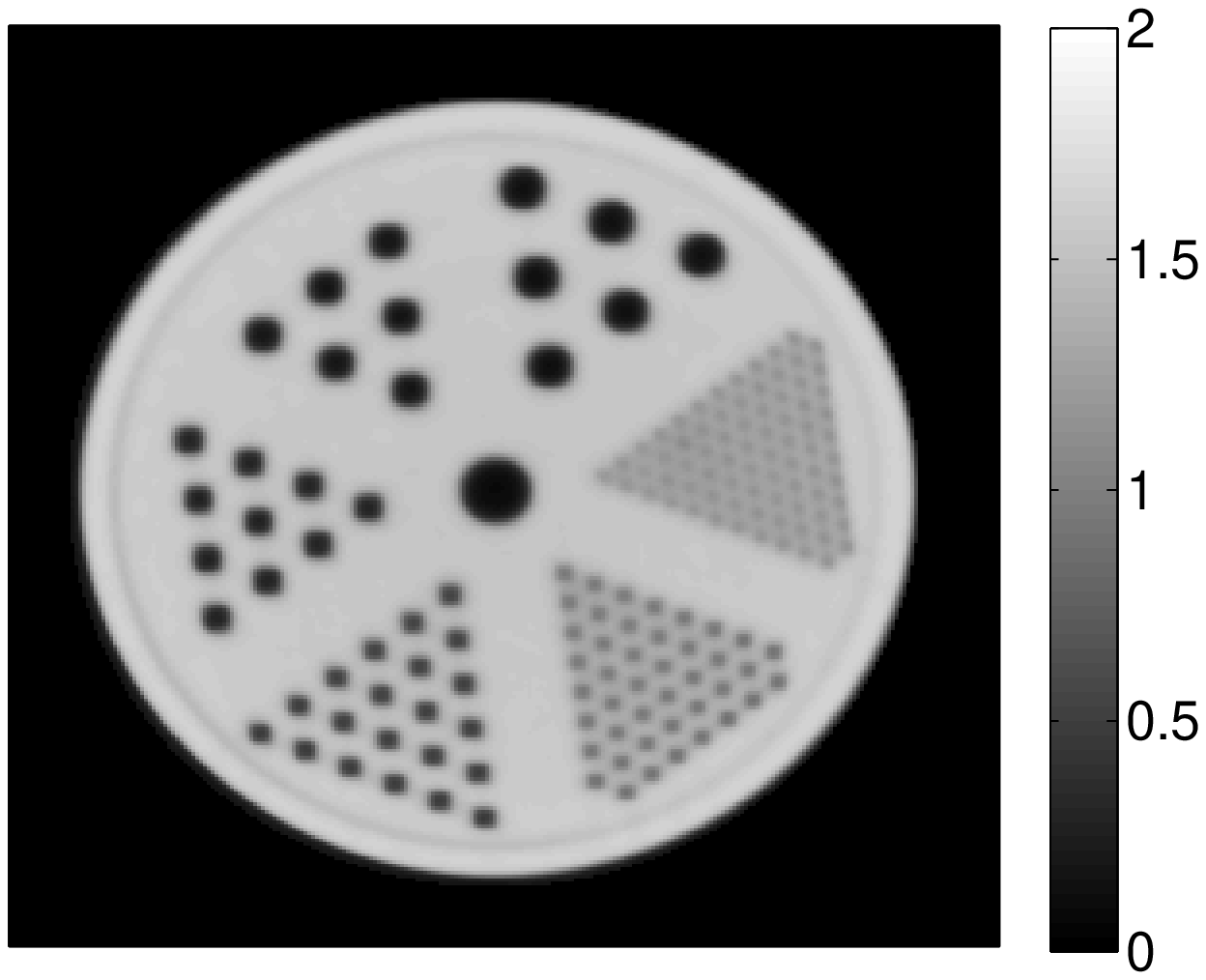} 
\label{fig_rods_map}} %
\subfigure[MAP, $\beta=10^2$, $\delta=0.1$]{%
\includegraphics[width=6.3cm]{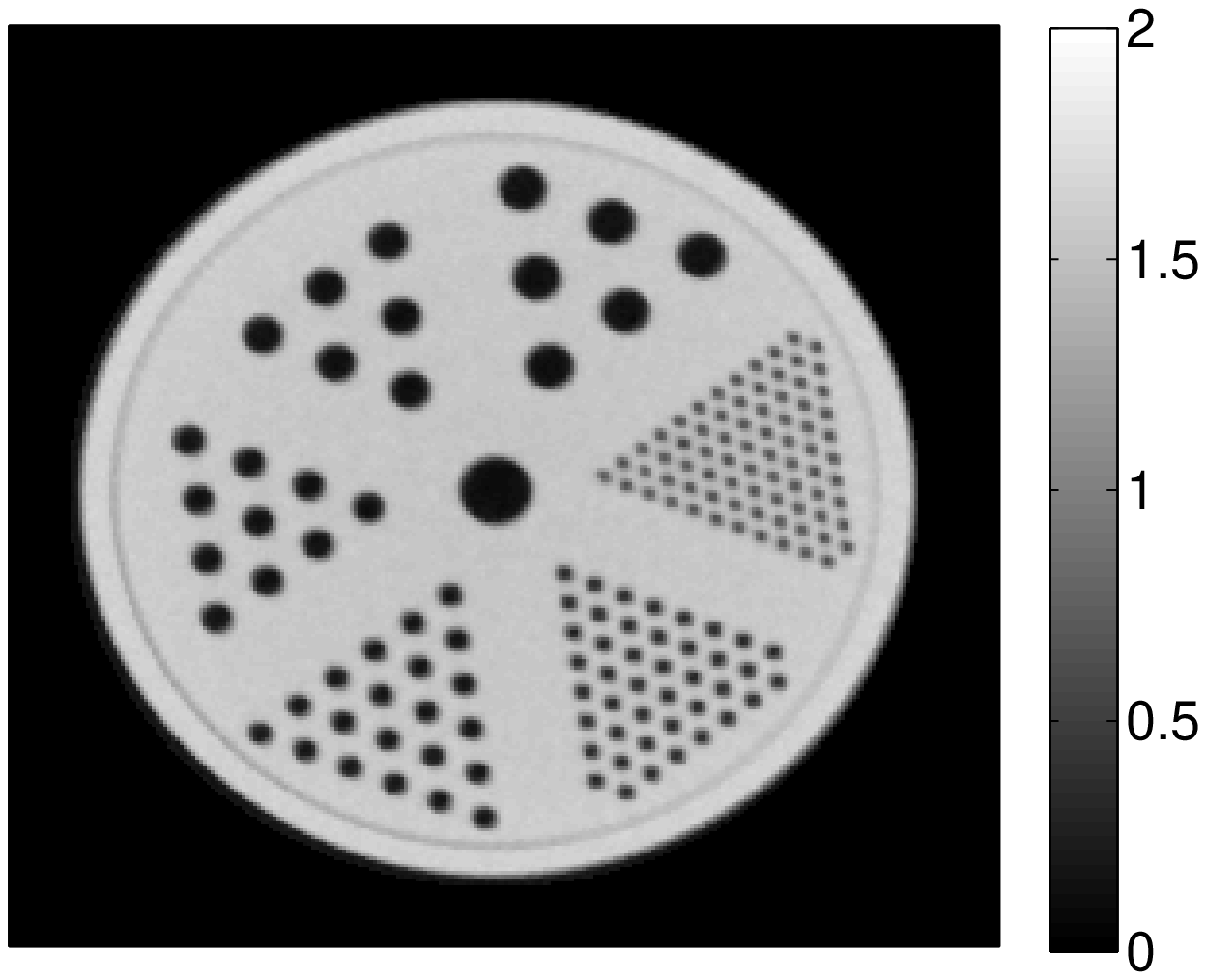} 
\label{fig_rods_map2}} %
\\
\subfigure[Re-$\ell_2$-O, $\epsilon=10^{-5}$]{%
\includegraphics[width=6.3cm]{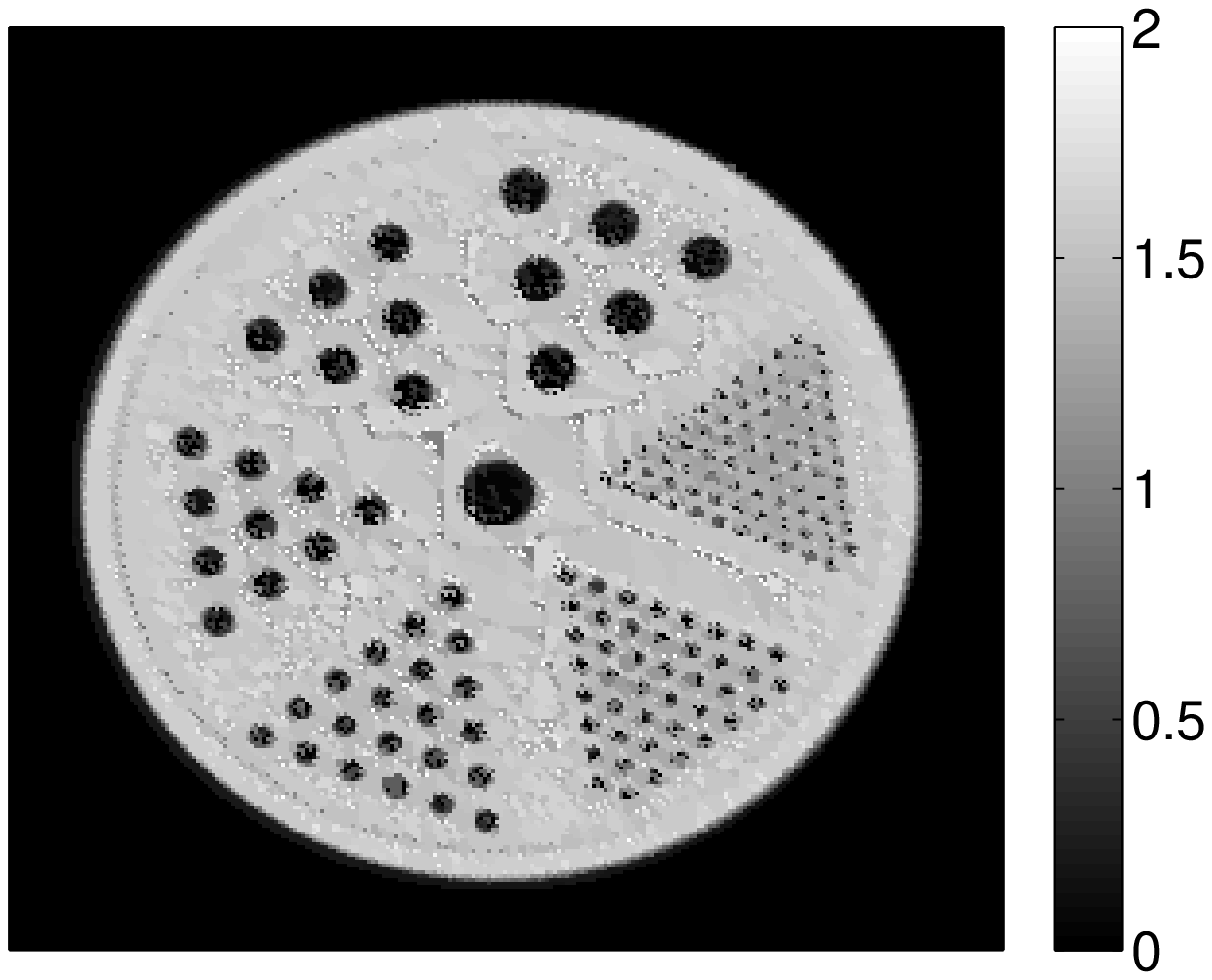}
\label{fig_rods_rel2_1}} %
\subfigure[Re-$\ell_2$-O, $\epsilon=10^{-3}$]{%
\includegraphics[width=6.3cm]{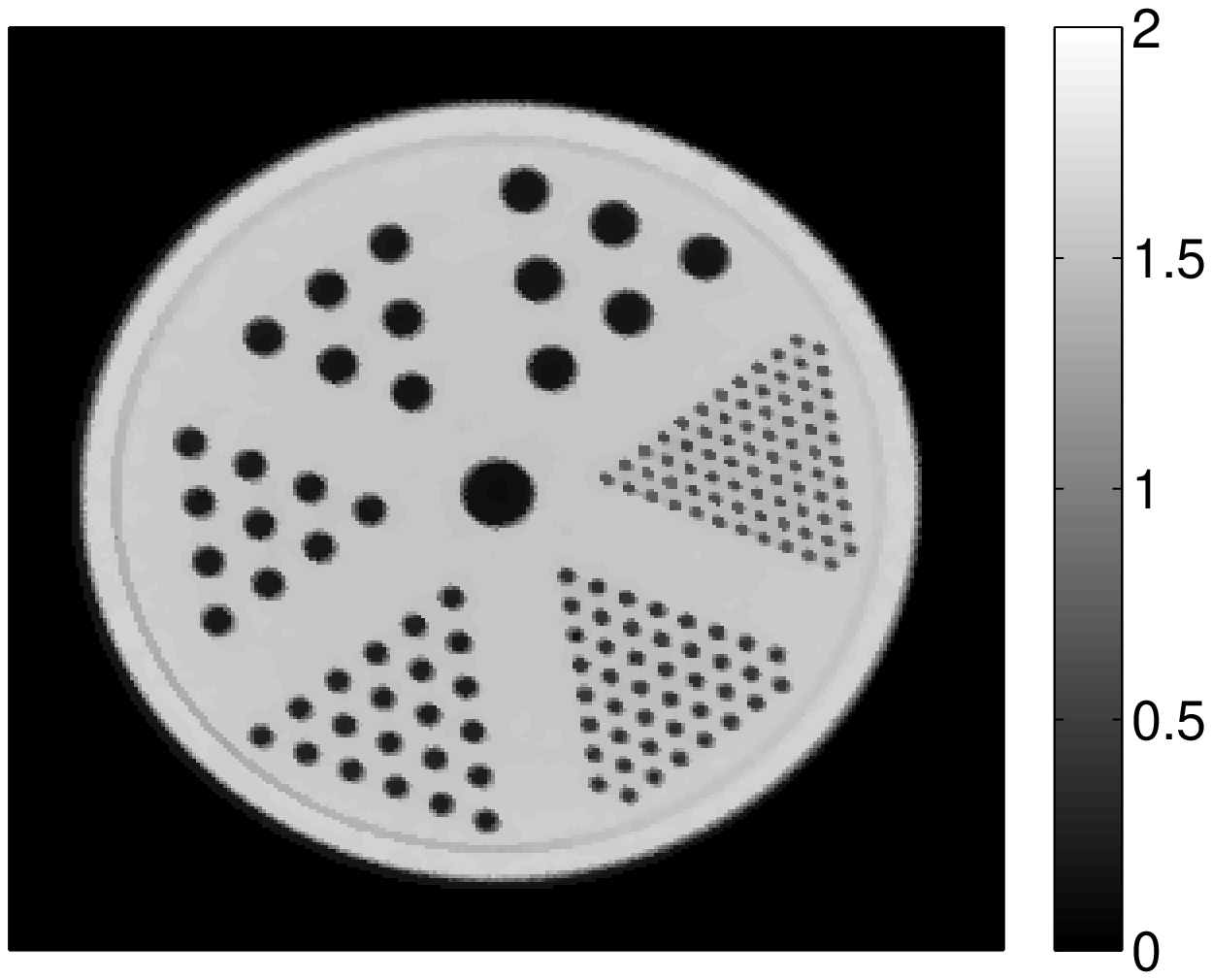} 
\label{fig_rods_rel2_2}} %
\caption{Reconstructions of attenuation maps (relative to water) from real data using different methods. (a) Maximum likelihood (MLE) estimate; (b) VARD-O; (c) Maximum a posteriori (MAP) estimate using the smooth edge-preserving neighborhood penalty with $\beta=10^3$ and $\delta=0.1$;  (d) Same as (c) but with $\beta=10^2$; (e) Reweighted $\ell_2$-O with $\epsilon=10^{-5}$;  (f) Same as (e) but with $\epsilon=10^{-3}$.  All algorithms use 2000 iterations to ensure convergence. All except MLE use a  neighborhood penalty with one horizontal and one vertical neighbor per pixel. We recommend viewing the figures with at least 150\% zoom-in. All objectives were monotonically decreasing and flat when the algorithms were terminated.}  
\label{fig_inv_rods}
\end{figure}

\noindent coefficient of $1.458$ relative to water (more details in caption of Fig.~\ref{fig_duke}). 
The expected image is sparse so we used the prior in \eqref{prior} with $\bd{\Psi}=\bd{I}$ to promote sparsity in the pixel basis. 

We subsampled the angular grid for the source-locations (views) so only $1/8$ or $1/16$ of the views are used;  the VARD reconstructions for these two cases are shown in Fig.~\ref{fig_Duke_mean_sub8} and Fig.~\ref{fig_Duke_mean_sub16}, respectively. In addition, we include in Fig.~\ref{fig_Duke_FBP_sub8_full} and ~\ref{fig_Duke_FBP_sub16_full} the reconstructions using filtered-back-projection (FBP) \cite{slaney1988principles}, 
which is a linear reconstruction method based on Fourier inversion. Note that the FBP reconstruction exhibits much more prominent artifacts than VARD in the form of streaks, which are due to view undersampling. We present the FBP reconstructions again in Figs~\ref{fig_Duke_FBP_sub8}--\ref{fig_Duke_FBP_sub16} where the displayed attenuation values were limited to a smaller range in order to enhance the artifacts. There is a striking resemblance between the streaks in the posterior std images obtained by VARD in Figs~\ref{fig_Duke_var_sub8}--\ref{fig_Duke_var_sub16} and the streaks in the FBP reconstructions in Figs~\ref{fig_Duke_FBP_sub8}--\ref{fig_Duke_FBP_sub16}. This demonstrates that the variances can indicate regions in the image which are insufficiently sampled. A more detailed study of this relationship will be deferred to a subsequent publication.   Lastly, we note that although we have focused on the 2D fan-beam geometry in the numerical experiments, the proposed method can be applied to any other geometry.

\begin{figure}[h]
\begin{center}
\subfigure[VARD prior std]{%
\includegraphics[width=6.3cm]{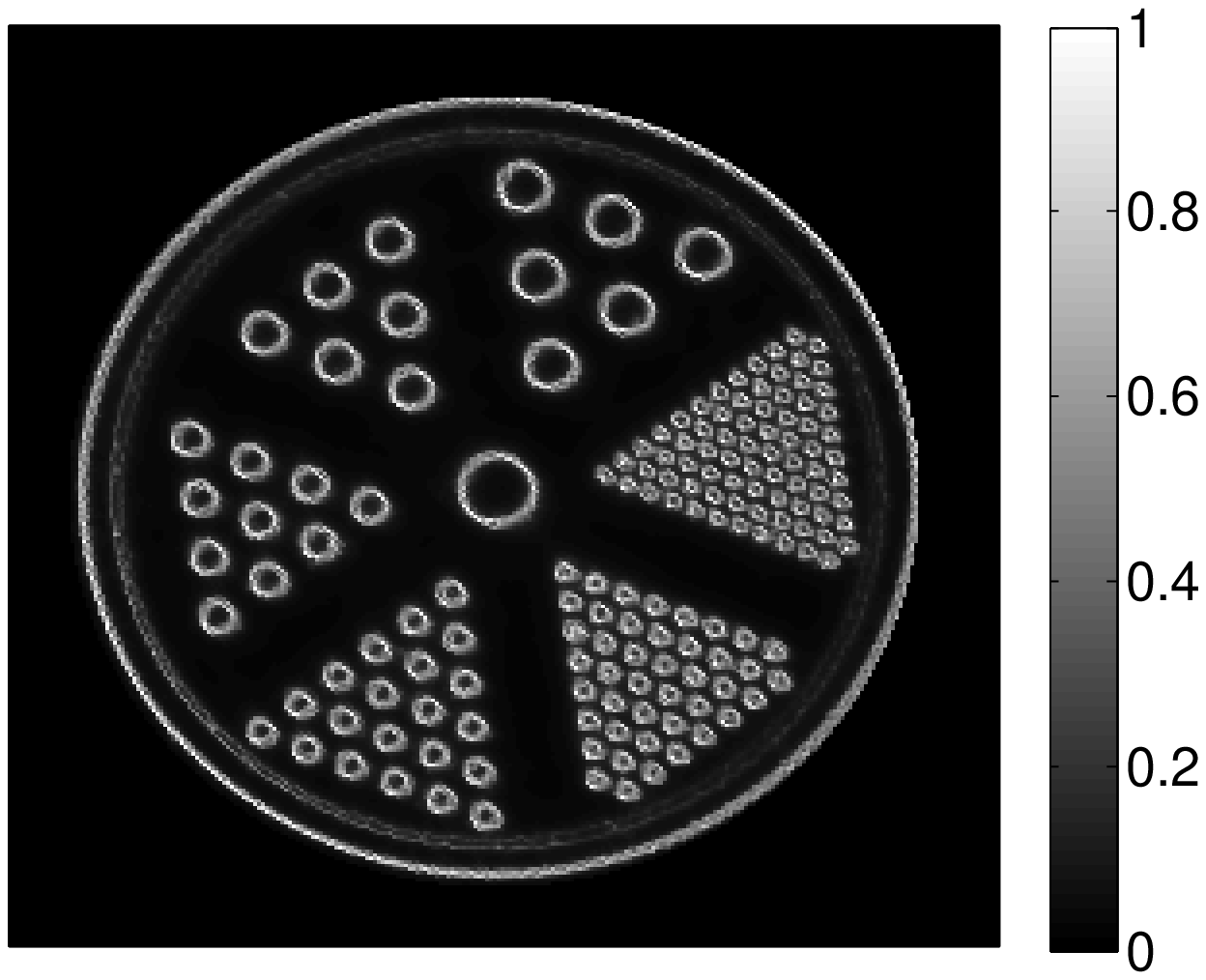}
\label{fig_rods_xsi}} %
\subfigure[VARD posterior std]{%
\includegraphics[width=6.3cm]{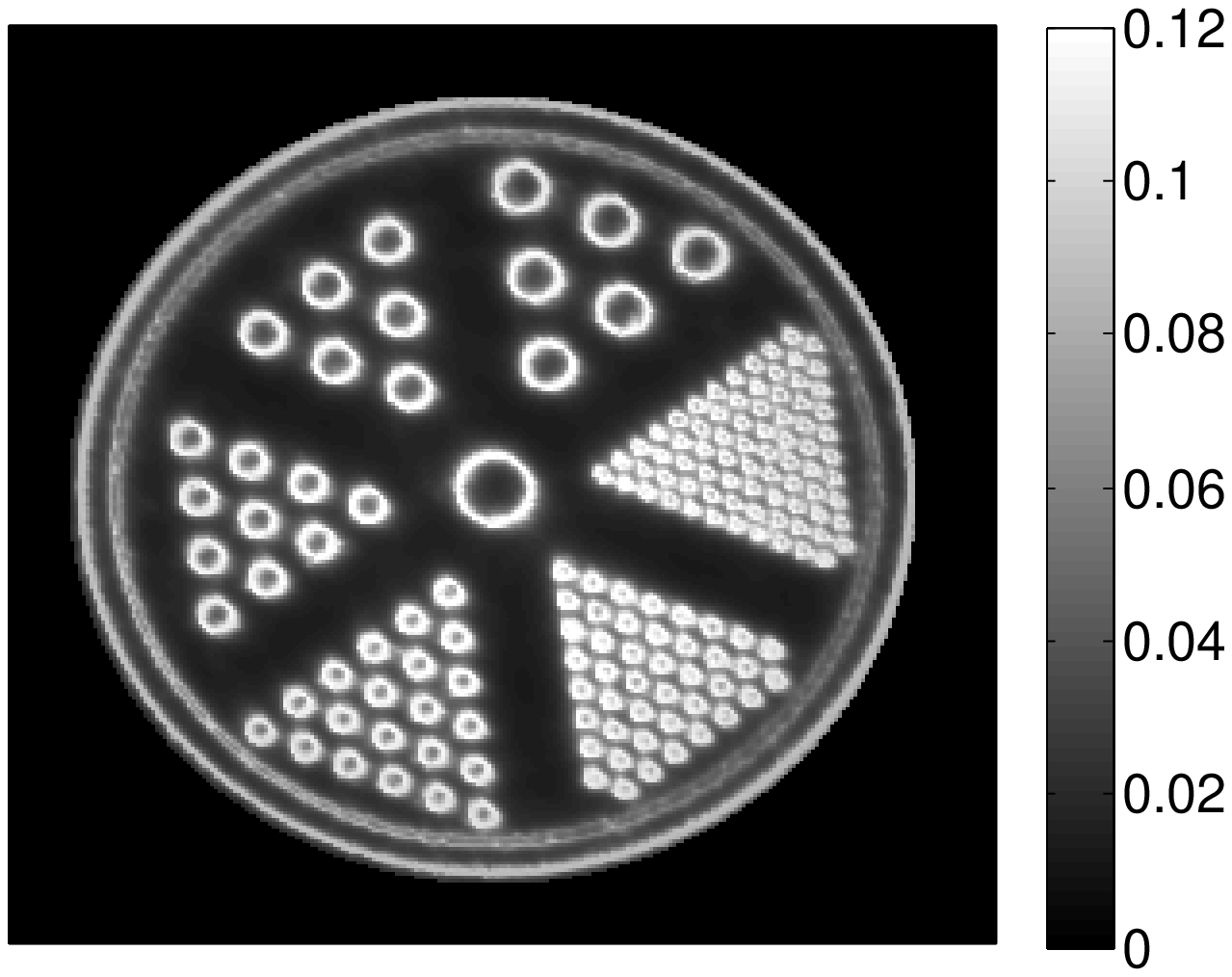} 
\label{fig_rods_var}} %
\end{center}
\caption{Reconstructed standard deviations (std) using VARD on real data. (a) squared root of hyperparameters $\bd{\gamma}$; (b) squared root of posterior variances $\bd{v}$.}
\label{fig_inv_rods_variance}
\end{figure}

\begin{figure}[t]
\centering
\subfigure[VARD - view sub. $\times 8$]{%
\includegraphics[width=6cm]{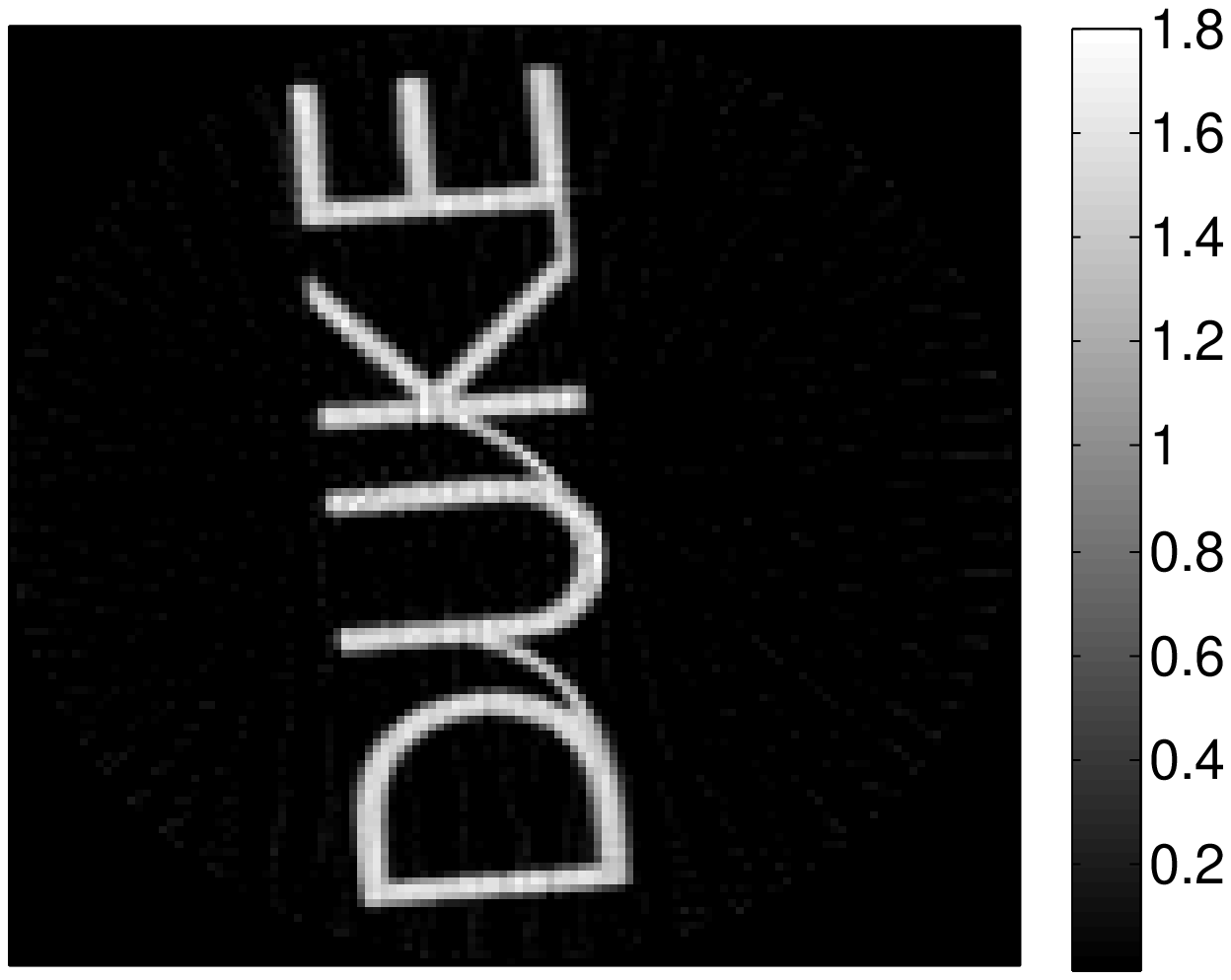}
\label{fig_Duke_mean_sub8}} %
\subfigure[VARD - view sub. $\times 16$]{%
\includegraphics[width=6cm]{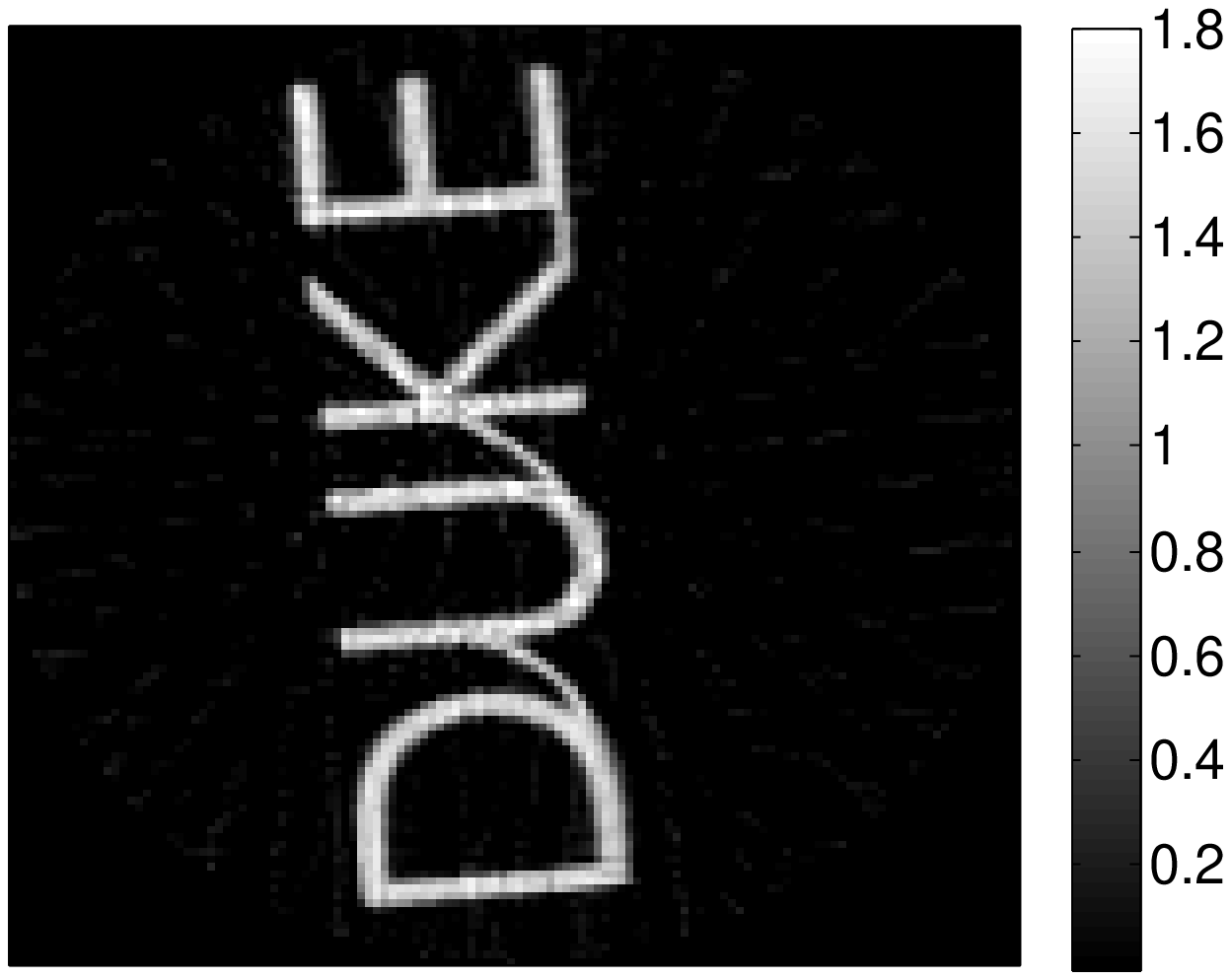} 
\label{fig_Duke_mean_sub16}} %
\\
\subfigure[FBP - view sub. $\times 8$]{%
\includegraphics[width=6cm]{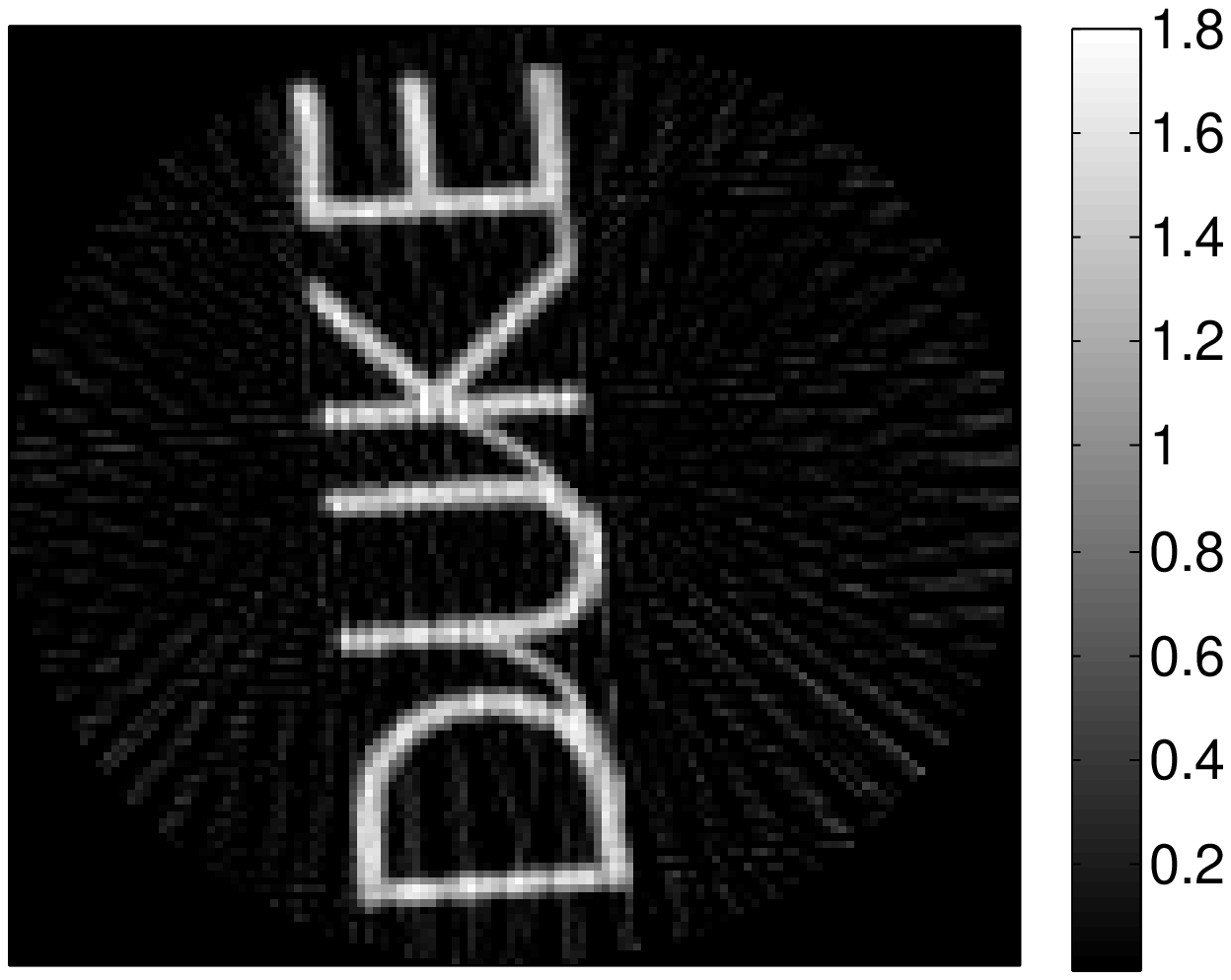}
\label{fig_Duke_FBP_sub8_full}} %
\subfigure[FBP - view sub. $\times 16$]{%
\includegraphics[width=6cm]{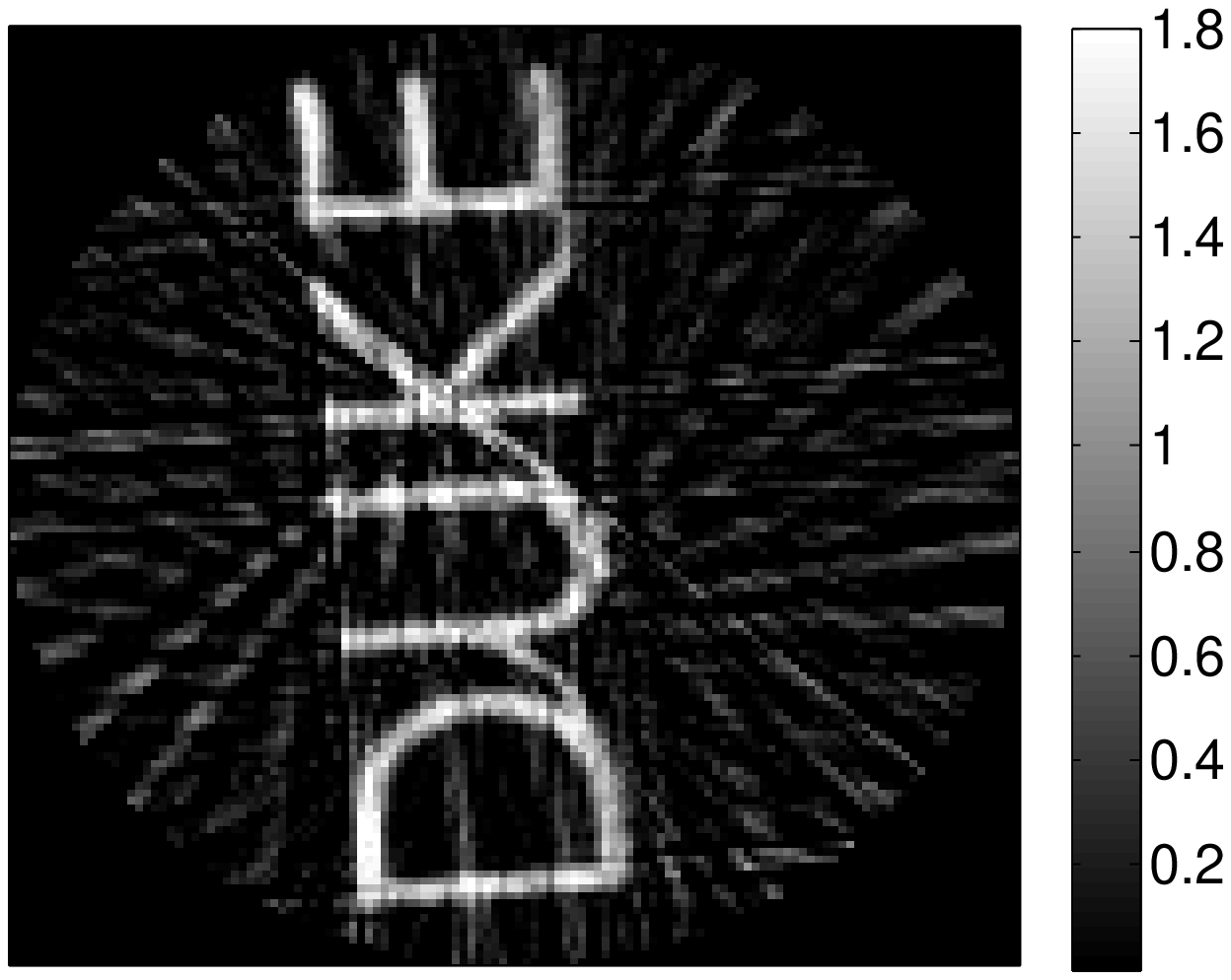}
\label{fig_Duke_FBP_sub16_full}} %
\caption{Reconstructions from real data. (a) VARD with sparsity promoted in the pixel basis and using $1/8$ of the views; (b) same as (a) but using $1/16$ of the views; (c) Filtered back projection using $1/8$ of the views; (d) same as in (c) but using $1/16$ of the views. For FBP, we used cubic interpolation for the backprojection, and a  cropped Ram-Lak (ramp) filter. Image resolution is $128\times 128$. The full dataset (before the view subsampling) was acquired using a full $360^\circ$ scan around the object with view-angle resolution of $1^\circ$ and $1573$ detector pixels per view. $510$ iterations were used for VARD.} 
\label{fig_duke} 
\end{figure}

\subsection{Implementation Details}
In our implementation, the 1D optimizations in lines~11--12 of Algorithm~1, line~9 of Algorithm~2 and Algorithm~3 are all done in parallel using MATLAB's vectorization. Since it is not necessary to solve the 1D surrogate problems exactly, we found that it is more efficient to use Newton's method and perform just one Newton step for each iteration.  However, for MAP this leads to divergence whenever the solution is in the linear regime of the smooth edge preserving penalty. Therefore, we use a 1D trust-region Newton method (see details in Appendix~\ref{apx_C}), which typically costs about $1-5$ Newton steps. For VARD we could use Newton steps for the mean updates, but had to resort to the trust-region method for the variance updates (see details in Appendix~\ref{apx_C}) to ensure decrease of the objective. For reweighted $\ell_2$ the Newton steps worked successfully. 
For the problem sizes considered here, we could store the system matrix and its transpose in memory using MATLAB's compressed sparse representation. The time for the forward/back-projections was negligible relative to the trust-region Newton updates due to MATLAB's highly optimized code for matrix-vector multiplications.  
However, this is certainty not a typical situation to be expected in large scale problems since it will not be possible to store the matrices in memory (even in sparse format) and ``on-the-fly'' implementations of the forward/back-projection operations would be required. Therefore, we do not report actual run time but instead discuss below the computational complexity per iteration for large scale problems. Details regarding the implementation of SBL can be found in Appendix~\ref{apx_B}.


\subsection{Computational Complexity} \label{sec_comp}
In analyzing the computational complexity per iteration for large scale problems it is common to consider only the forward/back-projections. It is also usually the case that these operations dominate the execution time per iteration in large scale problems.  
The MAP and reweighted $\ell_2$ algorithms both have the same computational complexity, as they both require one forward-projection and one back-projection per iteration.
To compare VARD to the other algorithms, first note that a variance-type operation (corresponding to an elementwise squared system matrix) requires double the number of multiplications relative to a mean-type operation when the projections are performed via ray-tracing and computing line intersections.  If the mean-type and variance-type operations can be done in parallel then VARD requires the same number of forward/backprojections as MAP but it will have double the complexity of MAP due to the $\times 2$ larger number of multiplications in the variance-type operations. 
Also, if the computational resources are limited and the mean-type and variance-type operations can only be performed sequentially, then  VARD will be at least 3 times more expensive per iteration than a single MAP or reweighted $\ell_2$ trial.
However, in cases when the tuning parameters are unknown or better ones are needed, MAP-based methods require multiple trials to determine the values of the tuning parameters, while VARD requires a single run. In these cases, VARD can still offer a lower \emph{total} run-time than the alternative methods and it may be more convenient since it does not rely on human expert evaluation of the different images for different values of tuning parameters. 

Each iteration of SBL involves the solution of several $p\times p$ matrix equations using the conjugate gradient (CG) method \cite{Saad} in order to estimate the posterior variances (see details in Appendix~\ref{apx_B}) with one additional equation for computing the mean. The solution to each matrix equation requires multiple CG sub-iterations, each sub-iteration requiring one forward projection and one back-projection. Denoting $K$ as the number of matrix equations for the variance estimation, and denoting $I_m$, $I_v$ as the number of CG sub-iterations for the mean and variance parts, respectively, each iteration of SBL involves $K\times I_v+I_m$ forward/back-projections. The choice for $K$ and the typical values of $I_m,I_v$ depend on the specific system and object of interest. 
For the example in Sec.~\ref{sec_sim} we used $K=64$, and the chosen residual thresholds (see Appendix~\ref{apx_B}) lead to $I_v\approx 10$ and $I_m \approx 100$, resulting in 2 orders of magnitude more forward and back-projections per iteration than VARD.
  
\begin{figure}[h]
\centering
\subfigure[FBP - view sub. $\times 8$]{%
\includegraphics[width=6.3cm]{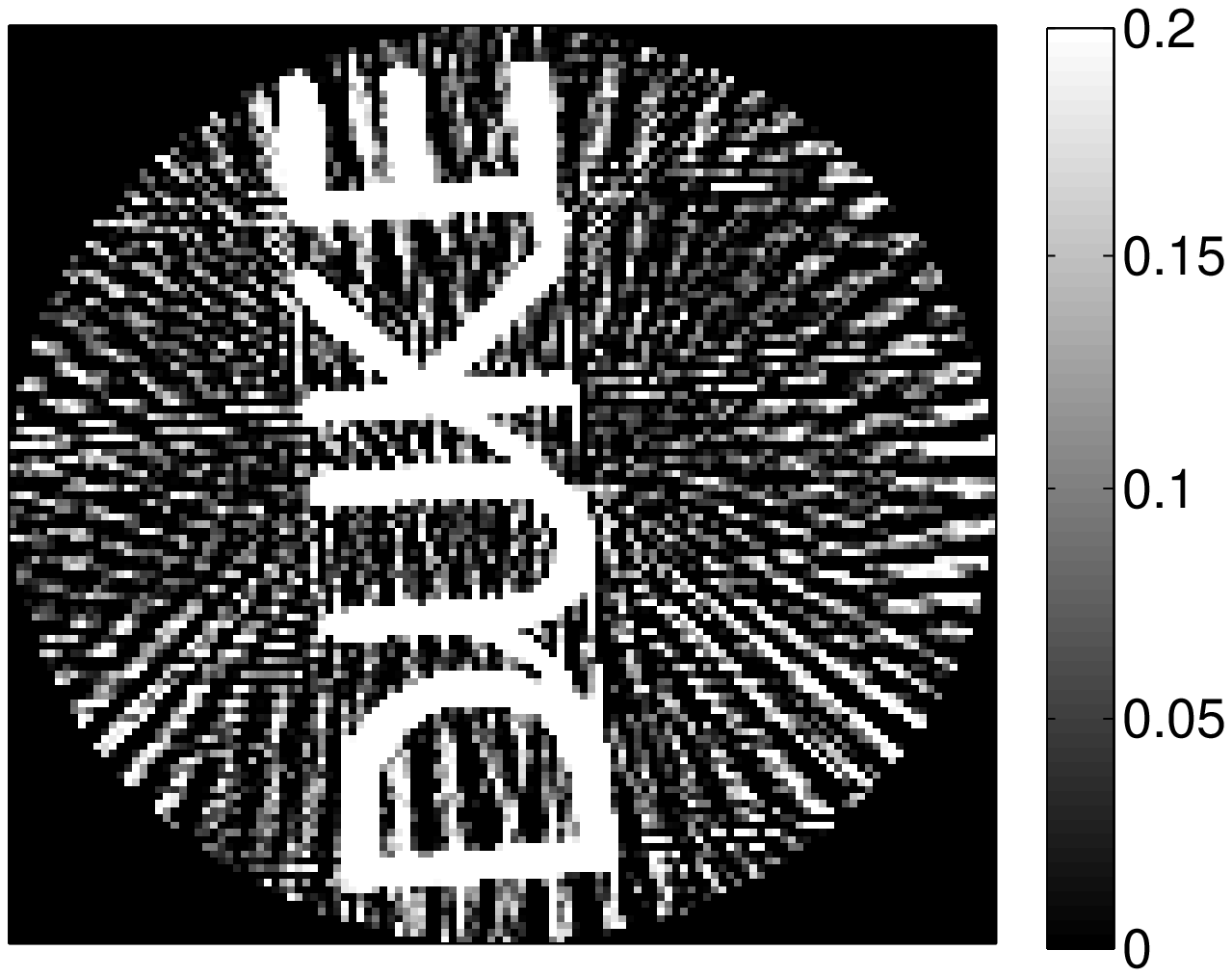}
\label{fig_Duke_FBP_sub8}} %
\subfigure[FBP - view sub. $\times 16$]{%
\includegraphics[width=6.3cm]{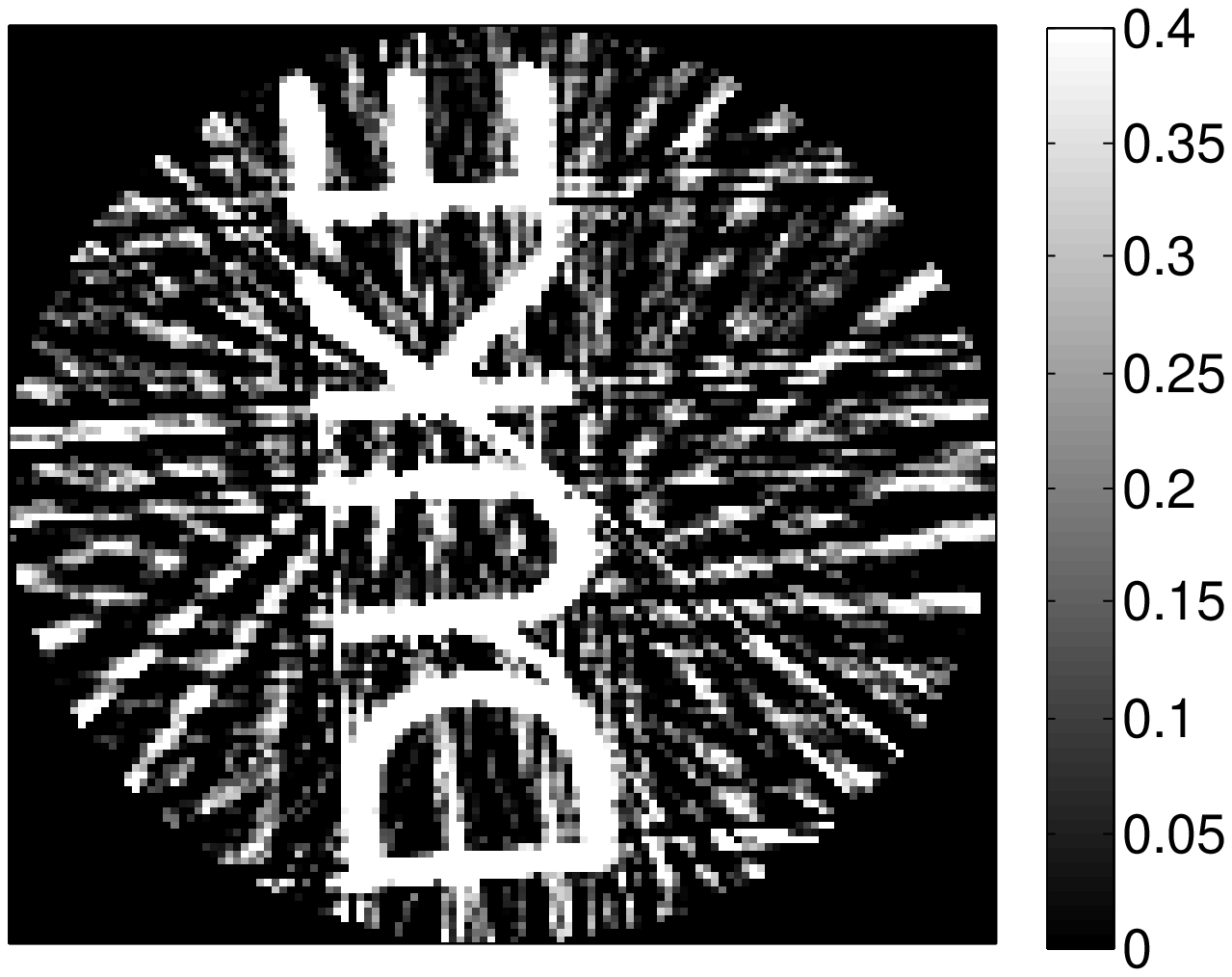}
\label{fig_Duke_FBP_sub16}} %
\\
\subfigure[VARD posterior std - view sub. $\times 8$]{%
\includegraphics[width=6.3cm]{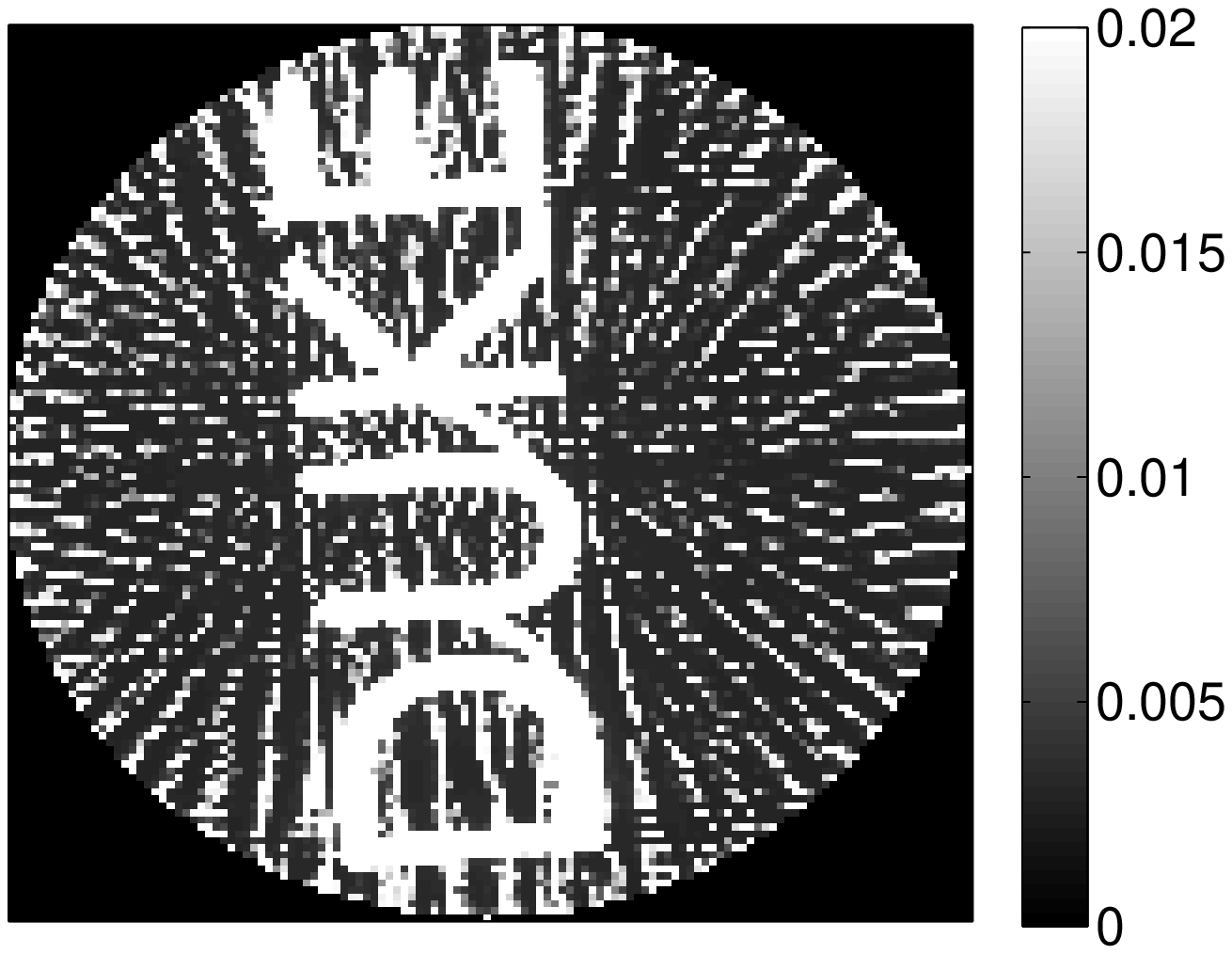} 
\label{fig_Duke_var_sub8}} %
\subfigure[VARD posterior std - view sub. $\times 16$]{%
\includegraphics[width=6.3cm]{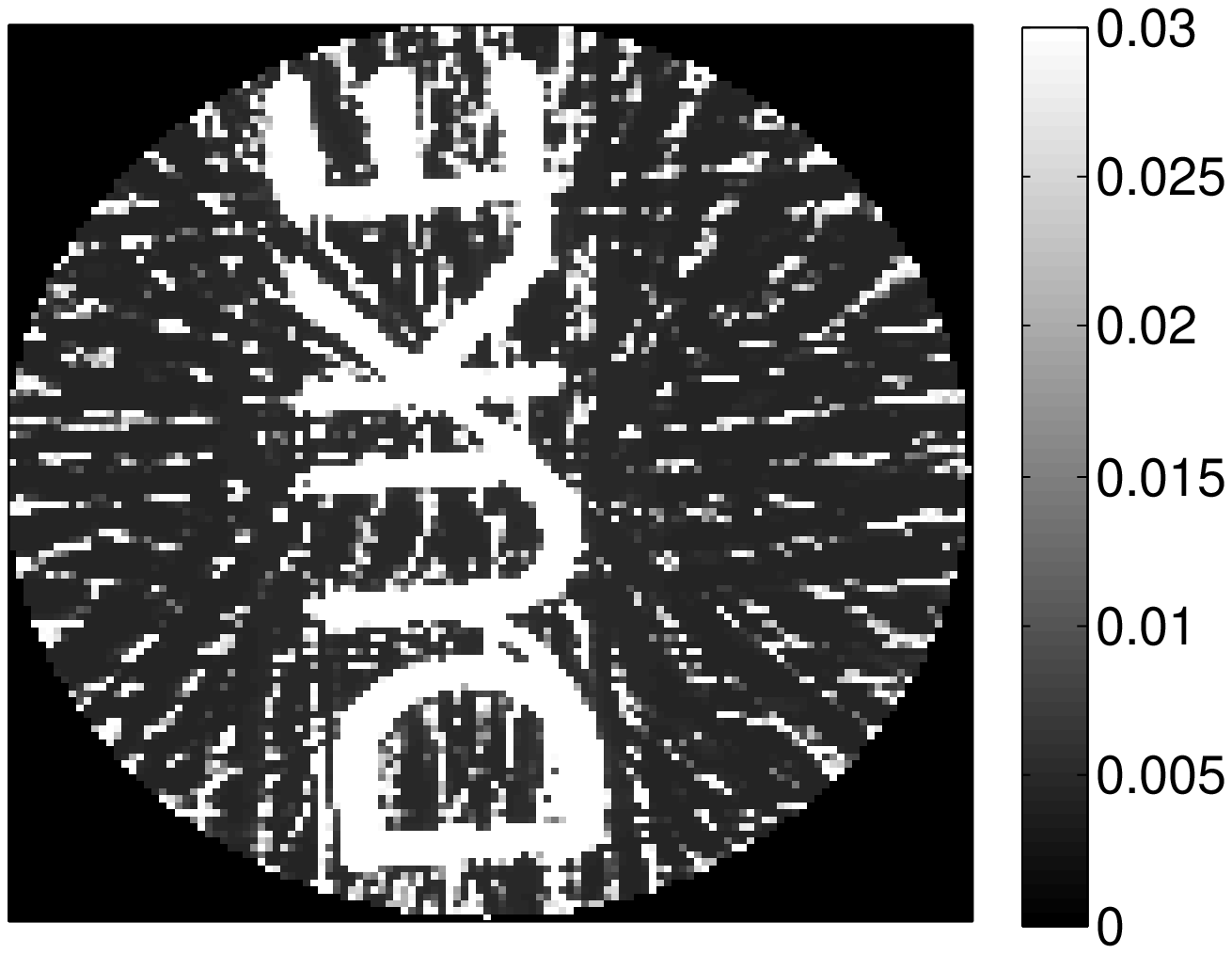} 
\label{fig_Duke_var_sub16}} %
\caption{Comparing artifacts in FBP to the posterior variances obtained via VARD. (a) Zoom-in view of Fig.~\ref{fig_Duke_FBP_sub8_full}; (b) Zoom-in view of Fig.~\ref{fig_Duke_FBP_sub16_full}; (c) VARD posterior std using $1/8$ of the  views; (d) same as (c) but using $1/16$ of the views.  Note the resemblance between the streaks in (a) and (c), and also (b) and (d).} 
\label{fig_duke2} 
\end{figure}

\section{Conclusions} \label{conc}
We presented a new framework for automatic relevance determination (ARD), which we call variational ARD (VARD). 
It provides an extension of previous ARD methods to Poisson noise models for transmission tomography and allows the use of neighborhood penalties that promote sparsity in the pixel/voxel difference domain. We revealed the sparsity-promoting mechanism in VARD, and established key differences compared to previous ARD methods that are applicable only to Gaussian noise models. In addition to computing the estimate for the imaged object, which is given by an approximation to the posterior mean, VARD also provides posterior variances, i.e., Bayesian confidence intervals.  
An important advantage of the VARD framework is that \emph{it avoids the cumbersome process of choosing tuning parameters, while still retaining good image quality}. We have shown that the posterior variances play a key role in learning the correct balance between the datafit and the prior. Straightforward simplifications that avoid using posterior variances, such as reweighted $\ell_2$, do not retain this auto-tuning property.

We presented a parallel AM algorithm for VARD (PAR-VARD) which is applicable to large scale problems. We studied a parallel implementation of VARD using both synthetic and real data. The numerical experiments demonstrate that the images produced by VARD are comparable in quality to the best maximum a posteriori (MAP) and reweighted $\ell_2$ estimates. However, VARD needs to be run only once, whereas MAP and reweighted $\ell_2$ need to be run multiple times whenever the tuning parameters are unknown or better ones are needed.
For MAP-based algorithms, it is possible to speed-up the search for tuning parameters by employing a continuation approach where the algorithm is initialized using the solution from a previous point in parameter space.  A more principled approach to determine the tuning parameters in MAP-based algorithms is cross-validation, which requires multiple runs for different subsets of the data and becomes computationally involved when the number of tuning parameters increases.

The unique features of VARD such as auto-tuning and computation of variances do not come without a price. We have shown that VARD has at least double computational complexity than a MAP-based algorithm for large scale problems. For commercial CT systems which are based on protocols with preset values for the tuning parameters, the auto-tuning feature of VARD might not be needed and would entail a longer reconstruction time.  However, we expect our method to be useful in non-standard situations when the preset values for the tuning parameters are not satisfactory and better values need to be found. 
In addition, auto-tuning can greatly simplify research when investigating non-standard geometries for which the tuning parameters might not be known. 
Exploring ways to reduce the computational burden of VARD so that is compares more favorably to MAP-based algorithms seems like an important direction for future work.

Nevertheless, we have considerably advanced the computational feasibility of ARD for large scale problems in this work. 
In the EM algorithm for ARD  \cite{tipping2001sparse} with a Gaussian noise model, the straightforward computation of the posterior variances scales cubically with the number of pixels/voxels, which is infeasible for large scale problems, so approximations are needed. The scalable ARD algorithm in \cite{Danniel} (SBL) which uses sophisticated approximations to the variances does not have convergence guarantees.  In contrast, the proposed VARD algorithm is feasibly for large-scale problems and also has \emph{guarantees for global convergence}. By comparing VARD to SBL,  we have shown that for low photon flux (high noise levels) VARD leads to better image quality than SBL, when the data is generated from the considered Poisson model. Furthermore, this is also \emph{true for high photon flux} (low noise levels), when the Gaussian approximation to the Poisson distribution is good enough, and seems to be due to the over-complete prior that can be used by VARD but cannot be used by SBL. We have shown that VARD is much more efficient computationally than SBL, requiring considerably less forward/back-projections (the main computational bottleneck), and VARD is also much simpler to implement. As opposed to previous ARD methods \cite{tipping2001sparse,Danniel}, the exact evaluation of the VARD objective is feasible for large-scale problems and can be used to assess convergence in practice and to verify the correctness of the implementation.  


We have shown that when the object is undersampled, the posterior variances computed by VARD can indicate regions in the image where sampling is insufficient. The variances can therefore be used to define relative levels of confidence in the reconstruction at particular regions in the image. 
The variances can also be used for adaptive sensing/experimental design \cite{Danniel,seeger2011large}. Our results suggest that the variances can be used for image segmentation or boundary detection. These potential uses are not studied here, and will be the subject of future work. 

Finally, we note that many algorithms for transmission tomography use ordered subsets (OS) to accelerate convergences while sacrificing guarantees for convergence \cite{sonka2000handbook,Erdogan,Keesing,Soysal}.  In this paper we have excluded the use of ordered subsets and focused on convergent algorithms in order to be consistent with the convergence theory presented in the paper. Initial experiments with OS for VARD appear to be promising but will be reported in a future publication.




\section{Acknowledgments} 
We would like to thank Prof. Martin P. Tornai for providing us access to the Duke Multi-Modality Imaging Laboratory (MMIL) where the experimental data was acquired. We would like to thank Mr. Andrew Holmgren for conducting the experiments. 
We thank Dr. Ikenna Odinaka for suggesting the 1D trust-region  Newton method used in our code. 
We thank Dr. David Carlson for proofreading the manuscript and providing helpful comments.    
This work was supported by the Department of Homeland Security, Science and Technology Directorate, Explosives Division, through contract HSHQDC-11-C-0083.

\newpage
\Appendix

\appendix
\section{Proof of Theorem~\ref{theorem_dec}} \label{apx_A}
\begin{proof}
We split the difference in \eqref{B_diff} into two terms
\begin{align}
&\mathcal{F}(\bd{m}^{(t)},\bd{v}^{(t)},\bd{\gamma}^{(t)})-\mathcal{F}(\bd{m}^{(t+1)},\bd{v}^{(t+1)},\bd{\gamma}^{(t+1)})  =\mathcal{F}(\bd{m}^{(t)},\bd{v}^{(t)},\bd{\gamma}^{(t)})-\mathcal{F}(\bd{m}^{(t+1)},\bd{v}^{(t+1)},\bd{\gamma}^{(t)}) \nonumber \\
&+\mathcal{F}(\bd{m}^{(t+1)},\bd{v}^{(t+1)},\bd{\gamma}^{(t)})-\mathcal{F}(\bd{m}^{(t+1)},\bd{v}^{(t+1)},\bd{\gamma}^{(t+1)}). \label{F_diff}
\end{align}
First we focus on the first difference in \eqref{F_diff} and write it explicitly using \eqref{mu}-\eqref{b}
\begin{align}\label{B_diff2}
&\mathcal{F}(\bd{m}^{(t)},\bd{v}^{(t)},\bd{\gamma}^{(t)})-\mathcal{F}(\bd{m}^{(t+1)},\bd{v}^{(t+1)},\bd{\gamma}^{(t)}) = \sum_j b^y_j(m^{(t)}_j-m^{(t+1)}_j)+\sum_i\mu^{(t)}_i- \nonumber \\
& \sum_i\mu^{(t)}_i\exp{\bigg[-\sum_j \phi_{ij} (m^{(t+1)}_j-m^{(t)}_j) +\sum_j \phi_{ij}^2(v^{(t+1)}_j-v^{(t)}_{j})/2 \bigg]}+\nonumber \\
&\sum_k\frac{1}{2\gamma^{(t)}_k}\big[(\sum_j\psi_{kj}m^{(t)}_j)^2-(\sum_j\psi_{kj}m^{(t+1)}_j)^2\big]+\nonumber \\
& 
\sum_k\frac{1}{2\gamma^{(t)}_k}\big[\sum_j\psi^2_{kj}(v^{(t)}_j-v^{(t+1)}_j)\big]+\sum_j \log(v^{(t+1)}_j/v^{(t)}_j)/2. 
\end{align}

Next, we present several bounds for terms in \eqref{B_diff2}. Using \eqref{sum_qr}--\eqref{qr} and \eqref{b} we obtain
\begin{align}
\sum_i\mu^{(t)}_i \geq \sum_i\mu^{(t)}_i \sum_j\frac{\phi_{ij}+\phi^2_{ij}/2}{Z_1} = \sum_j \frac{b^{(t)}_j}{Z_1}+\frac{\tilde{b}^{(t)}_j}{Z_1}   \geq \sum_j \frac{b^{(t)}_j}{Z^{(t)}_j}+\frac{\tilde{b}_j^{(t)}}{Z_1},  
\end{align}
where the last inequality is based on the definition for $Z^{(t)}_j$ in \eqref{Z_tilde_m}.
We bound the exponential term in \eqref{B_diff2} using the convex decomposition lemma (lemma~\ref{lemma}) 
\begin{align}
&-\sum_i\mu^{(t)}_i\exp{\big(-\sum_j \phi_{ij} (m^{(t+1)}_j-m^{(t)}_j)+\sum_j \phi_{ij}^2(v^{(t+1)}_{j}-v^{(t)}_j)/2 \big)} \geq  \nonumber \\[-2ex] 
&\qquad -\sum_j \frac{b^{(t)}_j}{Z^{(t)}_j} \exp\big(-Z^{(t)}_j(m^{(t+1)}_j-m^{(t)}_j)\big)-\sum_j \frac{\tilde{b}^{(t)}_j}{Z_1} \exp\big(Z_1(v^{(t+1)}_j-v^{(t)}_j)\big), \label{bound_exp} 
\end{align}
where we used  \eqref{qr} with $Z_1$ replaced by $Z^{(t)}_j \geq  Z_1$ in $r_{ij}$ (lemma~\ref{lemma} with the dummy variable extension applied), and also used \eqref{b}. 
Using lemma~\ref{lemma} for the quadratic term in \eqref{B_diff2} we obtain
\begin{align}
&\sum_k\frac{1}{2\gamma^{(t)}_k}\big[(\sum_j\psi_{kj}m^{(t)}_j)^2-(\sum_j\psi_{kj}m^{(t+1)}_j)^2\big] \geq  -\sum_j \big[ f^{(t)}_j\,(m^{(t+1)}_j-m^{(t)}_j)- g^{(t)}_j\,(m^{(t+1)}_j-m^{(t)})^2\big], \label{bound_m} 
\end{align}
where the derivation is similar to the one in \eqref{S2_temp} and we also used \eqref{d_f}--\eqref{g_h}. 
Substituting \eqref{m_update}, \eqref{v_update} into the right hand sides of \eqref{bound_exp} and \eqref{bound_m} and then substituting the result into \eqref{B_diff2}, we obtain
\begin{align}
&\mathcal{F}(\bd{m}^{(t)},\bd{v}^{(t)},\bd{\gamma}^{(t)})-\mathcal{F}(\bd{m}^{(t+1)},\bd{v}^{(t+1)},\bd{\gamma}^{(t)}) \geq \nonumber \\
&   \sum_j \bigg \{ \frac{b^y_j}{Z^{(t)}_j} \log\big (\frac{\beta^{(t)}_j}{b^{(t)}_j} \big)+\frac{b^{(t)}_j}{Z^{(t)}_j}+\frac{\tilde{b}^{(t)}_j}{Z_1}-\frac{\beta^{(t)}_j}{Z^{(t)}_j}-\frac{\tilde{\beta}^{(t)}_j}{Z_1}+\frac{f^{(t)}_j}{Z^{(t)}_j} \log\big [\frac{\beta^{(t)}_j}{b^{(t)}_j} \big]  +\nonumber \\
& \frac{g^{(t)}_j}{Z^{(t)}_j} \log\big [\frac{\beta^{(t)}_j}{b^{(t)}_j} \big](m^{(t+1)}_j-m^{(t)}_j)-\frac{\xi^{(t)}_j}{2Z_1} \log\big [\frac{\tilde{\beta}^{(t)}_j}{\tilde{b}^{(t)}_j} \big]+\log(v^{(t+1)}_j/v^{(t)}_j)/2 \bigg \}. \label{B_diff_temp}
\end{align}
We also use 
\begin{align}
&(m^{(t+1)}_j-m^{(t)}_j)=2(\bar{m}^{(t)}_j-m^{(t)}_j)+(m^{(t+1)}_j-m^{(t)}_j+2(m^{(t)}_j-\bar{m}^{(t)}_j) )= \nonumber \\[-0.5ex]
&2(\bar{m}^{(t)}_j-m^{(t)}_j)+(\frac{2}{Z_1}-\frac{1}{Z^{(t)}_j})\log\big (\frac{\beta^{(t)}_j}{b^{(t)}_j}  \big) 
\geq 2(\bar{m}^{(t)}_j-m^{(t)}_j)+\frac{1}{Z^{(t)}_j}\log\big (\frac{\beta^{(t)}_j}{b^{(t)}_j} \big ), \label{m_diff}
\end{align}
where we have used \eqref{m_KKT} and \eqref{m_update}. Next, we use \eqref{beta_m} in \eqref{m_diff} to obtain
\begin{align}
(m^{(t+1)}_j-m^{(t)}_j) \geq (\beta^{(t)}_j -b^y_j-f^{(t)}_j)/g^{(t)}_j+\frac{1}{Z^{(t)}_j}\log\big (\frac{\beta^{(t)}_j}{b^{(t)}_j} \big )
\end{align}
and substituting into \eqref{B_diff_temp} we have
\begin{align}
&\mathcal{F}(\bd{m}^{(t)},\bd{v}^{(t)},\bd{\gamma}^{(t)})-\mathcal{F}(\bd{m}^{(t+1)},\bd{v}^{(t+1)},\bd{\gamma}^{(t)}) \geq  \frac{b^{(t)}_j}{Z^{(t)}_j}+\frac{\tilde{b}^{(t)}_j}{Z_1}-\frac{\beta^{(t)}_j}{Z^{(t)}_j}-\frac{\tilde{\beta}^{(t)}_j}{Z_1} +\nonumber \\
&g^{(t)}_j\bigg[ \frac{1}{Z^{(t)}_j}\log\big (\frac{\beta^{(t)}_j}{b^{(t)}_j} \big ) \bigg] ^2+ \frac{\beta^{(t)}_j}{Z^{(t)}_j} \log\big (\frac{\beta^{(t)}_j}{b^{(t)}_j} \big)-
\frac{\xi^{(t)}_j}{2Z_1} \log\big (\frac{\tilde{\beta}^{(t)}_j}{\tilde{b}^{(t)}_j} \big) +\frac{1}{2}\log\big (\frac{v^{(t+1)}_j}{v^{(t)}_j}\big) . \label{B_diff_temp2}
\end{align}

Finally, we use $\log(x)\geq 1 -1/x$ applied to the last term in \eqref{B_diff_temp2}, which gives
\begin{align}
\log(v^{(t+1)}_j/v^{(t)}_j) \geq \frac{1}{v^{(t+1)}_j} (v^{(t+1)}_j-v^{(t)}_j) =\frac{1}{v^{(t+1)}_j}\frac{1}{Z_1} \log\big (\frac{\tilde{\beta}^{(t)}_j}{\tilde{b}^{(t)}_j} \big), \label{bound_log_v}
\end{align}
where we used \eqref{v_update}. Substituting \eqref{bound_log_v} into \eqref{B_diff_temp2}, using \eqref{beta_v} and the definition of the I-divergence in \eqref{I_divergence}, we obtain 
\begin{align}
\mathcal{F}(\bd{m}^{(t)},\bd{v}^{(t)},\bd{\gamma}^{(t)})-\mathcal{F}(\bd{m}^{(t+1)},\bd{v}^{(t+1)},\bd{\gamma}^{(t)}) \geq &I\bigg (\frac{\bd{\beta}^{(t)}}{\bd{Z}^{(t)}}\big|\big|\frac{\bd{b}^{(t)}}{\bd{Z}^{(t)}}\bigg)+I\bigg (\frac{\tilde{\bd{\beta}}^{(t)}}{Z_1}\big|\big|\frac{\tilde{\bd{b}}^{(t)}}{Z_1}\bigg). \label{bound1}
\end{align}
Next, we calculate the second difference in \eqref{F_diff}
\begin{align}
&\mathcal{F}(\bd{m}^{(t+1)},\bd{v}^{(t+1)},\bd{\gamma}^{(t)})-\mathcal{F}(\bd{m}^{(t+1)},\bd{v}^{(t+1)},\bd{\gamma}^{(t+1)}) = \nonumber \\
&\frac{1}{2}\sum_k\big(\frac{1}{\gamma^{(t)}_k}-\frac{1}{\gamma^{(t+1)}_k}\big)\bigg [(\sum_j \psi_{kj} m^{(t+1)}_j )^2+\sum_j \psi^2_{kj}v^{(t+1)}_j     \bigg]
+\frac{1}{2}\sum_k\log\big(\frac{\gamma^{(t)}_k}{\gamma^{(t+1)}_k}\big). \label{F2_temp}
\end{align}
Substituting \eqref{sol_forward} into \eqref{F2_temp} we obtain
\begin{align}
&\mathcal{F}(\bd{m}^{(t+1)},\bd{v}^{(t+1)},\bd{\gamma}^{(t)})-\mathcal{F}(\bd{m}^{(t+1)},\bd{v}^{(t+1)},\bd{\gamma}^{(t+1)}) = \nonumber \\
&\frac{1}{2}\sum_k\log\big(\frac{\gamma^{(t)}_k}{\gamma^{(t+1)}_k}\big)+\frac{1}{2}\big(\frac{\gamma^{(t+1)}_k}{\gamma^{(t)}_k}-1\big)  \triangleq \frac{1}2{}D_{IS}(\bd{\gamma}^{(t)}||\bd{\gamma}^{(t+1)}), \label{bound2}
\end{align}
where we used the definition in \eqref{IS_divergence}. Combining the results in \eqref{bound1} and \eqref{bound2} we have
\begin{align}
&\mathcal{F}(\bd{m}^{(t)},\bd{v}^{(t)},\bd{\gamma}^{(t)})-\mathcal{F}(\bd{m}^{(t+1)},\bd{v}^{(t+1)},\bd{\gamma}^{(t+1)})  =\mathcal{F}(\bd{m}^{(t)},\bd{v}^{(t)},\bd{\gamma}^{(t)})-\mathcal{F}(\bd{m}^{(t+1)},\bd{v}^{(t+1)},\bd{\gamma}^{(t)})+ \nonumber \\
&\mathcal{F}(\bd{m}^{(t+1)},\bd{v}^{(t+1)},\bd{\gamma}^{(t)})-\mathcal{F}(\bd{m}^{(t+1)},\bd{v}^{(t+1)},\bd{\gamma}^{(t+1)}) \geq I\bigg (\frac{\bd{\beta}^{(t)}}{\bd{Z}^{(t)}}\big|\big|\frac{\bd{b}^{(t)}}{\bd{Z}^{(t)}}\bigg)+I\bigg (\frac{\tilde{\bd{\beta}}^{(t)}}{Z_1}\big|\big|\frac{\tilde{\bd{b}}^{(t)}}{Z_1}\bigg)+ \nonumber \\
&\sum_j g^{(t)}_j\bigg[ \frac{1}{Z^{(t)}_j}\log\big (\frac{\beta^{(t)}_j}{b^{(t)}_j} \big ) \bigg] ^2+\frac{1}{2}D_{IS}(\bd{\gamma}^{(t)}||\bd{\gamma}^{(t+1)})
\end{align}
which is the result stated in Theorem~\ref{theorem_dec}.
\end{proof}

\section{Implementation Details} \label{apx_C}
Next, we address the implementation of lines 11-12 in algorithm~1 which involve minimization of 1D surrogate functions. Since it is not necessary to find the exact solutions for the surrogate problems, we use Newton's method with a single Newton step when solving for the mean and using thresholding to impose positivity, i.e.,
\begin{align}
&m^{(t+1)}_j=\big [m^{(t)}_j-\big (\partial_{x} \mathcal{S}(x)/\partial^2_{x} \mathcal{S}(x)\big )\big|_{x=m^{(t)}_j}\big ]_+ = \big [ m^{(t)}_j-\big (b_j^y-b_j^{(t)}+f_j^{(t)})/(Z_1b_j^{(t)}+2g_j^{(t)})  \big]_+, \label{newton_m}
\end{align}
where $\partial_x$ denotes the derivative with respect to $x$, $\mathcal{S}$ denotes the surrogate in \eqref{surogate1}, and all parameters are defined in \eqref{b}, \eqref{qr},\eqref{d_f},\eqref{g_h}. While Newton steps as in \eqref{newton_m} are not guaranteed to decrease the objective function, we have found that in all our numerical experiments the objective function decreased at each iteration. 
If the  objective function is not monotonically decreasing, then the step size can be adjusted or a trust-region approach can be used, as described next for the variance updates.  

When solving for the variance, we use a Newton trust region approach \cite{sorensen} where Newton steps for $v_j$ are considered (similarly to \eqref{newton_m}) but the step sizes are adjusted such that the solution will lie within a trusted region. At each iteration, the trust region is adjusted from both left and right by comparing predicted reduction based on a local quadratic approximation and the actual reduction of the 1D surrogate function. If the prediction is poor, the trust region is shrunk. The left boundary is limited to the positive axis. This approach also ensures the decrease of the objective function at each step. 

For MAP with the smooth edge-preserving neighborhood penalties we found that a trust region approach is also required to guarantee convergence in cases where the linear part of the penalty dominates. For reweighted $\ell_2$, taking a single Newton step each iteration was sufficient.

\section{Details about the SBL Algorithm} \label{apx_B}
We provide a short overview of the SBL algorithm in \cite{Danniel}. For further details please see \cite{tipping2001sparse,Danniel}.
A commonly used alternative to the Poisson noise model in \eqref{poisson} is the post-log Gaussian noise model given by 
\begin{align}
   \tilde{y}_i=\log(\eta_i/y_i), \qquad \qquad  p(\bd{\tilde{y}}|\bd{x})=\prod_{i=1}^n p(\tilde{y}_i|\bd{x})=\prod_{i=1}^n\cN(\tilde{y}_i|\bd{\phi}^T_i\bd{x},y^{-1}_i),   \label{Gauss}
\end{align}
where $\tilde{\bd{y}}$ denotes the post-log calibrated measurements. Note that the variance is inversely proportional to the data $y_i$.  
The Gaussian model in \eqref{Gauss} can approximate the Poisson model in \eqref{poisson} quite well for moderate to high photon flux.
As described in Sec.~\ref{ARD}, the solution to \eqref{evidence_max} under the Gaussian model in \eqref{Gauss} can be found using an expectation maximization (EM) algorithm \cite{tipping2001sparse}. In the expectation (E) step, one uses the conjugate gradient method \cite{Saad} to solve the following matrix equation for the posterior mean $\bd{m}$ \cite{Danniel}
\begin{align} 
\bd{P}\bd{m}=\bd{\Phi}^T\bd{B}\tilde{\bd{y}}, \label{eq_mean}
\end{align} 
where $\bd{P}$ is a $p\times p$ matrix, defined as
\begin{align}
\bd{P}=\bd{\Phi}^T\bd{B}\bd{\Phi}+\bd{\Psi}^T\bd{\Gamma}^{-1}\bd{\Psi}, \label{P_prob}
\end{align}
with $\bd{B}=\text{Diag}(\bd{y})$ and $\bd{\Gamma}=\text{Diag}(\bd{\gamma})$ ($\bd{\gamma}$ and $\bd{\Psi}$ are defined in \eqref{prior}). Another operation required for the E-step is the computation of the diagonal of the covariance matrix \cite{Danniel}
\begin{align}
\bd{\Sigma}=\bd{\Psi}^T\bd{P}^{-1}\bd{\Psi}, \label{Sigma_prob}
\end{align}
which direct computation has $O(p^3)$ complexity ($p$=number of pixels/voxels) due to the matrix inverse, which is not practical. To resolve this issue, Jeon et al. \cite{Danniel} used an estimator for the diagonal of the covariance matrix based on a method proposed by Tang and Saad \cite{Tang}. The general idea is to use the conjugate gradient (CG) method \cite{Saad} to solve the following matrix equations for $\bd{x}_k$
\begin{align}
\bd{P}\bd{x}_k=\bd{\Psi}^T \bd{q}_k, \qquad k=1,2,.....,K \label{probe}
\end{align}
where $\bd{P}$ is given in \eqref{P_prob} and $\bd{q}_k$ are user-specified ``probing vectors'', and then use all the solutions $\bd{x}_k$ $(k=1,2,...,K)$ to construct an estimator for the diagonal of $\bd{\Sigma}$ in \eqref{Sigma_prob} in a manner described in \cite{Danniel}.  
Since the CG method only requires the multiplication of the matrix $\bd{P}$ with some vector, the computational complexity is reduced from $O(p^3)$ to $O(n p^{1/D})$ ($D=2,3$ for 2D/3D images, respectively). 
Here we used the probing vectors that are based on the graph coloring scheme in \cite{Danniel}. Note that the structure and number of probing vectors required for some given accuracy depend on the properties of the matrix $\bd{P}$, so they are system and object dependent. 
In the M-step, $\bd{\gamma}$ is updated as
\begin{align}
\bd{\gamma}=[(\bd{\Psi}\bd{m})\odot (\bd{\Psi}\bd{m}) + \text{diag}(\bd{\Sigma})]^{-1},
\end{align}
where the inverse is performed elementwise. The E and M steps are repeated till convergence. 

We define the residuals for the matrix equations in \eqref{eq_mean} and \eqref{probe} as $R_m=\|\bd{P}\bd{m}-\bd{\Phi}^T\bd{B}\tilde{\bd{y}} \|_2$ and $R_v=\|\bd{P}\bd{x}_k-\bd{\Psi}^T \bd{q}_k\|_2$, respectively, for a given $\bd{m}$ and $\bd{x}_k$ ($k=1,2.,...,K$).
The CG iterations for \eqref{eq_mean} and \eqref{probe} are stopped  when $R_m\leq \varepsilon_m$ and $R_v\leq \varepsilon_v$, where $\varepsilon_m$ and $\varepsilon_v$ are thresholds for the residuals.
In the example of Sec.~\ref{sec_sim}, we performed several trials with different choices for $K$ and for the thresholds $\varepsilon_m$ and $\varepsilon_v$. Based on the outcomes of these trials, we chose $K=64$, $\varepsilon_m=10^{-8}$, and $\varepsilon_v=10^{-2}$. Note that decreasing $\varepsilon_v$ had negligible effect on the accuracy of the mean $\bd{m}$. Increasing $\varepsilon_m$ led to significantly larger error in $\bd{m}$ in later iterations.

%

\section{Details for the Experimental Lab Setup} \label{apx_D}
The real data has been acquired at the Duke Multi-Modality Imaging Lab (MMIL) \cite{crotty2007}. The x-ray source specifications are $60$kVp, $50$mA, and $25$ms.  A $0.55$mm Cerium filter was used to strongly filter the Bremsstrahlung (continuous spectrum x-rays) and create a pseudo monoenergetic source spectrum. The detector has a pixel pitch of $127 \mu m$  and a fill-factor of $0.75$ (pixel pitch divided by pixel size). 
By integrating the flux-energy curve computed using the XSPECT software \cite{dodge2008rapid} and accounting for detector fill-factor, we estimated a total of  $\eta=2\times 10^{3}$ photons per detector pixel. The view angle resolution is $1^\circ$.


\begin{thebibliography}{1}

\bibitem{slaney1988principles}
A. C. Kak and M. Slaney, \emph{Principles of Computerized Tomographic Imaging,} SIAM, 2001, \url{http://www.slaney.org/pct/}.


\bibitem{natterer1986computerized}
F.~Natterer, \emph{The Mathematics of Computerized Tomography}, John Wiley \&
  Sons, 1986.

\bibitem{kalender2006}
W. A. Kalender, \emph{X-ray computed tomography}, Phys. Med. Biol., 51(13), 2006, pp. R29-R43.


\bibitem{frank1992electron}
J. Frank, \emph{Electron tomography}, Springer, 1992.

\bibitem{li2002}
M.~H. Li, H.~Q. Yang, and H.~Kudo, \emph{An accurate iterative
  reconstruction algorithm for sparse objects: Application to 3d blood vessel
  reconstruction from a limited number of projections}, Phys. Med. Biol., 47, 2002, pp. 2599--2609.

    
\bibitem{sonka2000handbook}
J. A. Fessler, \emph{Handbook of Medical Imaging Vol. 2, Chapter 1: Statistical Image Reconstruction Methods for Transmission Tomography}, Edited by M.~Sonka and J.~M. Fitzpatrick, SPIE, 2000.
  
\bibitem{simoncelli1999modeling}
  E. P. Simoncelli, \emph{Modeling the joint statistics of images in the wavelet domain}, Proc. SPIE 1999, pp. 188--195.

\bibitem{boubchir2005}
  L. Boubchir and J. M. Fadili  , \emph{Multivariate statistical modeling of images with the curvelet transform}, Proc. 8th International Conference on Signal  Proc. and Its Applications, 2005,  pp. 747--750.

\bibitem{lange}
K. Lang and R. Carson, \emph{EM reconstruction algorithms for emission and transmission tomography}, J. Comput. Assist. Tomog., 8(2), 1984, pp. 306--316. 

\bibitem{holland1977robust}
P. W. Holland, and R. E. Welsch, \emph{Robust regression using iteratively reweighted least-squares}, Communications in Statistics-theory and Methods, 6(9), pp. 813--827, 1977.


\bibitem{lange1995globally}
K. Lange and J. A. Fessler, \emph{Globally convergent algorithms for maximum a posteriori transmission tomography},  
IEEE Trans. Image Proc.,  4(10), 1995, pp. 1430--1438. 

\bibitem{Erdogan}  
H. Erdogan, J.A. Fessler, \emph{Ordered subsets algorithms for transmission tomography}, Phys. Med. Biol., vol. 44, pp. 1999, pp. 2835--2851.
    
\bibitem{OSullivan}
  J. A. O'Sullivan and  J. Benac, \emph{Alternating minimization algorithms for transmission tomography}, IEEE Trans. Med. Imaging, 26(3), 2007, pp. 283-297.
  

    
\bibitem{neal1995bayesian}
  R. M. Neal, \emph{Bayesian Learning for Neural Networks}, Springer-Verlag, 1996.

\bibitem{tipping2001sparse}
  M. E. Tipping, \emph{Sparse Bayesian learning and the relevance vector machine},  J. Mach. Learn. Res., 1, 2001, pp. 211--244.


\bibitem{wipf2011latent}
D. P. Wipf, B. D. Rao, and S. Nagarajan, \emph{Latent variable Bayesian models for promoting sparsity}, IEEE Trans.
Inform. Theory, 57,  2011, pp. 6236--6255.

\bibitem{Danniel}
H. Jeon, Y. Kaganovsky, S. Han, and L. Carin, \emph{GPU-based sparse Bayesian learning for adaptive transmission tomography}, Nuclear Science Symposium and Medical Imaging Conference, IEEE, 2015.

\bibitem{Photon_Count}
J. S. Iwanczyk, E. Nyg{\aa}rd, O. Meirav, J. Arenson, W. C Barber, N. E. Hartsough, N. Malakhov, and J. C Wessel, \emph{Photon counting energy dispersive detector arrays for x-ray imaging}, Nuclear Science, IEEE Transactions on, 56(3), pp 535--542,2009

\bibitem{bouman1996}
 C. A. Bouman, K. Sauer, \emph{A unified approach to statistical tomography using coordinate descent optimization},
IEEE Trans. Image Proc., 5(3), 1996, pp. 480-492.

\bibitem{ramani2012splitting}
S. Ramani, and J. Fessler, \emph{A splitting-based iterative algorithm for accelerated statistical X-ray CT reconstruction},
Medical Imaging, IEEE Transactions on, 31(3), pp. 677--688, 2012.

\bibitem{sukovic2000penalized}
P. Sukovic, and N. H. Clinthorne, \emph{Penalized weighted least-squares image reconstruction for dual energy X-ray transmission tomography}, Medical Imaging, IEEE Transactions on, 19(11), pp. 1075--1081, 2000.

\bibitem{fessler1995hybrid}
J. Fessler, \emph{Hybrid Poisson/polynomial objective functions for tomographic image reconstruction from transmission scans}, Image Processing, IEEE Transactions on, 4(10), pp. 1439--1450, 1995.


\bibitem{zhuang1994numerical}
W. Zhuang, S. S. Gopal, and T. J Hebert, \emph{Numerical evaluation of methods for computing tomographic projections}, IEEE Trans. Nucl. Sci., 4, 1994, pp. 1660--1665. 



\bibitem{cizar}
I. Csisz\'ar and G. Tusn\'ady, \emph{Information geometry and alternating minimization procedures}, Statistics and Decisions, Supplementary Issue~1, 1984, pp. 205-237. 

\bibitem{sidky2}
E.~Y. Sidky, C.~M. Kao, and X.~H. Pan,  \emph{Accurate image reconstruction
  from few-views and limited-angle data in divergent-beam {CT}}, J. X-ray Sci. Tech., 14,  2006, pp. 119--139.

\bibitem{sidky}
E.~Y. Sidky and X.~C. Pan, \emph{Image reconstruction in circular cone-beam
  computed tomography by constrained  total-variation minimization}, Phys. Med. Biol., 53, 2008, pp. 4777--4807.
  
\bibitem{wipf2004perspectives}
 D. P.  Wipf, J. Palmer, and B. Rao, \emph{Perspectives on sparse Bayesian learning}, Proc. Neural Information Processing Systems, 16(1), 2004, pp. 249-256. 
 
     
\bibitem{Wipf2008}
D. P. Wipf and S. Nagarajan. \emph{A new view of automatic relevance determination}, Proc. Neural Information Processing Systems, 20, 2008.


\bibitem{dempster1977maximum}
A. P. Dempster, N. M. Laird, and D. B. Rubin, \emph{Maximum likelihood from incomplete data via the EM algorithm}, 
J. Roy. Stat. Soc., 39(1), 1977, pp. 1--38.


\bibitem{seeger2011large}
 M. Seeger and H. Nickisch, \emph{Large scale Bayesian inference and experimental design for sparse linear models},  SIAM J. Imaging Sci., 4(1), 2011, pp. 166--199.

\bibitem{seeger_mag}
M. Seeger and P. Wipf, \emph{Variational Bayesian inference techniques}, IEEE Mag. Signal Proc., 27(6), 2010, pp. 81-91. 




\bibitem{bekas2007estimator}
  C. Bekas, E. Kokiopoulou, and Y. Saad, \emph{An estimator for the diagonal of a matrix}, Applied Numerical Mathematics, 57(11) , 2007, pp. 1214--1229. 

\bibitem{Papandreou2011}
G. Papandreou and A. L. Yuille, \emph{Efficient variational inference in large-scale Bayesian compressed sensing},   Proc. IEEE Workshop on Information Theory in Computer Vision and Pattern Recognition, 2011.

\bibitem{Luenberger}
D. G. Luenberger, Y. Ye, \emph{Linear and Nonlinear Programming}, 3rd edition, 2008.  Springer,

%
%

\bibitem{neal1998view}
R. M. Neal and G. E. Hinton, \emph{A view of the EM algorithm that justifies incremental, sparse, and other variants}, Learning in Graphical Models, Springer, 1998, pp. 355--368.

\bibitem{gunawardana}
 A. Gunawardana and W. Byrne, \emph{Convergence theorems for generalized alternating minimization procedures},  J. Mach. Learn. Res.,6,  2005, pp. 2049--2073.
 


\bibitem{tzikas}
D. G. Tzikas, A. C. Likas, and N. P. Galatsanos, \emph{The variational approximation for Bayesian inference},  IEEE Mag. Signal Proc., 25(6), 2008, pp. 131--146. 

 
\bibitem{Palmer}
  J. Palmer, D. Wipf, K. Kreutz-Delgado, and B. Rao, \emph{Variational EM algorithms for non-Gaussian latent variable models}, Proc. Advances in Neural Information Processing systems, 18, 2006, page 1059.  

\bibitem{Jordan}
M. Jordan, Z. Ghahramani, T. Jaakkola, and L. Saul. \emph{Introduction to variational methods for graphical models}, Machine Learning, 37, 1999, pp.183-233. 

\bibitem{Wu}
C. F. Jeff Wu, \emph{On the convergence properties of the EM algorithm},  The Annals of Statistics, 11, 1983, pp. 95--103.


\bibitem{Boyd}
  S. Boyd and L. Vandenberghe, \emph{Convex Optimization}, Cambridge University Press, 2004.

\bibitem{Bishop}
C. M. Bishop, \emph{Pattern Recognition and Machine Learning}, 2006, Springer New York. 
  




                              
 
\bibitem{Zangwill}
W. I. Zangwill, \emph{Nonlinear Programming: A Unified Approach},  Prentice-Hall, 1969.


\bibitem{chartrand2008iteratively}
R. Chartrand and W. Yin, \emph{Iteratively reweighted algorithms for compressive sensing}, Proc. IEEE international conference on Acoustics, speech and signal processing,  2008, pp. 3869--3872. 

\bibitem{wipf2010iterative}
D. Wipf and S. Nagarajan, \emph{Iterative reweighted $\ell_1$ and $\ell_2$ methods for finding sparse solutions}, IEEE Journal of Selected Topics in Signal Proc., 4(2), 2010,  pp.~317--329. 


\bibitem{cover}
T. M. Cover and J. A. Thomas, \emph{Elements of Information Theory}, 2012, John Wiley \& Sons.


\bibitem{Keesing}
D. Keesing, \emph{Development and Implementation of Fully 3D Statistical Image Reconstruction Algorithms for
Helical CT and Half- Ring PET Insert System}, PhD thesis, Washington University in St. Louis, 2009.

\bibitem{Soysal}
S. Degirmenci, D. G. Politte, C. Bosch, N. Tricha and J. A. O'Sullivan, \emph{Acceleration of iterative image reconstruction for x-ray imaging for security applications}, Proc. SPIE 9401, Computational Imaging XIII, 94010C, March 2015, doi:10.1117/12.2082966.

 
\bibitem{challis2011concave}
  E. Challis and D. Barber, \emph{Concave Gaussian variational approximations for inference in large-scale Bayesian linear models}, Proc. International Conference on Artificial Intelligence and Statistics,  6, 2011, page 7.

\bibitem{Khan}
M. E. Khan,  A. Aravkin, M.  Friedlander, and M. Seeger, \emph{Fast dual variational inference for non-conjugate latent Gaussian models}, Proc. International Conference on Machine Learning, 2013, pp. 951--959.   

\bibitem{Banerjee}
S. Banerjee, B. P. Carlin, and A. E. Gelfand, \emph{Hierarchical Modeling and Analysis for Spatial Data}, CRC Press, 2014. 

\bibitem{Whiting1}
B. R. Whiting, \emph{Signal Statistics of X-ray Computed Tomography}, Proc. SPIE Conference on Medical Imaging 2002: Physics of Medical Imaging, 4682 (L. Antonuk and M. Yaffe, Eds.), San Diego, CA,
Feb. 2002.

\bibitem{Whiting2}
B. R. Whiting, P. Massoumzadeh, O. A. Earl, J. A. O’Sullivan,D. L. Snyder, and J. F. Williamson, \emph{X-ray Computed Tomography Signal Properties},  Med. Phys. 33, 3290, 2006.

\bibitem{nuyts}
J. Nuyts, B. De Man, J. A. Fessler, W. Zbijewski, and F. J. Beekman,  \emph{Modelling the physics in the iterative reconstruction for transmission computed tomography},
Phys. Med. Biol., 58(12), 2013, pp. R63--R96.

\bibitem{Basu}
S. Basu, and B. De Man, \emph{Branchless distance-driven projection and backprojection}, Proc. SPIE Electronic Imaging, 6065, 2006.

\bibitem{DeMan}
B.  De Man, and S. Basu, \emph{Distance-driven projection and backprojection in three dimensions}, Phys. Med. Biol. 49, 2004, pp. 2463--2475.

\bibitem{Yu}
Z. Yu, J. B. Thibault, C.A. Bouman, K.D. Sauer, J. Hsieh, \emph{Fast Model-Based X-Ray CT Reconstruction Using Spatially Nonhomogeneous ICD Optimization, 
IEEE Trans. Image Process., 20(1), 2011, pp. 161--175.}

\bibitem{website}
\url{http://www.yan-kaganovsky.com/#!code/c24bp} 

\bibitem{Saad}
Y. Saad, \emph{Iterative methods for sparse linear systems}, SIAM, 2003.

\bibitem{Tang}
T. Tang, and Y. Saad, \emph{A Probing method for computing the diagonal of a matrix inverse}, Numerical Linear Algebra with Applications, 19(3), pp. 485--501, 2012. 

\bibitem{sorensen}
D. C. Sorensen, \emph{Newton's method with a model trust region modification},
SIAM J. Numerical Analysis, 19 (2), 1982, pp. 409--426.

\bibitem{crotty2007}
D.~J. Crotty, R.~L. McKinley, and M.~P. Tornai, \emph{Experimental spectral
  measurements of heavy k-edge filtered beams for x-ray computed
  mammotomography}, Phys. Med. Biol. 52, 2007, pp.  603.

\bibitem{dodge2008rapid}
C.~W. Dodge, \emph{A rapid method for the simulation of filtered X-ray spectra
  in diagnostic imaging systems} (ProQuest, 2008).

\end{thebibliography}
\end{document}